\documentclass[12pt]{amsart}

\usepackage{amsmath, amssymb, amsthm}
\usepackage{geometry}
\usepackage[utf8]{inputenc}

\usepackage{longtable}
\usepackage{hyperref}

\usepackage{listings}
\usepackage{xcolor}
\usepackage{pdflscape}     
\usepackage{fancyvrb}     
\usepackage{algorithm}
 \usepackage{algpseudocode}
\algrenewcommand\alglinenumber[1]{\scriptsize$\triangleright$}

\definecolor{codegreen}{rgb}{0,0.6,0}
\definecolor{codegray}{rgb}{0.5,0.5,0.5}
\definecolor{codepurple}{rgb}{0.58,0,0.82}
\definecolor{backcolour}{rgb}{0.95,0.95,0.92}

\lstdefinestyle{SAGEMATHstyle}{
    backgroundcolor=\color{backcolour},   
    commentstyle=\color{codegreen},
    keywordstyle=\color{blue},
    numberstyle=\tiny\color{codegray},
    stringstyle=\color{codepurple},
    basicstyle=\ttfamily\footnotesize,
    breakatwhitespace=false,         
    breaklines=true,                 
    captionpos=b,                    
    keepspaces=true,                 
    numbers=none,                     
    numbersep=5pt,                  
    showspaces=false,                
    showstringspaces=false,
    showtabs=false,                  
    tabsize=2
}
\newtheorem{theorem}{Theorem}[section]
\newtheorem{lemma}[theorem]{Lemma}
\newtheorem{proposition}[theorem]{Proposition}
\newtheorem{corollary}[theorem]{Corollary}
\theoremstyle{definition}
\newtheorem{definition}[theorem]{Definition}
\newtheorem{example}[theorem]{Example}
\newtheorem{note}[theorem]{Note}
\theoremstyle{remark}
\newtheorem{remark}[theorem]{Remark}

\newtheorem*{acknowledgment}{Acknowledgment}

\newcommand{\Z}{\mathbb{Z}}

\newcommand{\A}{\mathcal{A}}

\newcommand{\Ext}{\text{Ext}}

\newcommand{\Sq}{\text{Sq}}

\newcommand{\PK}{\mathcal{P}_k}
\newcommand{\PKd}{(\mathcal{P}_k)_d}
\newcommand{\PKzero}{\mathcal{P}_k^0}
\newcommand{\PKplus}{\mathcal{P}_k^+}
\newcommand{\QPK}{Q\mathcal{P}_k}
\newcommand{\QPKd}{(Q\mathcal{P}_k)_d}
\newcommand{\QPKzero}{Q\mathcal{P}_k^0}
\newcommand{\QPKplus}{Q\mathcal{P}_k^+}
\newcommand{\Aplus}{\overline{\mathcal{A}}}
\newcommand{\Ztwo}{\mathbb{Z}/2}

\newcommand{\QPKomega}{Q\mathcal{P}_k(\omega)}
\newcommand{\GLK}{G_k}
\newcommand{\SigmaK}{\Sigma_k}

\usepackage{enumitem}

\usepackage{framed}

\definecolor{stepcolor}{RGB}{0,102,204}
\definecolor{outputcolor}{RGB}{153,0,153}
\definecolor{theoremcolor}{RGB}{0,128,0}

\newenvironment{stepbox}[1]{
\begin{framed}
\textbf{\color{stepcolor}#1}
\vspace{0.2cm}
}{\end{framed}}

\newenvironment{outputbox}{
\begin{framed}
\textbf{\color{outputcolor}Computational Output:}
\vspace{0.2cm}
}{\end{framed}}

\title[Computational Approaches to the Singer Transfer]{Computational Approaches to the Singer Transfer: Preimages in the Lambda Algebra\\ and $G_k$-Invariant Theory}

\author{Phuc Vo Dang$^{*}$}
\address{Department of AI, FPT University, Quy Nhon AI Campus\\
An Phu Thinh New Urban Area, Quy Nhon City, Binh Dinh, Vietnam}
\email{dangphuc150488@gmail.com}

\begin{document}

\maketitle

\begin{abstract}
This work revisits the algebraic transfer, a crucial homomorphism in algebraic topology introduced by W.M. Singer. We focus on its representation within the framework of the lambda algebra, particularly leveraging the explicit, invariant-theoretic formula developed by P.H. Chon and L.M. Ha \cite{ChonHaCRASS2011}. From this foundation, we derive both recursive and explicit formulas for the transfer in low ranks. The central contribution of this work is the formulation of a systematic, algorithmic method for computing the preimage of elements under this transfer. This framework is essential for addressing the fundamental question of whether specific cohomology classes in the Adams spectral sequence, $\Ext_{\A}^{s,t}(\Z/{2},\Z/{2})$, lie within the image of the Singer transfer. We formalize this preimage search as a solvable problem in linear algebra and  illustrate its application to important, well-known cohomology classes. As a consequence, we show that the proof in Nguyen Sum's paper \cite{Sum}, which asserts that the indecomposable element $d_0 \in \Ext_{\A}^{4,18}(\mathbb{Z}/2, \mathbb{Z}/2)$ lies in the image of the fourth Singer algebraic transfer, is false. 

Additionally, we explicitly describe the structure of a preimage under the Singer transfer of the indecomposable element $p_0 \in \Ext_{\A}^{4,37}(\mathbb{Z}/2, \mathbb{Z}/2)$. This preimage had not been explicitly determined in the previous work of N.H.V. Hung and V.T.N. Quynh \cite{HQ}. The present work can be seen as a continuation of our recent paper~\cite{Phuc}, as part of an ongoing research project on the Peterson hit problem and its applications via an algorithmic approach.

We have also constructed an algorithm in \textsc{SageMath} that produces explicit output for the dimension and a basis of the $G_k$-invariant space \([(Q\mathcal{P}_k)_d]^{G_k}\), for arbitrary values of \(k\) and \(d\), provided they lie within the memory limits of the executing machine. This algorithm is used to verify the manually computed results in our previous works~\cite{Phuc1, Phuc2, Phuc3}, which addressed the solution to Singer's conjecture on the injectivity of the algebraic transfer for rank 4.

\end{abstract}

\bigskip


\medskip


\section{Introduction}

The stable homotopy groups of spheres, $\pi_*^S(\mathbb S^0)$, represent one of the most central and challenging objects of study in algebraic topology. The primary computational tool for approaching these groups is the Adams spectral sequence, whose $E_2$-term is given by the cohomology of the mod-2 Steenrod algebra $\A$, denoted $\Ext_{\A}^{k,t}(\Z/{2},\Z/{2})$ \cite{Adams1958}. Consequently, a deep understanding of the structure of $\Ext_{\A}^{*,*}(\Z/{2},\Z/{2})$ is of fundamental importance.

In a series of foundational papers, W. M. Singer introduced the algebraic transfer, an algebraic homomorphism that provides a powerful tool for constructing and analyzing elements in $\Ext_{\A}^{*,*}(\Z/{2},\Z/{2})$ \cite{Singer1989}. The Singer transfer, $\varphi_k,$ is a homomorphism from the space of coinvariants of $\A$-primitive elements in the homology of the classifying space of the elementary abelian 2-group $(\Z/2)^k$, to the Ext groups \cite{ChonHaTA2014}:
\[
\varphi_k: [P_{\A}H_*(B(\Z/2)^k)]_{G_k} \rightarrow \Ext_{\A}^{k,k+*}(\Z/2, \Z/2).
\]
Here, $G_k$ is the general linear group $G_k(\Z/2)$, and $P_{\A}H_*(B(\Z/2)^k)$ is the subspace of elements annihilated by positive degree Steenrod squares. The Singer transfer has been studied by many authors from various perspectives (see e.g., Boardman \cite{Boardman},  Bruner, Ha and Hung \cite{BHH}, Chon and Ha \cite{ChonHaCRASS2011, ChonHaTA2014}, Ha \cite{Ha}, Hung and Quynh \cite{HQ}, Hung and Powell \cite{HP}, the present author \cite{Phuc1, Phuc2, Phuc3}, Sum \cite{Sum}, Tin \cite{Tin}, Walker and Wood \cite{Walker-Wood2}, etc). This construction connects the intricate world of $\Ext_{\A}$ to the``\textit{hit problem}'' of F. P. Peterson and the modular representation theory of general linear groups, suggesting that tools from the latter can illuminate the structure of the former \cite{ChonHaTA2014}.

The study of the transfer is often more tractable in a different algebraic setting: the Lambda algebra, $\Lambda$, introduced by Bousfield, Curtis, Kan, et al. \cite{Bousfield1966}. The lambda algebra provides an alternative chain complex for computing $\Ext_{\A}$ \cite{ChonHaCRASS2011}. Singer's transfer can be realized at this chain level as a map $\varphi_k: H_*(B(\Z/2)^k) \to \Lambda_k$, where $\Lambda_k$ is the part of the lambda algebra of length $k.$ Nevertheless, Singer's original construction was not easily applicable for direct computations. A significant advancement was made by Chon and Ha \cite{ChonHaCRASS2011}, who provided a direct, invariant-theoretic description of the map $\varphi_k.$ Their formula gives a concrete, computational bridge between the homology of $B(\Z/2)^k$ and the lambda algebra.

The present work focuses on a critical question that arises from their work: \textbf{the preimage problem}. Given an element $y \in \Lambda_s$ that represents a non-zero class in $\Ext_{\A}$, can we find a corresponding element $x \in H_*(B(\Z/2)^k)$ such that $[\varphi_k(x)] = [y]$?

The importance of this question is paramount. If such a solution $x$ exists and, crucially, is $\A$-annihilated (i.e., it is a primitive element, $x \in P_{\A}H_*(B(\Z/2)^k)$), it proves that the corresponding cohomology class is in the image of the Singer transfer. Conversely, if no such $\A$-annihilated preimage exists, the class is not detected by the transfer. Determining which elements are in the image is a major goal in the field, as it sheds light on the origins and structure of elements in the Adams spectral sequence.

\subsection*{Organization of the paper}
This paper is structured as follows. In Section \ref{s2}, we review the necessary background on the Lambda algebra and the Singer transfer, establishing the recursive formula essential for our computations. Section \ref{s3} presents the main theoretical framework of this work. We formulate the preimage problem as a solvable system of linear equations and apply it to analyze the  known cohomology classes $c_0$, $d_0$, and $p_0$ in $\Ext_{\mathcal A}.$ As a consequence, the proof of the claim that $d_0 \in \operatorname{Im}(\varphi_4)$ given in Nguyen Sum's paper~\cite{Sum} is shown to be incorrect. Additionally, we explicitly describe the structure of a preimage under the Singer transfer of the indecomposable element $p_0.$ This preimage had not been explicitly determined in the previous work of Hung and Quynh \cite{HQ}. In Section \ref{s4}, we connect our work to the Singer conjecture by describing a comprehensive algorithmic framework for computing the associated spaces of $G_k$-invariants. Finally, a detailed appendix provides the full computational workflow of our algorithm described in Section~\ref{s4}, including a complete sample run for the case $k = 4$, $d = 33$, to ensure that all results are fully verifiable and reproducible.

\medskip

\begin{note}

Additionally, in this paper, we make publicly available our complete \textsc{SageMath} code (see Appendix~\ref{pl}). The code has been fully optimized to the best of our ability, spans 30 pages, and serves two main purposes:

\begin{itemize}
  \item[(1)] To fully automate the explicit computation of the domain of the Singer transfer and its dual, a process that has long been handled manually and is therefore highly prone to human error. While our algorithm is functional and effective, we acknowledge that further improvements are possible in terms of computational efficiency and optimization.
  
  \item[(2)] To provide a tool for verifying the computational results in our previous works~\cite{Phuc1, Phuc2, Phuc3}, which are related to the solution of Singer's conjecture in the rank 4 case.
\end{itemize}

\end{note}


\section{The Singer algebraic transfer and the Lambda algebra}\label{s2}

Let us recall that the Lambda algebra $\Lambda$ is a bigraded differential algebra over the field $\mathbb Z/2$ generated by the monomials $\lambda_t,\, t\geq 0,$ with the Adem relations:
\begin{equation}\label{ct1}
\lambda_k \lambda_{\ell} = \sum_{t \ge 0} \binom{t-\ell-1}{2t-k} \lambda_{k+\ell-t} \lambda_{\ell} \quad \text{for any } k, \ell \ge 0,
\end{equation}
and the differential
\begin{equation}\label{ct2}
 \delta(\lambda_k) = \sum_{t\geq 0}\binom{k-1-t}{t+1}\lambda_{k-1-t}\lambda_t.
\end{equation}

Denote by $\Lambda_k$ denotes the vector subspace of $\Lambda$ spanned by monomials of length $k$ in $\lambda_t$. This space has a basis consisting of all \textbf{admissible monomials} of the form $\lambda_I = \lambda_{i_1} \lambda_{i_2} \cdots \lambda_{i_k}$, where $i_j \leq 2i_{j+1}$ for all $1 \leq j \leq k - 1.$ 

\begin{definition}[\cite{ChonHaCRASS2011}.]
The \textbf{Singer algebraic transfer} is a homomorphism $\varphi_k: H_*(B(\Z/2)^k) \to \Lambda_k$, where $H_*(B(\Z/2)^k) = \Gamma[a_k, \dots, a_1]$ is the divided power algebra on $k$ generators. 
\end{definition}

Chon and Ha \cite{ChonHaCRASS2011} provided a powerful, invariant-theoretic formula for this map.

\begin{theorem}[\cite{ChonHaCRASS2011}]\label{dlCH}
\label{thm:chonha}
The representation $\varphi_{k}$ for the Singer transfer is given in terms of a generating function correspondence: $\varphi_{k}: a[x_{k},x_{k-1},\ldots ,x_{1}] \longrightarrow \lambda[v_{1},v_{2}, \ldots, v_{k}].$ That is, the transfer $\varphi_{k}$ sends an element $z=a^{(I)}\in H_{*}(B(\Z/2)^{k})$ to the sum of all $\lambda_{J}\in\Lambda_{k}$ such that $x^{I}$ is a non-trivial summand in the expansion of $v^{J}$ in the variables $x_{1}, \ldots,x_{k}.$ This can be written as: $\varphi_s(z) = \sum_{J} \langle z, v^J \rangle \lambda_J,$ where $\langle \cdot, \cdot \rangle$ denotes the coefficient pairing.
\end{theorem}

The transfer map $\varphi_k$ in Theorem \ref{dlCH} can be defined recursively. This construction is based on splicing together connecting homomorphisms, a process described in \cite{ChonHaCRASS2011}.

\begin{corollary}\label{mdp}
The transfer map $\varphi_k: H_*(B(\Z/2)^k) \to \Lambda_k$ satisfies the following recursive formula for a basis element $a_k^{(t_k)}a_{k-1}^{(t_{k-1})}\dots a_1^{(t_1)}$: 
$$\varphi_k\bigg(a_k^{(t_k)}a_{k-1}^{(t_{k-1})}\dots a_1^{(t_1)}\bigg) = \sum_{i\ge t_1} \lambda_{i}\cdot \varphi_{k-1}\bigg(\big(a_{k}^{(t_k)}\dots a_2^{(t_2)}\big)\cdot \Sq_*^{i - t_1}\bigg).$$
Here $\Sq_*^{j}$ denotes the dual Steenrod operation.
\end{corollary}

Note that the right action of $\Sq_*^j$ on a single divided power generator is given by the formula: $(a^{(t)})\Sq_*^j = \binom{t-j}{j} a^{(t)}.$ Applying the Cartan formula repeatedly gives:
\begin{equation}\label{ct3}
(a_k^{(i_k)} \dots a_1^{(i_1)})\Sq_*^j = \sum_{j_k + \dots + j_1 = j} (a_k^{(i_k)})\Sq_*^{j_k} \dots (a_1^{(i_1)})\Sq_*^{j_1} = 
\sum_{\substack{j_1 + \cdots + j_k = j \\ j_t \leq i_t}}
\prod_{1\leq t\leq k}
\binom{i_t - j_t}{j_t}
\, a_t^{(i_t - j_t)}.
\end{equation}

Based on the proposition, we provide explicit descriptions of the results for $3 \leq k \leq 5.$

\subsubsection*{The case k=3}

\begin{align*}
&\varphi_3\left(a_3^{(t_3)} a_2^{(t_2)} a_1^{(t_1)}\right)
= \sum_{i_1 \ge t_1}
\lambda_{i_1} \,
\varphi_2\left(\mathrm{Sq}_*^{i_1 - t_1}\left(a_2^{(t_3)} a_1^{(t_2)}\right)\right) \\[1.5ex]
&= \sum_{i_1 \ge t_1}
\sum_{\substack{u_1 + u_2 = i_1 - t_1 \\ u_1, u_2 \ge 0}}
\lambda_{i_1} \,
\varphi_2\left(
\binom{-u_1 + t_3}{u_1}
\binom{-u_2 + t_2}{u_2} \,
a_2^{(-u_1 + t_3)} a_1^{(-u_2 + t_2)}
\right) \\[1.5ex]
&= \sum_{i_1 \ge t_1}
\sum_{\substack{u_1 + u_2 = i_1 - t_1 \\ u_1, u_2 \ge 0}}
\lambda_{i_1} \,
\binom{-u_1 + t_3}{u_1}
\binom{-u_2 + t_2}{u_2} \,
\varphi_2\left(a_2^{(-u_1 + t_3)} a_1^{(-u_2 + t_2)}\right) \\[1.5ex]
&= \sum_{i_1 \ge t_1}
\sum_{\substack{u_1 + u_2 = i_1 - t_1 \\ u_1, u_2 \ge 0}}
\sum_{i_2 \ge -u_2 + t_2}
\lambda_{i_1} \,
\binom{-u_1 + t_3}{u_1}
\binom{-u_2 + t_2}{u_2} \,
\lambda_{i_2} \,
\varphi_1\left(\mathrm{Sq}_*^{i_2 + u_2 - t_2}\left(a_1^{(-u_1 + t_3)}\right)\right) \\[1.5ex]
&= \sum_{i_1 \ge t_1}
\sum_{\substack{u_1 + u_2 = i_1 - t_1 \\ u_1, u_2 \ge 0}}
\sum_{i_2 \ge -u_2 + t_2}
\binom{t_3-u_1}{u_1}
\binom{t_2-u_2}{u_2} \,
\binom{t_1+t_2 + t_3 -  i_1 - i_2}{i_2 + u_2 - t_2} \,
\lambda_{i_1} \,
\lambda_{i_2} \,
\lambda_{t_1+t_2 + t_3 - i_1 - i_2}.
\end{align*}

\newpage
By completely analogous computations, we obtain the formulas for the cases $k = 4$ and $k = 5$ below.

\subsubsection*{The case k=4}

\begin{align*}
\varphi_4\left(a_4^{(t_4)} a_3^{(t_3)} a_2^{(t_2)} a_1^{(t_1)}\right)
&= \sum_{i_1 \ge t_1}
\sum_{\substack{k_1 + k_2 + k_3 = i_1 - t_1 \\ k_j \ge 0}}
\sum_{i_2 \ge -k_3 + t_2}
\sum_{\substack{u_1 + u_2 = i_2 + k_3 - t_2 \\ u_j \ge 0}}
\sum_{i_3 \ge -k_2 - u_2 + t_3} \\
&\quad 
\binom{t_4-k_1}{k_1}
\binom{t_3-k_2 }{k_2}
\binom{t_2-k_3 }{k_3}
\binom{t_4-k_1 - u_1}{u_1}
\binom{t_3-k_2 - u_2}{u_2} \\
&\quad 
\binom{t_1 + t_2 + t_3 + t_4 - i_1 - i_2 - i_3}{i_3 + k_2 + u_2 - t_3}
\lambda_{i_1}
\lambda_{i_2}
\lambda_{i_3}
\lambda_{t_1 + t_2 + t_3 + t_4 - i_1 - i_2 - i_3}.
\end{align*}

\medskip

\subsubsection*{The case k=5}
\begin{align*}
\varphi_{5}\!\left(
  a_{5}^{(t_{5})} a_{4}^{(t_{4})} a_{3}^{(t_{3})} a_{2}^{(t_{2})} a_{1}^{(t_{1})}
\right)
&=\!
\sum_{i_{1}\ge t_{1}}
\sum_{\substack{
      w_{1}+w_{2}+w_{3}+w_{4}=i_{1}-t_{1}\\[2pt]
      w_{j}\ge0
}}
\sum_{i_{2}\ge -w_{4}+t_{2}}
\sum_{\substack{
      k_{1}+k_{2}+k_{3}=i_{2}+w_{4}-t_{2}\\[2pt]
      k_{j}\ge0
}}
\sum_{i_{3}\ge -k_{3}-w_{3}+t_{3}}\\
&\sum_{\substack{
      u_{1}+u_{2}=i_{3}+k_{3}+w_{3}-t_{3}\\[2pt]
      u_{j}\ge0
}}
\sum_{i_{4}\ge -k_{2}-u_{2}-w_{2}+t_{4}}
\binom{t_5-w_{1}}{w_{1}}
\binom{t_4-w_{2}}{w_{2}}
\binom{t_3-w_{3}}{w_{3}}\\
&\binom{t_2-w_{4}}{w_{4}}
\binom{t_5-k_{1}-w_{1}}{k_{1}}
\binom{t_4-k_{2}-w_{2}}{k_{2}}
\binom{t_3-k_{3}-w_{3}}{k_{3}}\\
&\binom{t_5-k_{1}-u_{1}-w_{1}}{u_{1}}
\binom{t_4-k_{2}-u_{2}-w_{2}}{u_{2}}\\
&\binom{
  t_{1}+t_{2} + t_3 + t_4 + t_5 - i_1 - i_2 - i_3 - i_4
}{
  i_{4}+k_{2}+u_{2}+w_{2}-t_{4}
}
\lambda_{i_{1}}
\lambda_{i_{2}}
\lambda_{i_{3}}
\lambda_{i_{4}}
\lambda_{
    t_{1}+t_{2} + t_3 + t_4 + t_5 - i_1 - i_2 - i_3 - i_4
}.
\end{align*}

\section{An algorithmic approach to the Preimage problem}\label{s3}

The central goal is to determine if a given cohomology class in $\Ext_{\A}$ lies in the image of the Singer transfer. This can be rephrased at the chain level and solved algorithmically.

\subsection{The Preimage problem as a linear system}

\begin{lemma}
\label{thm:algorithm}
Let $y \in \Lambda_k$ be a cocycle (i.e., $\delta(y) = 0$) of degree $n$ representing a class $[y] \in \Ext_{\mathcal{A}}^{k,n}(\Z/2, \Z/2)$. The class $[y]$ is in the image of the Singer transfer if and only if there exists a solution $(x, z)$ to the linear equation
\[
\varphi_k(x) + \delta(z) = y
\]
where $z \in \Lambda_{k-1}$ and $x \in H_*(B(\mathbb{Z}/2)^k)$ is an $\mathcal{A}$-annihilated element (i.e., $x \in P_{\mathcal{A}}H_*(B(\mathbb{Z}/2)^k)$ satisfies $(x)Sq_*^{2^t} = 0$ for all $t \geq 0$).
\end{lemma}

\begin{proof}
The proof establishes the equivalence in both directions.

\textbf{($\Rightarrow$) Necessity:} 
Suppose $[y]$ is in the image of the Singer transfer. By definition, there exists a primitive element $x \in P_{\mathcal{A}}H_*(B(\mathbb{Z}/2)^k)$ such that the transfer map sends $x$ to a representative of $[y]$. That is, $\varphi_k(x)$ and $y$ represent the same cohomology class in $\Ext_{\mathcal{A}}^{k,*}(\Z/2, \Z/2)$.

Two cocycles $\varphi_k(x)$ and $y$ represent the same class in cohomology if and only if they are homologous in the lambda chain complex, meaning their difference is a boundary. Specifically, there must exist an element $z \in \Lambda_{k-1}$ such that 
\[
\varphi_k(x) - y = \delta(z).
\]
Since we work over the field $\Z/2,$ this is equivalent to 
\[
\varphi_k(x) + \delta(z) = y.
\]

\textbf{($\Leftarrow$) Sufficiency:}
Conversely, suppose there exists a solution $(x, z)$ with $x \in P_{\mathcal{A}}H_*(B(\mathbb{Z}/2)^k)$ and $z \in \Lambda_{k-1}$ satisfying $\varphi_k(x) + \delta(z) = y$. Then $\varphi_k(x)$ and $y$ differ by the boundary $\delta(z)$, so they represent the same cohomology class. Since $x$ is $\mathcal{A}$-annihilated, it lies in the domain of the Singer transfer, and thus $[y] = [\varphi_k(x)]$ is in the image of the transfer.

\textbf{Computational Formulation:}
The equation $\varphi_k(x) + \delta(z) = y$ is linear in the coefficients of $x$ and $z$. Let $x = \sum_{I \in \mathcal{B}_H} x_I a^{(I)}$ and $z = \sum_{J \in \mathcal{B}_{\Lambda'}} z_J \lambda_J$ be general elements, where:
\begin{itemize}
    \item $\mathcal{B}_H = \{a^{(I)} : I = (i_1, \ldots, i_k), \sum_{j=1}^k (2^j - 1)i_j = n\}$ is a basis for the degree-$n$ component of $H_*(B(\mathbb{Z}/2)^k)$,
    \item $\mathcal{B}_{\Lambda'} = \{\lambda_J : J \text{ is a positive composition of degree } n+1 \text{ into } k-1 \text{ parts}\}$ is a basis for the relevant component of $\Lambda_{k-1}$.
\end{itemize}

By linearity of $\varphi_k$ and $\delta$, the equation becomes:
\[
\sum_{I \in \mathcal{B}_H} x_I \varphi_k(a^{(I)}) + \sum_{J \in \mathcal{B}_{\Lambda'}} z_J \delta(\lambda_J) = y.
\]

This can be formulated as a system of linear equations $M\mathbf{v} = \mathbf{c}$ over $\Z/2$, where:
\begin{itemize}
    \item The vector of unknowns is $\mathbf{v} = (x_{I_1}, \ldots, x_{I_s}, z_{J_1}, \ldots, z_{J_t})^T$,
    \item The matrix $M$ has columns corresponding to $\varphi_k(a^{(I)})$ for $I \in \mathcal{B}_H$ and $\delta(\lambda_J)$ for $J \in \mathcal{B}_{\Lambda'}$,
    \item The vector $\mathbf{c}$ contains the coefficients of $y$ in the standard lambda basis.
\end{itemize}

A solution $(x, z)$ exists if and only if this system is consistent. However, not every solution yields an element in the transfer's image---one must additionally verify that the resulting polynomial $x$ is $\mathcal{A}$-annihilated, i.e., $(x)Sq_*^{2^t} = 0$ for all relevant $t$.

\textbf{Degree Considerations:}
For the equation to be meaningful, we require degree compatibility: if $\deg(y) = n$, then $\deg(x)$ must satisfy the relation imposed by the transfer map, and $\deg(z) = n+1$ to ensure $\deg(\delta(z)) = n$.
\end{proof}

\begin{remark}
The computational challenge lies in two aspects: (1) solving the potentially large linear system $M\mathbf{v} = \mathbf{c}$, and (2) checking the $\mathcal{A}$-annihilation condition, which requires verifying that $(x)Sq_*^{2^t} = 0$ for $t = 0, 1, \ldots, \lfloor \log_2(\deg(x)) \rfloor$. The algorithm's efficiency depends critically on the dimensions of the bases $\mathcal{B}_H$ and $\mathcal{B}_{\Lambda'}$, which grow rapidly with the degree and the parameter $k$.
\end{remark}

Lemma \ref{thm:algorithm} leads directly to the following computational algorithm.

\medskip

\noindent\rule{\textwidth}{0.4pt}
\noindent \textbf{Algorithm 0: Finding Preimages under the Singer transfer representation in $\Lambda$}
\label{alg:preimage}
\noindent\rule{\textwidth}{0.4pt}

\begin{algorithmic}[1]
\Function{FindAnnihilatedPreimageIterative}{$y_{\text{target}}, k, \text{deg}_x, \text{deg}_z, \text{max\_z\_terms}$}
    \State $b_{\text{adm}} \leftarrow \text{AdmissibleForm}(y_{\text{target}})$
    \State $\mathcal{B}_{H} \leftarrow \text{GenerateBasis}(\text{compositions}(\text{deg}_x, k))$
    \State $\mathcal{B}_{\Lambda'} \leftarrow \text{GenerateBasis}(\text{positive\_compositions}(\text{deg}_z, k-1))$
    \State $\text{AnnihilatedSolutions} \leftarrow []$
    \State $\text{total\_candidates\_checked} \leftarrow 0$
    
    \For{$n = 0$ \textbf{to} $\text{max\_z\_terms}$}
        \State \Comment{Search for solutions where $z$ has exactly $n$ terms}
        \State $\text{Print}(\text{``Searching z with''} + n + \text{" terms..."})$
        
        \If{$n=0$}
            \State $\text{combinations} \leftarrow [()]$ \Comment{Special case for $z = 0$}
        \Else
            \State $\text{combinations} \leftarrow \text{Combinations}(\mathcal{B}_{\Lambda'}, n)$
        \EndIf
        
        \State $\text{Print}(\text{``Number of z combinations:''} + |\text{combinations}|)$
        
        \For{each basis combination $C_z$ in combinations}
             \If{$n = 0$}
                 \State $z_{\text{cand}} \leftarrow 0$
             \Else
                 \State $z_{\text{cand}} \leftarrow \sum_{f \in C_z} f$
             \EndIf
             
             \State \Comment{--- Compute differential and target for phi equation ---}
             \State $\delta_z \leftarrow \delta(z_{\text{cand}})$
             \State $b' \leftarrow b_{\text{adm}} + \delta_z$ \Comment{Note: addition in $\Z/2$}
             
             \State \Comment{--- Solve the sub-problem $\varphi_k(x) = b'$ ---}
             \State $\text{solution\_space} \leftarrow \text{SolvePhiSystem}(b', k, \text{deg}_x)$
             
             \If{solution\_space $\neq$ null}
                \State $x_p \leftarrow \text{solution\_space.particular}$
                \State $\text{kernel\_basis} \leftarrow \text{solution\_space.kernel}$
                \State $\text{kernel\_dim} \leftarrow |\text{kernel\_basis}|$
                
                \State \Comment{--- Check all kernel combinations for $\A$-annihilation ---}
                \For{$i = 0$ \textbf{to} $2^{\text{kernel\_dim}} - 1$}
                    \State $x_{\text{candidate}} \leftarrow x_p$
                    \For{$j = 0$ \textbf{to} $\text{kernel\_dim} - 1$}
                        \If{$(i \gg j) \& 1 = 1$}
                            \State $x_{\text{candidate}} \leftarrow x_{\text{candidate}} + \text{kernel\_basis}[j]$
                        \EndIf
                    \EndFor
                    
                    \State $\text{total\_candidates\_checked} \leftarrow \text{total\_candidates\_checked} + 1$
                    
                    \If{\textsc{IsAnnihilated}($x_{\text{candidate}}$, $\text{deg}_x$)}
                        \State $z_{\text{final}} \leftarrow z_{\text{cand}}$
                        \State \Comment{--- Verify the solution ---}
                        \State $\text{verification} \leftarrow \delta(z_{\text{final}}) + \varphi_k(x_{\text{candidate}})$
                        \If{$\text{verification} = b_{\text{adm}}$}
                            \State Append $(z_{\text{final}}, x_{\text{candidate}})$ to AnnihilatedSolutions
                            \State $\text{Print}(\text{``Found solution after checking''} + \text{total\_candidates\_checked} + \text{``candidates''})$
                            \State \textbf{return} AnnihilatedSolutions \Comment{Early termination option}
                        \EndIf
                    \EndIf
                    
                    \State \Comment{--- Progress reporting for large searches ---}
                    \If{$\text{total\_candidates\_checked} \bmod 10000 = 0$}
                        \State $\text{Print}(\text{``Progress: checked''} + \text{total\_candidates\_checked} + \text{``candidates...''})$
                    \EndIf
                \EndFor
             \Else
                \State \Comment{No solution exists for this $z_{\text{cand}}$}
                \State $\text{Print}(\text{``No solution for current $z$ combination''})$
             \EndIf
        \EndFor
    \EndFor
    
    \State $\text{Print}(\text{``Search completed. Total candidates checked:''} + \text{total\_candidates\_checked})$
    \State \Return AnnihilatedSolutions
\EndFunction

\State
\State \Comment{--- Helper function for $\mathcal A$-annihilation check ---}
\Function{IsAnnihilated}{$x_{\text{poly}}, \text{degree\_x}$}
    \If{$x_{\text{poly}} = 0$} \Return \textbf{true} \EndIf
    
    \State $\text{max\_t} \leftarrow \lfloor \log_2(\text{degree\_x}) \rfloor + 1$
    \For{$t = 0$ \textbf{to} $\text{max\_t} - 1$}
        \State $j_{\text{sq}} \leftarrow 2^t$
        \State $\text{sq\_action} \leftarrow \text{ApplySteenrodSquares}(x_{\text{poly}}, j_{\text{sq}})$
        \If{$\text{sq\_action} \neq 0$}
            \Return \textbf{false}
        \EndIf
    \EndFor
    \Return \textbf{true}
\EndFunction

\State
\State \Comment{--- Helper function for solving phi system ---}
\Function{SolvePhiSystem}{$b', k, \text{deg\_x}$}
    \State $\mathcal{B}_{H} \leftarrow \text{GenerateBasis}(\text{compositions}(\text{deg\_x}, k))$
    \State $\text{phi\_matrix} \leftarrow \text{BuildPhiMatrix}(\mathcal{B}_{H}, k)$
    \State $b'_{\text{vec}} \leftarrow \text{ConvertToVector}(b')$
    
    \If{$\text{SystemIsConsistent}(\text{phi\_matrix}, b'_{\text{vec}})$}
        \State $\text{particular\_sol} \leftarrow \text{SolveLinearSystem}(\text{phi\_matrix}, b'_{\text{vec}})$
        \State $\text{kernel\_basis} \leftarrow \text{ComputeKernel}(\text{phi\_matrix})$
        \State \Return $\{\text{particular}: \text{particular\_sol}, \text{kernel}: \text{kernel\_basis}\}$
    \Else
        \State \Return null \Comment{System is inconsistent}
    \EndIf
\EndFunction
\end{algorithmic}

\vspace{0.5cm}
    
Note that a critical step in the algorithm is to verify if a candidate solution $x \in H_*(B(\mathbb{Z}/2)^k)$ is $\mathcal{A}$-annihilated, meaning $(x)Sq_*^{2^{t}} = 0$ for all $t \geq 0$ where $2^t \leq \deg(x)$. A candidate solution $x$ is a polynomial of the form $x = \sum_{I} c_I a^{(I)}$, where $a^{(I)} = a_k^{(i_k)} \cdots a_1^{(i_1)}$ and $c_I \in \Z/2$. By linearity of the Steenrod operations, we need to compute $(x)Sq_*^{2^{t}} = \sum_I c_I (a^{(I)})Sq_*^{2^{t}}$. The action on a monomial is determined by equation \eqref{ct3}. This allows us to break down the action on a monomial into actions on its individual divided power factors $a_m^{(i_m)}$. By substituting the formula for the action on a generator, we can compute $(a^{(I)})Sq_*^{2^{t}}$ for any monomial basis element. Summing these results for a candidate solution $x$ and checking if the sum is zero for all required $Sq_*^{2^t}$ (specifically for $t = 0, 1, \ldots, \lfloor \log_2(\deg(x)) \rfloor$) confirms whether it is $\mathcal{A}$-annihilated. The computational complexity of this verification step depends on the number of terms in $x$ and the degree of $x$, as we must check $O(\log(\deg(x)))$ different Steenrod operations, each potentially producing multiple terms.

\subsection{Application to the detection of the indecomposables elements $c_0,$ $d_0$ and $p_0$}
The algorithm described above provides a powerful and systematic tool for investigating fundamental questions about the structure of $\Ext_{\A}$. In particular, it allows us to tackle the problem of detecting specific, well-known, and important cocycles.

In the subsequent application of this work, we will apply Algorithm \ref{alg:preimage} to the indecomposables elements $c_0\in \Ext_{\A}^{3,11}(\Z/2, \Z/2),$ $d_0 \in \Ext_{\A}^{4,18}(\Z/2, \Z/2)$ and $p_0 \in \Ext_{\A}^{4,37}(\Z/2, \Z/2)$. By constructing the appropriate linear system for each element and solving for a preimage, we can determine if they lie in the image of the Singer transfers $\varphi_3$ and $\varphi_4,$ respectively. This provides a direct computational path to confirm or deny their detection by the transfer, a question of significant interest in the field.

\medskip

$\bullet$ \underline{\textbf{The indecomposable element $c_0$:}}

According to Lin~\cite{Lin}, the element $y = \lambda_3\lambda_3 \lambda_2 \in \Lambda_3$ is a representative of $c_0 \in \Ext_{\A}^{3,11}(\mathbb{Z}/2, \mathbb{Z}/2)$.  
By applying Algorithm~\ref{alg:preimage}, we obtain the following results:

$\bullet$ For $z = 0,$ we have
$$x = a_3^{(2)} a_2^{(3)} a_1^{(3)} + a_3^{(1)} a_2^{(4)} a_1^{(3)} + a_3^{(1)} a_2^{(2)} a_1^{(5)} + a_3^{(1)} a_2^{(1)} a_1^{(6)}.$$

$\bullet$ For $z\neq 0,$

$-$ $z \in \big\{\lambda_2 \lambda_7,\  \lambda_7 \lambda_2\big\}$ and $x = a_3^{(2)} a_2^{(3)} a_1^{(3)} + a_3^{(1)} a_2^{(4)} a_1^{(3)} + a_3^{(1)} a_2^{(2)} a_1^{(5)} + a_3^{(1)} a_2^{(1)} a_1^{(6)}.$

$-$ $z = \lambda_8 \lambda_1$ and 
\begin{align*}
 x &= a_3^{(4)}a_2^{(3)}a_1^{(1)} + a_3^{(3)}a_2^{(3)}a_1^{(2)} + a_3^{(2)}a_2^{(5)}a_1^{(1)}+ a_3^{(2)}a_2^{(3)}a_1^{(3)}\\
&  + a_3^{(1)}a_2^{(6)}a_1^{(1)} + a_3^{(1)}a_2^{(4)}a_1^{(3)} + a_3^{(1)}a_2^{(2)}a_1^{(5)} + a_3^{(1)}a_2^{(1)}a_1^{(6)}.
\end{align*}

By hand-checking the expression above, we get: $x \in P_{\A} H_8(B(\mathbb{Z}/2)^3),$ $\delta(z) = 0$ and $\varphi_3(x) = y$.  
Therefore, $[\varphi_3(x)] = [y] = c_0$, which implies that $c_0$ lies in the image of the Singer transfer $\varphi_3$. This result was previously investigated by Boardman \cite{Boardman} through a different method.

\medskip

$\bullet$ \underline{\textbf{The indecomposable element $d_0$:}}

As is known \cite{Lin}, the element $$y =\lambda^{2}_3\lambda_2\lambda_6 +  \lambda^{2}_3\lambda_4^{2} + \lambda_3 \lambda_5\lambda_4\lambda_2 + \lambda_7 \lambda_1\lambda_5\lambda_1 =   \lambda^{2}_3\lambda_2\lambda_6 +  \lambda^{2}_3\lambda_4^{2} + \lambda_3 \lambda_5\lambda_4\lambda_2 + \lambda_3 \lambda_5\lambda^{2}_3.$$ 
is a representative of $d_0 \in \Ext_{\A}^{4,18}(\mathbb{Z}/2, \mathbb{Z}/2)$.  

By applying Algorithm~\ref{alg:preimage}, we obtain
$$
z = \lambda^{2}_3\lambda_{9} + \lambda_{3}\lambda_{9}\lambda_{3},  
$$
and
\begin{align*}
x=& a_4^{(1)} a_3^{(1)} a_2^{(6)} a_1^{(6)} 
+ a_4^{(1)} a_3^{(2)} a_2^{(5)} a_1^{(6)} 
+ a_4^{(1)} a_3^{(3)} a_2^{(4)} a_1^{(6)} 
+ a_4^{(1)} a_3^{(4)} a_2^{(3)} a_1^{(6)} \\
& + a_4^{(1)} a_3^{(5)} a_2^{(2)} a_1^{(6)} 
+ a_4^{(1)} a_3^{(6)} a_2^{(1)} a_1^{(6)} 
+ a_4^{(2)} a_3^{(1)} a_2^{(6)} a_1^{(5)} 
+ a_4^{(2)} a_3^{(2)} a_2^{(5)} a_1^{(5)} \\
& + a_4^{(2)} a_3^{(3)} a_2^{(4)} a_1^{(5)} 
+ a_4^{(2)} a_3^{(4)} a_2^{(3)} a_1^{(5)} 
+ a_4^{(2)} a_3^{(5)} a_2^{(2)} a_1^{(5)} 
+ a_4^{(2)} a_3^{(6)} a_2^{(1)} a_1^{(5)} \\
& + a_4^{(3)} a_3^{(1)} a_2^{(5)} a_1^{(5)} 
+ a_4^{(3)} a_3^{(2)} a_2^{(6)} a_1^{(3)} 
+ a_4^{(3)} a_3^{(3)} a_2^{(2)} a_1^{(6)} 
+ a_4^{(3)} a_3^{(4)} a_2^{(1)} a_1^{(6)} \\
& + a_4^{(3)} a_3^{(4)} a_2^{(2)} a_1^{(5)} 
+ a_4^{(3)} a_3^{(4)} a_2^{(4)} a_1^{(3)} 
+ a_4^{(3)} a_3^{(6)} a_2^{(2)} a_1^{(3)} 
+ a_4^{(4)} a_3^{(1)} a_2^{(6)} a_1^{(3)} \\
& + a_4^{(4)} a_3^{(2)} a_2^{(5)} a_1^{(3)} 
+ a_4^{(4)} a_3^{(3)} a_2^{(4)} a_1^{(3)} 
+ a_4^{(4)} a_3^{(4)} a_2^{(3)} a_1^{(3)} 
+ a_4^{(4)} a_3^{(5)} a_2^{(2)} a_1^{(3)} \\
& + a_4^{(4)} a_3^{(6)} a_2^{(1)} a_1^{(3)} 
+ a_4^{(5)} a_3^{(1)} a_2^{(3)} a_1^{(5)} 
+ a_4^{(5)} a_3^{(2)} a_2^{(1)} a_1^{(6)} 
+ a_4^{(5)} a_3^{(2)} a_2^{(2)} a_1^{(5)} \\
& + a_4^{(5)} a_3^{(2)} a_2^{(4)} a_1^{(3)} 
+ a_4^{(5)} a_3^{(3)} a_2^{(1)} a_1^{(5)} 
+ a_4^{(5)} a_3^{(3)} a_2^{(3)} a_1^{(3)} 
+ a_4^{(5)} a_3^{(5)} a_2^{(1)} a_1^{(3)} \\
& + a_4^{(6)} a_3^{(1)} a_2^{(1)} a_1^{(6)} 
+ a_4^{(6)} a_3^{(1)} a_2^{(2)} a_1^{(5)} 
+ a_4^{(6)} a_3^{(1)} a_2^{(4)} a_1^{(3)} 
+ a_4^{(6)} a_3^{(2)} a_2^{(3)} a_1^{(3)}.
\end{align*}

By direct verification, we have the following:

$\bullet$ $x \in P_{\A} H_14(B(\mathbb{Z}/2)^4),$ and due to the instability condition, we only consider the effect of $Sq_*^{2^{j}}$ for $0\leq j\leq 2;$

$\bullet$

\begin{align*}
\varphi_4(x) &= \lambda_3 \lambda_1 \lambda_5^2 + \lambda_3 \lambda_1 \lambda_6 \lambda_4 + \lambda_3 \lambda_1 \lambda_8 \lambda_2 + \lambda_3^2 \lambda_2 \lambda_6 \\
&+ \lambda_3^3 \lambda_5 + \lambda_3^2 \lambda_4^2 + \lambda_3^2 \lambda_7 \lambda_1 + \lambda_3 \lambda_5 \lambda_1 \lambda_5 \\
&+ \lambda_3 \lambda_5 \lambda_3^2 + \lambda_3 \lambda_5 \lambda_4 \lambda_2 + \lambda_3 \lambda_5^2 \lambda_1\\
&= \lambda_3^2 \lambda_2 \lambda_6 + \lambda_3^3 \lambda_5 + \lambda_3^2 \lambda_4^2 + \lambda_3 \lambda_5 \lambda_4 \lambda_2\\
&=(\lambda_3^2 \lambda_2 \lambda_6 + \lambda_3^2 \lambda_4^2 + \lambda_3 \lambda_5 \lambda_4 \lambda_2 + \lambda_3^2 \lambda_5\lambda_3)+ (\lambda_3^2 \lambda_5\lambda_3+ \lambda_3^3 \lambda_5)\\
&=y +  (\lambda_3^2 \lambda_5\lambda_3+ \lambda_3^3 \lambda_5).
\end{align*}

$\bullet$ $\delta(z) = \lambda_3^2 \lambda_5\lambda_3+ \lambda_3^3 \lambda_5.$ 

Thus, we get
 $$\varphi_4(x) +  \delta(z) =   y.$$  
Therefore, $[\varphi_4(x)] = [y + \delta(z)] = [y] = d_0$, which implies that $d_0$ lies in the image of $\varphi_4.$ Note that an alternative approach to this result was given earlier by Ha \cite{Ha}.

\begin{remark}

The author of \cite{Sum} mentioned that the recursive function he used differs from the one in Corollary \ref{mdp}, as follows:

\begin{equation}\label{ctS}
\varphi_k\bigg(a_k^{(t_k)}a_{k-1}^{(t_{k-1})}\dots a_1^{(t_1)}\bigg) = \sum_{i\ge t_k} \varphi_{k-1}\bigg(\big(a_{k-1}^{(t_{k-1})}\dots a_1^{(t_1)}\big)\cdot \Sq_*^{i - t_k}\bigg)\lambda_{i}.
\end{equation}

Then, by a simple computation, we obtain:

\begin{align*}
&\varphi_4(a_4^{(t_4)}a_{3}^{(t_{3})}\dots a_1^{(t_1)})\\
&=\sum_{i_4 \ge t_4} 
\sum_{\substack{
    s_1 + s_2 + s_3 = i_4 - t_4 \\
    s_1, s_2, s_3 \ge 0
}} 
\sum_{i_3 \ge -s_1 + t_3} 
\sum_{\substack{
    u_1 + u_2 = i_3 + s_1 - t_3 \\
    u_1, u_2 \ge 0
}} 
\sum_{i_2 \ge -s_2 - u_1 + t_2} \\
&\times
    \binom{-s_1 + t_3}{s_1}
    \binom{-s_2 + t_2}{s_2}
    \binom{-s_3 + t_1}{s_3} \\
&\times
    \binom{-s_2 - u_1 + t_2}{u_1}
    \binom{-s_3 - u_2 + t_1}{u_2} 
    \binom{
        -i_2 - s_2 - s_3 - u_1 - u_2 + t_1 + t_2
    }{
        i_2 + s_2 + u_1 - t_2
    } \lambda_{-i_2 - s_2 - s_3 - u_1 - u_2 + t_1 + t_2}\lambda_{i_2} \lambda_{i_3}    \lambda_{i_4}.
\end{align*}

\medskip

After that, the author of \cite{Sum} considers the following element:

\medskip

\begin{align*}
x = q_{4, 3} = & a_4^{(6)} a_3^{(6)} a_2^{(1)} a_1^{(1)}
+ a_4^{(6)} a_3^{(5)} a_2^{(2)} a_1^{(1)}
+ a_4^{(6)} a_3^{(4)} a_2^{(3)} a_1^{(1)}
+ a_4^{(6)} a_3^{(3)} a_2^{(4)} a_1^{(1)} \\[2pt]
& + a_4^{(6)} a_3^{(2)} a_2^{(5)} a_1^{(1)}
+ a_4^{(6)} a_3^{(1)} a_2^{(6)} a_1^{(1)}
+ a_4^{(5)} a_3^{(6)} a_2^{(1)} a_1^{(2)}
+ a_4^{(5)} a_3^{(5)} a_2^{(2)} a_1^{(2)} \\[2pt]
& + a_4^{(5)} a_3^{(4)} a_2^{(3)} a_1^{(2)}
+ a_4^{(5)} a_3^{(3)} a_2^{(4)} a_1^{(2)}
+ a_4^{(5)} a_3^{(2)} a_2^{(5)} a_1^{(2)}
+ a_4^{(5)} a_3^{(1)} a_2^{(6)} a_1^{(2)} \\[2pt]
& + a_4^{(5)} a_3^{(5)} a_2^{(1)} a_1^{(3)}
+ a_4^{(3)} a_3^{(6)} a_2^{(2)} a_1^{(3)}
+ a_4^{(6)} a_3^{(2)} a_2^{(3)} a_1^{(3)}
+ a_4^{(6)} a_3^{(1)} a_2^{(4)} a_1^{(3)} \\[2pt]
& + a_4^{(5)} a_3^{(2)} a_2^{(4)} a_1^{(3)}
+ a_4^{(3)} a_3^{(4)} a_2^{(4)} a_1^{(3)}
+ a_4^{(3)} a_3^{(2)} a_2^{(6)} a_1^{(3)}
+ a_4^{(3)} a_3^{(6)} a_2^{(1)} a_1^{(4)} \\[2pt]
& + a_4^{(3)} a_3^{(5)} a_2^{(2)} a_1^{(4)}
+ a_4^{(3)} a_3^{(4)} a_2^{(3)} a_1^{(4)}
+ a_4^{(3)} a_3^{(3)} a_2^{(4)} a_1^{(4)}
+ a_4^{(3)} a_3^{(2)} a_2^{(5)} a_1^{(4)} \\[2pt]
& + a_4^{(3)} a_3^{(1)} a_2^{(6)} a_1^{(4)}
+ a_4^{(5)} a_3^{(3)} a_2^{(1)} a_1^{(5)}
+ a_4^{(6)} a_3^{(1)} a_2^{(2)} a_1^{(5)}
+ a_4^{(5)} a_3^{(2)} a_2^{(2)} a_1^{(5)} \\[2pt]
& + a_4^{(3)} a_3^{(4)} a_2^{(2)} a_1^{(5)}
+ a_4^{(5)} a_3^{(1)} a_2^{(3)} a_1^{(5)}
+ a_4^{(3)} a_3^{(3)} a_2^{(3)} a_1^{(5)}
+ a_4^{(3)} a_3^{(1)} a_2^{(5)} a_1^{(5)} \\[2pt]
& + a_4^{(6)} a_3^{(1)} a_2^{(1)} a_1^{(6)}
+ a_4^{(5)} a_3^{(2)} a_2^{(1)} a_1^{(6)}
+ a_4^{(3)} a_3^{(4)} a_2^{(1)} a_1^{(6)}
+ a_4^{(3)} a_3^{(3)} a_2^{(2)} a_1^{(6)}
\end{align*}

\medskip

In~\cite{Sum}, the author uses the variable ordering \( a_1, a_2, a_3, a_4 \), so the polynomial is written as
\[
q_{4,3} = a_1^{(1)} a_2^{(1)} a_3^{(6)} a_4^{(6)} + a_1^{(1)} a_2^{(2)} a_3^{(5)} a_4^{(6)} + a_1^{(1)} a_2^{(3)} a_3^{(4)} a_4^{(6)}+ \cdots + a_1^{(6)} a_2^{(2)} a_3^{(3)} a_4^{(3)}.
\]
Under the convention used in our present paper, where the variable order is \( (a_4, a_3, a_2, a_1) \), this same polynomial is represented in the form shown above.

\medskip

Then, he claimed that $\varphi_4(q_{4,3}) = \overline{d_0} + \delta(\lambda_1\lambda_3\lambda_9\lambda_3 + \lambda_1\lambda^{2}_3\lambda_9)$. However, \textbf{this is fundamentally false}, since 
\[
z = \lambda_1\lambda_3\lambda_9\lambda_3 + \lambda_1\lambda^{2}_3\lambda_9 \not\in \Lambda_3.
\]
This is not merely a minor inaccuracy; the \textbf{serious error} lies in the use of the recursive function~\eqref{ctS}, as mentioned in Sum's paper~\cite{Sum}. More precisely, with the choice of representative element
\[
\overline{d_0} = \lambda_6 \lambda_2 \lambda^{2}_3 + \lambda^{2}_4 \lambda^{2}_3 + \lambda_2 \lambda_4 \lambda_5 \lambda_3 + \lambda_1 \lambda_5 \lambda_1 \lambda_7
\]
for \( d_0 \in \mathrm{Ext}_{\mathcal{A}}^{4,18}(\mathbb{Z}/2, \mathbb{Z}/2) \), as presented in~\cite{Sum}, there \textbf{does not exist} any
\[
x \in P_{\mathcal{A}} H_{14}(B(\mathbb{Z}/2)^4) \quad \text{and} \quad z \in \Lambda_3
\]
such that
\[
\varphi_4(x) + \delta(z) = y.
\]
We have also verified this rigorously by computer using our algorithm~\ref{alg:preimage}, in which we replaced the recursive function in Corollary~\ref{mdp} with the one~\eqref{ctS} used in \cite{Sum}. The results clearly confirm the issue described above.

\medskip

We now provide a detailed explanation of why the computation in \cite{Sum} is false, and why \textbf{the result is essentially irreparable}, both due to the use of the recursive function \eqref{ctS} and the specific choice of the representative element $\overline{d_0}$ as given in \cite{Sum}.

\medskip

First,  by applying \eqref{ctS} and reducing to the admissible form using the Adem relations \eqref{ct1}, we get
\begin{align*}
\varphi_4(q_{4, 3}) &= \lambda_1 \lambda_5^2 \lambda_3 + \lambda_1 \lambda_7 \lambda_3^2 + \lambda_2 \lambda_4 \lambda_5 \lambda_3 + \lambda_2 \lambda_8 \lambda_1 \lambda_3 \\
&+ \lambda_3^2 \lambda_5 \lambda_3 + \lambda_4^2 \lambda_3^2 + \lambda_4 \lambda_6 \lambda_1 \lambda_3 + \lambda_5 \lambda_1 \lambda_5 \lambda_3 \\
&+ \lambda_5 \lambda_3^3 + \lambda_5^2 \lambda_1 \lambda_3 + \lambda_6 \lambda_2 \lambda_3^2\\
&=\lambda_1 \lambda_5^2 \lambda_3 + \lambda_2 \lambda_3 \lambda_6 \lambda_3 + \lambda_2 \lambda_5 \lambda_4 \lambda_3 + \lambda_4 \lambda_3 \lambda_4 \lambda_3 + \lambda_5 \lambda_3^3\\
&= (\lambda_1 \lambda_3^2 \lambda_7 + \lambda_2 \lambda_4 \lambda_5 \lambda_3 + \lambda_4^2 \lambda_3^2 + \lambda_5 \lambda_3^3)\\
&+ (\lambda_1 \lambda_5^2 \lambda_3 + \lambda_2 \lambda_3 \lambda_6 \lambda_3 + \lambda_2 \lambda_5 \lambda_4 \lambda_3 + \lambda_4 \lambda_3 \lambda_4 \lambda_3 +  \lambda_1 \lambda_3^2 \lambda_7 + \lambda_2 \lambda_4 \lambda_5 \lambda_3 + \lambda_4^2 \lambda_3^2)
&= d^{*}_0 +R,
\end{align*}
where $d^{*}_0 = \lambda_1 \lambda_3^2 \lambda_7 + \lambda_2 \lambda_4 \lambda_5 \lambda_3 + \lambda_4^2 \lambda_3^2 + \lambda_5 \lambda_3^3$ and $R = \lambda_1 \lambda_5^2 \lambda_3 + \lambda_2 \lambda_3 \lambda_6 \lambda_3 + \lambda_2 \lambda_5 \lambda_4 \lambda_3 + \lambda_4 \lambda_3 \lambda_4 \lambda_3 +  \lambda_1 \lambda_3^2 \lambda_7 + \lambda_2 \lambda_4 \lambda_5 \lambda_3 + \lambda_4^2 \lambda_3^2.$

\medskip

To verify the above computational result, we provide the following simple algorithm implemented in \textsc{SageMath}, which enables the reader to easily check the calculation by hand.

\medskip

\begin{lstlisting}[language=Python, caption={\textbf{A SAGEMATH implementation for computing $\varphi_4(q_{4,3})$}}, label={lst:SAGEMATH_code}]
from collections import Counter
from functools import lru_cache
from typing import Tuple, Dict, List

@lru_cache(maxsize=None)
def binom2(n: int, k: int) -> int:
    if k < 0 or k > n: return 0
    return 1 if (k & n) == k else 0

@lru_cache(maxsize=None)
def sq_star(monomial: Tuple[int, ...], i: int) -> Dict[Tuple[int, ...], int]:
    if i < 0: return {}
    if not monomial: return {(): 1} if i == 0 else {}
    if i > sum(monomial): return {}
    t_k, m_prime = monomial[0], monomial[1:]
    result = Counter()
    for i_k in range(min(i, t_k) + 1):
        if binom2(t_k - i_k, i_k) == 0: continue
        new_tk = (t_k - i_k,) if t_k - i_k > 0 else ()
        sub_poly = sq_star(m_prime, i - i_k)
        for sub_mon, sub_coeff in sub_poly.items():
            result[new_tk + sub_mon] = (result[new_tk + sub_mon] + sub_coeff) % 2
    return {m: c for m, c in result.items() if c == 1}

@lru_cache(maxsize=None)
def phi(k: int, monomial: Tuple[int, ...]) -> Dict[Tuple[int, ...], int]:
    if k == 0: return {(): 1}
    padded_monomial = (0,) * (k - len(monomial)) + monomial
    total_result = Counter()
    t_k = padded_monomial[0]
    remaining_part = padded_monomial[1:]
    for i in range(t_k, sum(padded_monomial) + k * 4):
        j = i - t_k
        poly_from_sq = sq_star(remaining_part, j)
        if not poly_from_sq: continue
        for m_new in poly_from_sq:
            poly_from_phi = phi(k - 1, m_new)
            for lambda_exp in poly_from_phi:
                new_lambda_exp = lambda_exp + (i,)
                total_result[new_lambda_exp] = (total_result[new_lambda_exp] + 1) % 2
    return {exp: c for exp, c in total_result.items() if c == 1}

def reduce_lambda_poly(poly_dict: Dict[Tuple[int, ...], int]) -> Dict[Tuple[int, ...], int]:
    poly = Counter({tuple(t): c & 1 for t, c in poly_dict.items() if c & 1})
    while True:
        found_reduction = False
        for term in list(poly.keys()):
            if term not in poly: continue
            for i in range(len(term) - 1):
                s, t = term[i], term[i + 1]
                if s > 2 * t:
                    found_reduction = True
                    coeff = poly.pop(term)
                    prefix, suffix = term[:i], term[i + 2:]
                    for j in range(s + t + 1):
                        if binom2(j - t - 1, 2 * j - s):
                            new_term = prefix + (s + t - j, j) + suffix
                            poly[new_term] = (poly.get(new_term, 0) + coeff) % 2
                            if poly[new_term] == 0: del poly[new_term]
                    break
            if found_reduction: break
        if not found_reduction: break
    return {p: c for p, c in poly.items() if c == 1}

def format_lambda_poly(poly: Dict[Tuple[int, ...], int]) -> str:
    if not poly: return "0"
    sorted_terms = sorted(poly.keys(), key=lambda t: (sum(t), t))
    return " + ".join("".join(f"lambda_{i}" for i in term) if term else "1" for term in sorted_terms)

def clear_caches():
    binom2.cache_clear()
    sq_star.cache_clear()
    phi.cache_clear()

def compute_varphi_sum(k: int, monomials: List[Tuple[int, ...]]) -> Dict[Tuple[int, ...], int]:
    print(f"Computing phi_{k} for {len(monomials)} monomials...")
    clear_caches()
    total_poly = Counter()
    for i, mono in enumerate(monomials):
        if i > 0 and i % 10 == 0:
            print(f"  Processing monomial {i+1}/{len(monomials)}: {mono}")
        result = phi(k, mono)
        total_poly.update(result)
    unreduced = {p: c % 2 for p, c in total_poly.items() if c % 2 == 1}
    print(f"\nUnreduced polynomial has {len(unreduced)} terms. Applying Adem reduction...")
    reduced = reduce_lambda_poly(unreduced)
    print(f"Reduced polynomial has {len(reduced)} terms.")
    return unreduced, reduced

if __name__ == "__main__":
    print("\n" + "=" * 60)
    print("MAIN COMPUTATION")
    print("=" * 60)
    time_start = cputime()
    q_sample = [
    (6, 6, 1, 1), (6, 5, 2, 1), (6, 4, 3, 1), (6, 3, 4, 1),
    (6, 2, 5, 1), (6, 1, 6, 1), (5, 6, 1, 2), (5, 5, 2, 2),
    (5, 4, 3, 2), (5, 3, 4, 2), (5, 2, 5, 2), (5, 1, 6, 2),
    (5, 5, 1, 3), (3, 6, 2, 3), (6, 2, 3, 3), (6, 1, 4, 3),
    (5, 2, 4, 3), (3, 4, 4, 3), (3, 2, 6, 3), (3, 6, 1, 4),
    (3, 5, 2, 4), (3, 4, 3, 4), (3, 3, 4, 4), (3, 2, 5, 4),
    (3, 1, 6, 4), (5, 3, 1, 5), (6, 1, 2, 5), (5, 2, 2, 5),
    (3, 4, 2, 5), (5, 1, 3, 5), (3, 3, 3, 5), (3, 1, 5, 5),
    (6, 1, 1, 6), (5, 2, 1, 6), (3, 4, 1, 6), (3, 3, 2, 6)
    ]
    unreduced, reduced = compute_varphi_sum(4, q_sample)
    print(f"\nFinal unreduced result: varphi_4(sample) = {format_lambda_poly(unreduced)}")
    print(f"\nFinal reduced result: varphi_4(sample) = {format_lambda_poly(reduced)}")
    clear_caches()

    print("\n" + "=" * 80)
    print("ENTIRE COMPUTATION PROCESS COMPLETED")
    time_end = cputime()
    print(f"Total execution time: {time_end - time_start:.4f} seconds")
    print("=" * 80)
\end{lstlisting}

\medskip

The result produced by the above algorithm is as follows:

\medskip

\begin{lstlisting}
=============================================================================
MAIN COMPUTATION
=============================================================================
Computing phi_4 for 36 monomials...
  Processing monomial 11/36: (5, 2, 5, 2)
  Processing monomial 21/36: (3, 5, 2, 4)
  Processing monomial 31/36: (3, 3, 3, 5)

Unreduced polynomial has 11 terms. Applying Adem reduction...
Reduced polynomial has 5 terms.

Final unreduced result: varphi_4(sample) = lambda_1lambda_5lambda_5lambda_3 + lambda_1lambda_7lambda_3lambda_3 + lambda_2lambda_4lambda_5lambda_3 + lambda_2lambda_8lambda_1lambda_3 + lambda_3lambda_3lambda_5lambda_3 + lambda_4lambda_4lambda_3lambda_3 + lambda_4lambda_6lambda_1lambda_3 + lambda_5lambda_1lambda_5lambda_3 + lambda_5lambda_3lambda_3lambda_3 + lambda_5lambda_5lambda_1lambda_3 + lambda_6lambda_2lambda_3lambda_3

Final reduced result: varphi_4(sample) = lambda_1lambda_5lambda_5lambda_3 + lambda_2lambda_3lambda_6lambda_3 + lambda_2lambda_5lambda_4lambda_3 + lambda_4lambda_3lambda_4lambda_3 + lambda_5lambda_3lambda_3lambda_3

=============================================================================
ENTIRE COMPUTATION PROCESS COMPLETED
Total execution time: 0.0424 seconds
=============================================================================
\end{lstlisting}

\medskip

Note that in \cite{Sum}, the monomial 
\begin{align*}
 \overline{d_0} = \lambda_6 \lambda_2 \lambda^{2}_3 + \lambda^{2}_4\lambda^{2}_3 + \lambda_2\lambda_4\lambda_5\lambda_3 + \lambda_1\lambda_5\lambda_1\lambda_7 = d^{*}_0 
\end{align*}
is a representative of $d_0\in \Ext_{\A}^{4,18}(\mathbb{Z}/2, \mathbb{Z}/2).$ However, the polynomial $R$ is not a cocycle in $\Lambda.$ Indeed, we need to show that $\delta(R) \neq 0.$ By a direct computation from \eqref{ct2}, we obtain:

\begin{align*}
\delta(R) = &\lambda_{4} \lambda_{3} \lambda_{2} \lambda_{1} \lambda_{3} 
+ \lambda_{4} \lambda_{2} \lambda_{1} \lambda_{3}^2 
+ \lambda_{3} \lambda_{2}^2 \lambda_{3}^2 
+ \lambda_{2} \lambda_{1} \lambda_{4} \lambda_{3}^2 \\
&+ \lambda_{1} \lambda_{2} \lambda_{4} \lambda_{3}^2 
+ \lambda_{2} \lambda_{1} \lambda_{3} \lambda_{4} \lambda_{3} 
+ \lambda_{1} \lambda_{2} \lambda_{3} \lambda_{4} \lambda_{3} 
+ \lambda_{2}^2 \lambda_{1} \lambda_{5} \lambda_{3} 
+ \lambda_{2} \lambda_{1} \lambda_{2} \lambda_{5} \lambda_{3}
\end{align*}
This implies that $\delta(R) \neq 0$, and therefore there does not exist any $z \in \Lambda_3$ such that $\delta(z) = R.$ Thus, 
$$ \varphi_4(q_{4,3}) + \delta(z)\neq y.$$

\medskip

In fact, throughout the entire paper \cite{Sum}, the author also arbitrarily considers similar elements without explaining their origin or justification.

\medskip

In the case where $R$ is a cocycle, then after carrying out the manual computation and obtaining the result
\[
\varphi_4(q_{4,3}) = y + R,
\]
where
\[
R = \lambda_{5}\lambda^{3}_{3}
  + \lambda^{2}_{3}\lambda_{5}\lambda_{3} 
  + \lambda_{1}\lambda^{2}_{5}\lambda_{3},
\]
\textbf{it would still be extremely complicated and arguably infeasible to determine an element \( z\in \Lambda_3 \) such that \( \delta(z) = R\in \Lambda_4 \), purely by hand}. Of course, such an element \( z \) satisfying \( \delta(z) = R \) does not always exist, since for this to happen, \( R \) must first be a cocycle in \( \Lambda_4 \), i.e., \( \delta(R) = 0 \). It is also note that this is a \textbf{necessary but not sufficient} condition. The condition is \textbf{necessary}, as applying the differential to both sides of the equation $\delta(z) = R$ yields $\delta(R) = \delta(\delta(z)) = 0$. However, the condition is \textbf{not sufficient}, because even if $\delta(R)=0$ (meaning $R$ is a cocycle), there is no guarantee that it is also a coboundary, that is, that an element $z$ exists such that $\delta(z)=R$. If no such $z$ exists, then $R$ represents a non-trivial cohomology class.

This is precisely why we make such a claim, and in what follows, we will explain the reasoning behind it in detail.

The problem $\delta(z) = R$ is, in fact, a system of linear equations over the field $\Z/2$ (the field with only two elements, 0 and 1).

\begin{itemize}
    \item \textbf{Unknowns:} The unknowns are the coefficients of $z$. We write $z$ as a general polynomial $z = \sum c_i \lambda_{J_i}$, where $\lambda_{J_i}$ are the basis monomials and the $c_i$ are the unknowns (equal to 0 or 1) that we need to find.
    \item \textbf{Equations:} These are established by comparing the coefficients of each basis monomial on both sides of the equation.
\end{itemize}

Although it is theoretically just a matter of solving a system of linear equations, manual computation becomes nearly impossible for two main reasons:

\begin{itemize}
    \item \textbf{The enormous number of variables and equations:} The number of basis monomials in the Lambda algebra grows explosively with the degree. For elements of even moderate degree, the number of unknowns $c_i$ (and the corresponding number of equations) can reach into the hundreds, thousands, or even more.
    \item \textbf{The complexity of the differential $\delta$:} Computing $\delta$ for a long monomial requires repeated applications of the Leibniz rule and the use of the formula for the differential of each $\lambda_n$, which is itself a complex sum derived from the Adem relations. This process is very time-consuming and extremely error-prone.
\end{itemize}

If one were forced to compute this by hand, the following steps would have to be taken, which are also the steps implemented by the computer algorithm:

\subsubsection*{Step 1: Determine the spaces and unknowns}
\begin{itemize}
    \item Identify the target element $R \in \Lambda_k$ and its degree.
    \item The desired element $z$ will lie in the space $\Lambda_{k-1}$ and will have a degree one greater than the degree of $R$.
    \item List all basis monomials $\{\lambda_{J_1}, \lambda_{J_2}, \dots, \lambda_{J_m}\}$ in the space where $z$ lives.
    \item Write $z$ in its general form: $z = c_1\lambda_{J_1} + c_2\lambda_{J_2} + \dots + c_m\lambda_{J_m}$. The $c_i$ are the unknowns.
\end{itemize}

\subsubsection*{Step 2: Compute the differential of each basis element}
\begin{itemize}
    \item For each basis monomial $\lambda_{J_i}$, manually compute its differential $\delta(\lambda_{J_i})$.
    \item This is the most laborious step, requiring the use of the Leibniz rule and the formula $\delta(\lambda_n) = \sum_{j} \binom{n-j-1}{j+1}\lambda_{n-j-1}\lambda_j$. Each result is a polynomial in $\Lambda_k$.
\end{itemize}

\subsubsection*{Step 3: Set up the system of linear equations}
\begin{itemize}
    \item The full equation is: $c_1\delta(\lambda_{J_1}) + c_2\delta(\lambda_{J_2}) + \dots + c_m\delta(\lambda_{J_m}) = R$.
    \item List all basis monomials $\{\lambda_{K_1}, \dots, \lambda_{K_p}\}$ that appear on the left or right-hand side. Each $\lambda_{K_j}$ will correspond to one equation.
    \item For each $\lambda_{K_j}$, compare its coefficient on both sides to generate a linear equation for the unknowns $c_i$.
\end{itemize}

\subsubsection*{Step 4: Solve the system of equations}
\begin{itemize}
    \item Write the system of equations in matrix form and solve it using Gaussian elimination over the field $\Z/2$.
    \item If the system has a solution, one will find the values (0 or 1) for the unknowns $c_i$.
    \item Substitute the found values of $c_i$ back into the expression for $z$ from Step 1 to obtain the final result.
\end{itemize}

Thus, the problem mentioned above serves as a prime example of why computational algebra tools are essential: they transform a massive and error-prone manual task into one that a computer can handle quickly and accurately. We would also like to emphasize that performing explicit computations to precisely determine elements in \( P_{\A} H_*(B(\mathbb{Z}/2)^k) \) is a highly nontrivial task when carried out by hand.


\end{remark}

$\bullet$ \underline{\textbf{The indecomposable element $p_0$:}}

Based on \cite{Lin}, we can see that the element $$y = \lambda^{2}_7 \lambda_5 \lambda_{14} + \lambda^{2}_7 \lambda_9\lambda_{10} + \lambda_7 \lambda_{11}\lambda_9\lambda_6$$ 
is a representative of $p_0 \in \Ext_{\A}^{4,37}(\mathbb{Z}/2, \mathbb{Z}/2)$.  

By applying Algorithm~\ref{alg:preimage}, we obtain $z =0$ and
\[
\begin{aligned}
x = & a_4^{(14)} a_3^{(5)} a_2^{(7)} a_1^{(7)} + a_4^{(14)} a_3^{(3)} a_2^{(9)} a_1^{(7)} + a_4^{(14)} a_3^{(3)} a_2^{(5)} a_1^{(11)} + a_4^{(14)} a_3^{(3)} a_2^{(3)} a_1^{(13)} \\
& + a_4^{(13)} a_3^{(6)} a_2^{(7)} a_1^{(7)} + a_4^{(13)} a_3^{(3)} a_2^{(10)} a_1^{(7)} + a_4^{(13)} a_3^{(3)} a_2^{(6)} a_1^{(11)} + a_4^{(13)} a_3^{(3)} a_2^{(3)} a_1^{(14)} \\
& + a_4^{(11)} a_3^{(6)} a_2^{(9)} a_1^{(7)} + a_4^{(11)} a_3^{(6)} a_2^{(5)} a_1^{(11)} + a_4^{(11)} a_3^{(6)} a_2^{(3)} a_1^{(13)} + a_4^{(11)} a_3^{(5)} a_2^{(10)} a_1^{(7)} \\
& + a_4^{(11)} a_3^{(5)} a_2^{(6)} a_1^{(11)} + a_4^{(11)} a_3^{(5)} a_2^{(3)} a_1^{(14)} + a_4^{(10)} a_3^{(9)} a_2^{(7)} a_1^{(7)} + a_4^{(10)} a_3^{(7)} a_2^{(9)} a_1^{(7)} \\
& + a_4^{(10)} a_3^{(7)} a_2^{(5)} a_1^{(11)} + a_4^{(10)} a_3^{(7)} a_2^{(3)} a_1^{(13)} + a_4^{(10)} a_3^{(5)} a_2^{(11)} a_1^{(7)} + a_4^{(10)} a_3^{(3)} a_2^{(13)} a_1^{(7)} \\
& + a_4^{(9)} a_3^{(10)} a_2^{(7)} a_1^{(7)} + a_4^{(9)} a_3^{(7)} a_2^{(10)} a_1^{(7)} + a_4^{(9)} a_3^{(7)} a_2^{(6)} a_1^{(11)} + a_4^{(9)} a_3^{(7)} a_2^{(3)} a_1^{(14)} \\
& + a_4^{(9)} a_3^{(6)} a_2^{(11)} a_1^{(7)} + a_4^{(9)} a_3^{(3)} a_2^{(14)} a_1^{(7)} + a_4^{(7)} a_3^{(10)} a_2^{(9)} a_1^{(7)} + a_4^{(7)} a_3^{(10)} a_2^{(5)} a_1^{(11)} \\
& + a_4^{(7)} a_3^{(10)} a_2^{(3)} a_1^{(13)} + a_4^{(7)} a_3^{(9)} a_2^{(10)} a_1^{(7)} + a_4^{(7)} a_3^{(9)} a_2^{(6)} a_1^{(11)} + a_4^{(7)} a_3^{(9)} a_2^{(3)} a_1^{(14)} \\
& + a_4^{(7)} a_3^{(6)} a_2^{(13)} a_1^{(7)} + a_4^{(7)} a_3^{(5)} a_2^{(14)} a_1^{(7)} + a_4^{(6)} a_3^{(13)} a_2^{(7)} a_1^{(7)} + a_4^{(6)} a_3^{(11)} a_2^{(9)} a_1^{(7)} \\
& + a_4^{(6)} a_3^{(11)} a_2^{(5)} a_1^{(11)} + a_4^{(6)} a_3^{(11)} a_2^{(3)} a_1^{(13)} + a_4^{(6)} a_3^{(9)} a_2^{(7)} a_1^{(11)} + a_4^{(6)} a_3^{(7)} a_2^{(7)} a_1^{(13)} \\
& + a_4^{(6)} a_3^{(5)} a_2^{(11)} a_1^{(11)} + a_4^{(6)} a_3^{(3)} a_2^{(13)} a_1^{(11)} + a_4^{(5)} a_3^{(14)} a_2^{(7)} a_1^{(7)} + a_4^{(5)} a_3^{(11)} a_2^{(10)} a_1^{(7)} \\
& + a_4^{(5)} a_3^{(11)} a_2^{(6)} a_1^{(11)} + a_4^{(5)} a_3^{(11)} a_2^{(3)} a_1^{(14)} + a_4^{(5)} a_3^{(10)} a_2^{(7)} a_1^{(11)} + a_4^{(5)} a_3^{(7)} a_2^{(7)} a_1^{(14)} \\
& + a_4^{(5)} a_3^{(6)} a_2^{(11)} a_1^{(11)} + a_4^{(5)} a_3^{(3)} a_2^{(14)} a_1^{(11)} + a_4^{(3)} a_3^{(14)} a_2^{(9)} a_1^{(7)} + a_4^{(3)} a_3^{(14)} a_2^{(5)} a_1^{(11)} \\
& + a_4^{(3)} a_3^{(14)} a_2^{(3)} a_1^{(13)} + a_4^{(3)} a_3^{(13)} a_2^{(10)} a_1^{(7)} + a_4^{(3)} a_3^{(13)} a_2^{(6)} a_1^{(11)} + a_4^{(3)} a_3^{(13)} a_2^{(3)} a_1^{(14)} \\
& + a_4^{(3)} a_3^{(10)} a_2^{(7)} a_1^{(13)} + a_4^{(3)} a_3^{(9)} a_2^{(7)} a_1^{(14)} + a_4^{(3)} a_3^{(6)} a_2^{(11)} a_1^{(13)} + a_4^{(3)} a_3^{(5)} a_2^{(11)} a_1^{(14)} \\
& + a_4^{(3)} a_3^{(3)} a_2^{(14)} a_1^{(13)} + a_4^{(3)} a_3^{(3)} a_2^{(13)} a_1^{(14)}.
\end{aligned}
\]

\medskip

We can verify this directly through computation as follows:

\medskip

$\bullet$ Due to the instability condition, it suffices to check whether $x \in P_{\A} H_{33}(B(\mathbb{Z}/2)^4)$ is annihilated by the dual Steenrod operations $Sq^{2^{s}}_*$ for $0\leq s\leq 3.$

\[
\begin{aligned}
(x)Sq^1_*=& a_4^{(13)} a_3^{(5)} a_2^{(7)} a_1^{(7)} + a_4^{(13)} a_3^{(3)} a_2^{(9)} a_1^{(7)} + a_4^{(13)} a_3^{(3)} a_2^{(5)} a_1^{(11)} + a_4^{(13)} a_3^{(3)} a_2^{(3)} a_1^{(13)} \\
& + a_4^{(13)} a_3^{(5)} a_2^{(7)} a_1^{(7)} + a_4^{(13)} a_3^{(3)} a_2^{(9)} a_1^{(7)} + a_4^{(13)} a_3^{(3)} a_2^{(5)} a_1^{(11)} + a_4^{(13)} a_3^{(3)} a_2^{(3)} a_1^{(13)} \\
& + a_4^{(11)} a_3^{(5)} a_2^{(9)} a_1^{(7)} + a_4^{(11)} a_3^{(5)} a_2^{(5)} a_1^{(11)} + a_4^{(11)} a_3^{(5)} a_2^{(3)} a_1^{(13)} + a_4^{(11)} a_3^{(5)} a_2^{(9)} a_1^{(7)} \\
& + a_4^{(11)} a_3^{(5)} a_2^{(5)} a_1^{(11)} + a_4^{(11)} a_3^{(5)} a_2^{(3)} a_1^{(13)} + a_4^{(9)} a_3^{(9)} a_2^{(7)} a_1^{(7)} + a_4^{(9)} a_3^{(7)} a_2^{(9)} a_1^{(7)} \\
& + a_4^{(9)} a_3^{(7)} a_2^{(5)} a_1^{(11)} + a_4^{(9)} a_3^{(7)} a_2^{(3)} a_1^{(13)} + a_4^{(9)} a_3^{(5)} a_2^{(11)} a_1^{(7)} + a_4^{(9)} a_3^{(3)} a_2^{(13)} a_1^{(7)} \\
& + a_4^{(9)} a_3^{(9)} a_2^{(7)} a_1^{(7)} + a_4^{(9)} a_3^{(7)} a_2^{(9)} a_1^{(7)} + a_4^{(9)} a_3^{(7)} a_2^{(5)} a_1^{(11)} + a_4^{(9)} a_3^{(7)} a_2^{(3)} a_1^{(13)} \\
& + a_4^{(9)} a_3^{(5)} a_2^{(11)} a_1^{(7)} + a_4^{(9)} a_3^{(3)} a_2^{(13)} a_1^{(7)} + a_4^{(7)} a_3^{(9)} a_2^{(9)} a_1^{(7)} + a_4^{(7)} a_3^{(9)} a_2^{(5)} a_1^{(11)} \\
& + a_4^{(7)} a_3^{(9)} a_2^{(3)} a_1^{(13)} + a_4^{(7)} a_3^{(9)} a_2^{(9)} a_1^{(7)} + a_4^{(7)} a_3^{(9)} a_2^{(5)} a_1^{(11)} + a_4^{(7)} a_3^{(9)} a_2^{(3)} a_1^{(13)} \\
& + a_4^{(7)} a_3^{(5)} a_2^{(13)} a_1^{(7)} + a_4^{(7)} a_3^{(5)} a_2^{(13)} a_1^{(7)} + a_4^{(5)} a_3^{(13)} a_2^{(7)} a_1^{(7)} + a_4^{(5)} a_3^{(11)} a_2^{(9)} a_1^{(7)} \\
& + a_4^{(5)} a_3^{(11)} a_2^{(5)} a_1^{(11)} + a_4^{(5)} a_3^{(11)} a_2^{(3)} a_1^{(13)} + a_4^{(5)} a_3^{(9)} a_2^{(7)} a_1^{(11)} + a_4^{(5)} a_3^{(7)} a_2^{(7)} a_1^{(13)} \\
& + a_4^{(5)} a_3^{(5)} a_2^{(11)} a_1^{(11)} + a_4^{(5)} a_3^{(3)} a_2^{(13)} a_1^{(11)} + a_4^{(5)} a_3^{(13)} a_2^{(7)} a_1^{(7)} + a_4^{(5)} a_3^{(11)} a_2^{(9)} a_1^{(7)} \\
& + a_4^{(5)} a_3^{(11)} a_2^{(5)} a_1^{(11)} + a_4^{(5)} a_3^{(11)} a_2^{(3)} a_1^{(13)} + a_4^{(5)} a_3^{(9)} a_2^{(7)} a_1^{(11)} + a_4^{(5)} a_3^{(7)} a_2^{(7)} a_1^{(13)} \\
& + a_4^{(5)} a_3^{(5)} a_2^{(11)} a_1^{(11)} + a_4^{(5)} a_3^{(3)} a_2^{(13)} a_1^{(11)} + a_4^{(3)} a_3^{(13)} a_2^{(9)} a_1^{(7)} + a_4^{(3)} a_3^{(13)} a_2^{(5)} a_1^{(11)} \\
& + a_4^{(3)} a_3^{(13)} a_2^{(3)} a_1^{(13)} + a_4^{(3)} a_3^{(13)} a_2^{(9)} a_1^{(7)} + a_4^{(3)} a_3^{(13)} a_2^{(5)} a_1^{(11)} + a_4^{(3)} a_3^{(13)} a_2^{(3)} a_1^{(13)} \\
& + a_4^{(3)} a_3^{(9)} a_2^{(7)} a_1^{(13)} + a_4^{(3)} a_3^{(9)} a_2^{(7)} a_1^{(13)} + a_4^{(3)} a_3^{(5)} a_2^{(11)} a_1^{(13)} + a_4^{(3)} a_3^{(5)} a_2^{(11)} a_1^{(13)} \\
& + a_4^{(3)} a_3^{(3)} a_2^{(13)} a_1^{(13)} + a_4^{(3)} a_3^{(3)} a_2^{(13)} a_1^{(13)} = 0;
\end{aligned}
\]

\[
\begin{aligned}
(x)Sq^2_*=& a_4^{(14)} a_3^{(3)} a_2^{(7)} a_1^{(7)} + a_4^{(14)} a_3^{(3)} a_2^{(7)} a_1^{(7)} + a_4^{(14)} a_3^{(3)} a_2^{(3)} a_1^{(11)} + a_4^{(14)} a_3^{(3)} a_2^{(3)} a_1^{(11)} \\
& + a_4^{(11)} a_3^{(6)} a_2^{(7)} a_1^{(7)} + a_4^{(11)} a_3^{(3)} a_2^{(10)} a_1^{(7)} + a_4^{(11)} a_3^{(3)} a_2^{(6)} a_1^{(11)} + a_4^{(11)} a_3^{(3)} a_2^{(3)} a_1^{(14)} \\
& + a_4^{(11)} a_3^{(6)} a_2^{(7)} a_1^{(7)} + a_4^{(11)} a_3^{(6)} a_2^{(3)} a_1^{(11)} + a_4^{(11)} a_3^{(6)} a_2^{(3)} a_1^{(11)} + a_4^{(11)} a_3^{(3)} a_2^{(10)} a_1^{(7)} \\
& + a_4^{(11)} a_3^{(3)} a_2^{(6)} a_1^{(11)} + a_4^{(11)} a_3^{(3)} a_2^{(3)} a_1^{(14)} + a_4^{(10)} a_3^{(7)} a_2^{(7)} a_1^{(7)} + a_4^{(10)} a_3^{(7)} a_2^{(7)} a_1^{(7)} \\
& + a_4^{(10)} a_3^{(7)} a_2^{(3)} a_1^{(11)} + a_4^{(10)} a_3^{(7)} a_2^{(3)} a_1^{(11)} + a_4^{(10)} a_3^{(3)} a_2^{(11)} a_1^{(7)} + a_4^{(10)} a_3^{(3)} a_2^{(11)} a_1^{(7)} \\
& + a_4^{(7)} a_3^{(10)} a_2^{(7)} a_1^{(7)} + a_4^{(7)} a_3^{(7)} a_2^{(10)} a_1^{(7)} + a_4^{(7)} a_3^{(7)} a_2^{(6)} a_1^{(11)} + a_4^{(7)} a_3^{(7)} a_2^{(3)} a_1^{(14)} \\
& + a_4^{(7)} a_3^{(6)} a_2^{(11)} a_1^{(7)} + a_4^{(7)} a_3^{(3)} a_2^{(14)} a_1^{(7)} + a_4^{(7)} a_3^{(10)} a_2^{(7)} a_1^{(7)} + a_4^{(7)} a_3^{(10)} a_2^{(3)} a_1^{(11)} \\
& + a_4^{(7)} a_3^{(10)} a_2^{(3)} a_1^{(11)} + a_4^{(7)} a_3^{(7)} a_2^{(10)} a_1^{(7)} + a_4^{(7)} a_3^{(7)} a_2^{(6)} a_1^{(11)} + a_4^{(7)} a_3^{(7)} a_2^{(3)} a_1^{(14)} \\
& + a_4^{(7)} a_3^{(6)} a_2^{(11)} a_1^{(7)} + a_4^{(7)} a_3^{(3)} a_2^{(14)} a_1^{(7)} + a_4^{(6)} a_3^{(11)} a_2^{(7)} a_1^{(7)} + a_4^{(6)} a_3^{(11)} a_2^{(7)} a_1^{(7)} \\
& + a_4^{(6)} a_3^{(11)} a_2^{(3)} a_1^{(11)} + a_4^{(6)} a_3^{(11)} a_2^{(3)} a_1^{(11)} + a_4^{(6)} a_3^{(7)} a_2^{(7)} a_1^{(11)} + a_4^{(6)} a_3^{(7)} a_2^{(7)} a_1^{(11)} \\
& + a_4^{(6)} a_3^{(3)} a_2^{(11)} a_1^{(11)} + a_4^{(6)} a_3^{(3)} a_2^{(11)} a_1^{(11)} + a_4^{(3)} a_3^{(14)} a_2^{(7)} a_1^{(7)} + a_4^{(3)} a_3^{(11)} a_2^{(10)} a_1^{(7)} \\
& + a_4^{(3)} a_3^{(11)} a_2^{(6)} a_1^{(11)} + a_4^{(3)} a_3^{(11)} a_2^{(3)} a_1^{(14)} + a_4^{(3)} a_3^{(10)} a_2^{(7)} a_1^{(11)} + a_4^{(3)} a_3^{(7)} a_2^{(7)} a_1^{(14)} \\
& + a_4^{(3)} a_3^{(6)} a_2^{(11)} a_1^{(11)} + a_4^{(3)} a_3^{(3)} a_2^{(14)} a_1^{(11)} + a_4^{(3)} a_3^{(14)} a_2^{(7)} a_1^{(7)} + a_4^{(3)} a_3^{(14)} a_2^{(3)} a_1^{(11)} \\
& + a_4^{(3)} a_3^{(14)} a_2^{(3)} a_1^{(11)} + a_4^{(3)} a_3^{(11)} a_2^{(10)} a_1^{(7)} + a_4^{(3)} a_3^{(11)} a_2^{(6)} a_1^{(11)} + a_4^{(3)} a_3^{(11)} a_2^{(3)} a_1^{(14)} \\
& + a_4^{(3)} a_3^{(10)} a_2^{(7)} a_1^{(11)} + a_4^{(3)} a_3^{(7)} a_2^{(7)} a_1^{(14)} + a_4^{(3)} a_3^{(6)} a_2^{(11)} a_1^{(11)} + a_4^{(3)} a_3^{(3)} a_2^{(11)} a_1^{(14)} \\
& + a_4^{(3)} a_3^{(3)} a_2^{(14)} a_1^{(11)} + a_4^{(3)} a_3^{(3)} a_2^{(11)} a_1^{(14)} = 0;
\end{aligned}
\]

\[
\begin{aligned}
(x)Sq^4_*=& a_4^{(14)} a_3^{(3)} a_2^{(5)} a_1^{(7)} + a_4^{(14)} a_3^{(3)} a_2^{(5)} a_1^{(7)} + a_4^{(13)} a_3^{(3)} a_2^{(6)} a_1^{(7)} + a_4^{(13)} a_3^{(3)} a_2^{(6)} a_1^{(7)} \\
& + a_4^{(11)} a_3^{(6)} a_2^{(5)} a_1^{(7)} + a_4^{(7)} a_3^{(6)} a_2^{(9)} a_1^{(7)} + a_4^{(11)} a_3^{(6)} a_2^{(5)} a_1^{(7)} + a_4^{(7)} a_3^{(6)} a_2^{(5)} a_1^{(11)} \\
& + a_4^{(7)} a_3^{(6)} a_2^{(3)} a_1^{(13)} + a_4^{(11)} a_3^{(5)} a_2^{(6)} a_1^{(7)} + a_4^{(7)} a_3^{(5)} a_2^{(10)} a_1^{(7)} + a_4^{(11)} a_3^{(5)} a_2^{(6)} a_1^{(7)} \\
& + a_4^{(7)} a_3^{(5)} a_2^{(6)} a_1^{(11)} + a_4^{(7)} a_3^{(5)} a_2^{(3)} a_1^{(14)} + a_4^{(10)} a_3^{(5)} a_2^{(7)} a_1^{(7)} + a_4^{(6)} a_3^{(9)} a_2^{(7)} a_1^{(7)} \\
& + a_4^{(10)} a_3^{(7)} a_2^{(5)} a_1^{(7)} + a_4^{(6)} a_3^{(7)} a_2^{(9)} a_1^{(7)} + a_4^{(10)} a_3^{(7)} a_2^{(5)} a_1^{(7)} + a_4^{(6)} a_3^{(7)} a_2^{(5)} a_1^{(11)} \\
& + a_4^{(6)} a_3^{(7)} a_2^{(3)} a_1^{(13)} + a_4^{(10)} a_3^{(5)} a_2^{(7)} a_1^{(7)} + a_4^{(6)} a_3^{(5)} a_2^{(11)} a_1^{(7)} + a_4^{(6)} a_3^{(3)} a_2^{(13)} a_1^{(7)} \\
& + a_4^{(9)} a_3^{(6)} a_2^{(7)} a_1^{(7)} + a_4^{(5)} a_3^{(10)} a_2^{(7)} a_1^{(7)} + a_4^{(9)} a_3^{(7)} a_2^{(6)} a_1^{(7)} + a_4^{(5)} a_3^{(7)} a_2^{(10)} a_1^{(7)} \\
& + a_4^{(9)} a_3^{(7)} a_2^{(6)} a_1^{(7)} + a_4^{(5)} a_3^{(7)} a_2^{(6)} a_1^{(11)} + a_4^{(5)} a_3^{(7)} a_2^{(3)} a_1^{(14)} + a_4^{(9)} a_3^{(6)} a_2^{(7)} a_1^{(7)} \\
& + a_4^{(5)} a_3^{(6)} a_2^{(11)} a_1^{(7)} + a_4^{(5)} a_3^{(3)} a_2^{(14)} a_1^{(7)} + a_4^{(7)} a_3^{(10)} a_2^{(5)} a_1^{(7)} + a_4^{(7)} a_3^{(6)} a_2^{(9)} a_1^{(7)} \\
& + a_4^{(7)} a_3^{(10)} a_2^{(5)} a_1^{(7)} + a_4^{(7)} a_3^{(6)} a_2^{(5)} a_1^{(11)} + a_4^{(7)} a_3^{(6)} a_2^{(3)} a_1^{(13)} + a_4^{(7)} a_3^{(9)} a_2^{(6)} a_1^{(7)} \\
& + a_4^{(7)} a_3^{(5)} a_2^{(10)} a_1^{(7)} + a_4^{(7)} a_3^{(9)} a_2^{(6)} a_1^{(7)} + a_4^{(7)} a_3^{(5)} a_2^{(6)} a_1^{(11)} + a_4^{(7)} a_3^{(5)} a_2^{(3)} a_1^{(14)} \\
& + a_4^{(6)} a_3^{(11)} a_2^{(5)} a_1^{(7)} + a_4^{(6)} a_3^{(7)} a_2^{(9)} a_1^{(7)} + a_4^{(6)} a_3^{(11)} a_2^{(5)} a_1^{(7)} + a_4^{(6)} a_3^{(7)} a_2^{(5)} a_1^{(11)} \\
& + a_4^{(6)} a_3^{(7)} a_2^{(3)} a_1^{(13)} + a_4^{(6)} a_3^{(9)} a_2^{(7)} a_1^{(7)} + a_4^{(6)} a_3^{(5)} a_2^{(7)} a_1^{(11)} + a_4^{(6)} a_3^{(5)} a_2^{(11)} a_1^{(7)} \\
& + a_4^{(6)} a_3^{(5)} a_2^{(7)} a_1^{(11)} + a_4^{(6)} a_3^{(3)} a_2^{(13)} a_1^{(7)} + a_4^{(5)} a_3^{(11)} a_2^{(6)} a_1^{(7)} + a_4^{(5)} a_3^{(7)} a_2^{(10)} a_1^{(7)} \\
& + a_4^{(5)} a_3^{(11)} a_2^{(6)} a_1^{(7)} + a_4^{(5)} a_3^{(7)} a_2^{(6)} a_1^{(11)} + a_4^{(5)} a_3^{(7)} a_2^{(3)} a_1^{(14)} + a_4^{(5)} a_3^{(10)} a_2^{(7)} a_1^{(7)} \\
& + a_4^{(5)} a_3^{(6)} a_2^{(7)} a_1^{(11)} + a_4^{(5)} a_3^{(6)} a_2^{(11)} a_1^{(7)} + a_4^{(5)} a_3^{(6)} a_2^{(7)} a_1^{(11)} + a_4^{(5)} a_3^{(3)} a_2^{(14)} a_1^{(7)} \\
& + a_4^{(3)} a_3^{(14)} a_2^{(5)} a_1^{(7)} + a_4^{(3)} a_3^{(14)} a_2^{(5)} a_1^{(7)} + a_4^{(3)} a_3^{(13)} a_2^{(6)} a_1^{(7)} + a_4^{(3)} a_3^{(13)} a_2^{(6)} a_1^{(7)} \\
& + a_4^{(3)} a_3^{(6)} a_2^{(7)} a_1^{(13)} + a_4^{(3)} a_3^{(5)} a_2^{(7)} a_1^{(14)} + a_4^{(3)} a_3^{(6)} a_2^{(7)} a_1^{(13)} + a_4^{(3)} a_3^{(5)} a_2^{(7)} a_1^{(14)}=0;
\end{aligned}
\]

\[
\begin{aligned}
(x)Sq^8_*=& a_4^{(7)} a_3^{(6)} a_2^{(5)} a_1^{(7)} + a_4^{(7)} a_3^{(6)} a_2^{(5)} a_1^{(7)} + a_4^{(7)} a_3^{(5)} a_2^{(6)} a_1^{(7)} + a_4^{(7)} a_3^{(5)} a_2^{(6)} a_1^{(7)} \\
& + a_4^{(6)} a_3^{(5)} a_2^{(7)} a_1^{(7)} + a_4^{(6)} a_3^{(7)} a_2^{(5)} a_1^{(7)} + a_4^{(6)} a_3^{(7)} a_2^{(5)} a_1^{(7)} + a_4^{(6)} a_3^{(5)} a_2^{(7)} a_1^{(7)} \\
& + a_4^{(5)} a_3^{(6)} a_2^{(7)} a_1^{(7)} + a_4^{(5)} a_3^{(7)} a_2^{(6)} a_1^{(7)} + a_4^{(5)} a_3^{(7)} a_2^{(6)} a_1^{(7)} + a_4^{(5)} a_3^{(6)} a_2^{(7)} a_1^{(7)} \\
& + a_4^{(7)} a_3^{(6)} a_2^{(5)} a_1^{(7)} + a_4^{(7)} a_3^{(6)} a_2^{(5)} a_1^{(7)} + a_4^{(7)} a_3^{(5)} a_2^{(6)} a_1^{(7)} + a_4^{(7)} a_3^{(5)} a_2^{(6)} a_1^{(7)} \\
& + a_4^{(6)} a_3^{(7)} a_2^{(5)} a_1^{(7)} + a_4^{(6)} a_3^{(7)} a_2^{(5)} a_1^{(7)} + a_4^{(6)} a_3^{(5)} a_2^{(7)} a_1^{(7)} + a_4^{(6)} a_3^{(5)} a_2^{(7)} a_1^{(7)} \\
& + a_4^{(5)} a_3^{(7)} a_2^{(6)} a_1^{(7)} + a_4^{(5)} a_3^{(7)} a_2^{(6)} a_1^{(7)} + a_4^{(5)} a_3^{(6)} a_2^{(7)} a_1^{(7)} + a_4^{(5)} a_3^{(6)} a_2^{(7)} a_1^{(7)}=0.
\end{aligned}
\]

Using Corollary~\eqref{mdp} in the case $k =4,$ we derive:
$$ \varphi_4(x) = \lambda^{2}_7 \lambda_5 \lambda_{14} + \lambda^{2}_7 \lambda_9\lambda_{10} + \lambda_7 \lambda_{11}\lambda_9\lambda_6 = y,$$
which implies that $[\varphi_4(x)] = [y] = p_0.$ Thus, $p_0$ is in the image of the Singer transfer $\varphi_4.$  It should be noted that Hung and Quynh's work \cite{HQ} did not provide an explicit description of the preimage of $p_0.$

\section{An algorithmic computation of invariant spaces}\label{s4}
\label{sec:invariant_computation}

\subsection{Motivation: The central role of invariant theory}

We define $\mathcal{P}_k$ to be the polynomial algebra in $k$ variables over the field $\Z/2$:
\[
\mathcal{P}_k := \Z/2[x_1, x_2, \dots, x_k].
\]
This is a graded algebra where each variable $x_i$ is assigned degree one, i.e., $\deg(x_i) = 1$. Topologically, $\mathcal{P}_k$ is isomorphic to the mod-2 cohomology ring of the classifying space of the elementary abelian 2-group of rank $k$, which is the $k$-fold product of infinite-dimensional real projective space. This algebra is endowed with a rich structure as a module over the mod-2 Steenrod algebra, $\A$. The action of the Steenrod squares, $Sq^i \in \A$ for $i \ge 0$, on $\mathcal{P}_k$ is uniquely determined by the following properties:
\begin{enumerate}
    \item[$-$] The action on the generators is given by $Sq^0(x_j) = x_j$, $Sq^1(x_j) = x_j^2$, and $Sq^i(x_j) = 0$ for $i > 1$;
     \item[$-$] The action on products is governed by the Cartan formula: for any two polynomials $f, g \in \mathcal{P}_k$,
    \[
    Sq^n(fg) = \sum_{i=0}^{n} Sq^i(f)Sq^{n-i}(g).
    \]
\end{enumerate}

The algebra $\mathcal{P}_k$ is a canonical example of an \textbf{unstable $\mathcal{A}$-module}. This is characterized by the unstability condition, which states that for any homogeneous element $u \in \mathcal{P}_k$ of degree $d$, the action of $Sq^i$ is zero whenever the degree of the operation exceeds the degree of the element:
\[
Sq^i(u) = 0 \quad \text{for all } i > \deg(u).
\]
Furthermore, as an unstable algebra, it also satisfies the property that $Sq^{\deg(u)}(u) = u^2$ for any $u \in \mathcal{P}_k$. These properties are fundamental to the study of the Peterson hit problem and the behavior of the Singer transfer $\varphi_k$ \cite{Singer1989}.

\medskip

The results presented in Section \ref{s3} rely on a crucial, and computationally profound, step: the determination of the dimension of the space of group invariants. In~\cite{Singer1989}, Singer made a conjecture that \textit{the algebraic transfer is always a monomorphism in all bidegrees \((k, k+*)\)}. This was proven to be true by Singer himself for \(k = 1, 2\), by Boardman~\cite{Boardman} for \(k = 3\), and more recently by us~\cite{Phuc1, Phuc2, Phuc3} for \(k = 4\).  Note that our papers~\cite{Phuc1, Phuc2, Phuc3} have corrected several errors from manual computations found in earlier published versions. These earlier computations were carried out entirely by hand over many years, and as such, mistakes were inevitable due to the overwhelming complexity of manually controlling all intermediate steps. Therefore, the purpose of this section is to present an explicit algorithm implemented in \textsc{SageMath}, which allows readers to verify all the previously hand-calculated results in a transparent and rigorous manner, provided they lie within the memory limits of the computer. Nevertheless, since the dimension of the space \((Q\mathcal P_4)_d\) does not exceed 315 for any positive degree~\(d\) (see~\cite{Sum0}), the implementation of our algorithm for computing the invariant subspace $[(Q\mathcal P_4)_d]^{G_4}$ is entirely feasible and computationally manageable. \textit{It is important to note that there is no such thing as computing an ``\textbf{infinite number of cases}'' when the degree $d$ depends on certain parameters. The reason is that, for instance, $d = 2^{s+3} + 2^{s+1} - 3$ with $s$ a positive integer; consequently, as shown in \cite{Sum0}, for each fixed value of $s$, the admissible monomial basis of degree $d$ (as parameterized by $s$) follows a uniform structure.} Indeed, for $1\leq s\leq 3,$ the monomials involve powers of the variables that are finite in terms of $s.$ For each $s > 3,$ we observe that the final list includes the following three monomials:
\[
\begin{aligned}
x_1^{7}x_2^{2^{s}-5}x_3^{2^{s}-3}x_4^{2^{s+3}-2},\\
x_1^{7}x_2^{2^{s}-5}x_3^{2^{s+1}-3}x_4^{7\cdot 2^{s}-2},\\
x_1^{7}x_2^{2^{s}-5}x_3^{2^{s+3}-3}x_4^{2^{s}-2},
\end{aligned}
\]
and that the exponents depending on $s$ follow a consistent pattern. Therefore, in this case, it is sufficient to work with $s = 4$, corresponding to degree $d = 128$. In summary, to compute the $G_4$-invariant space $[(Q\mathcal P_4)_{d = 2^{s+3} + 2^{s+1} - 3}]^{G_4}$ using our algorithm, it is enough to consider $s\in \{1,\, 2,\, 3,\, 4\},$ which correspond to degrees $d\in \{17,\, 37,\, 77,\, 128\}.$ Once the explicit results are obtained, we can then match the monomial patterns and write down the general form From these, one can extract and confirm the general pattern, as the reader may already be familiar with. Unlike earlier works that relied solely on manual computations, our algorithmic approach ensures reproducibility and minimizes human error.

\medskip

To understand the Singer conjecture for a given rank $k$ and degree $d$, one must ultimately understand the domain of the transfer, which is dually isomorphic to the space of $G_k$-invariants, $[(Q\mathcal{P}_k)_d]^{G_k}$. This space sits at the confluence of representation theory, combinatorics, and algebraic topology, and its computation is of paramount importance for the following several reasons:

$-$ First, the dimension of this invariant space provides a definitive test for the Singer Conjecture in the corresponding bidegree. If $\dim ([(Q\mathcal{P}_k)_d]^{G_k}) = 0$, then the domain of the transfer is the zero vector space, and the transfer map $\varphi_k$ is trivially a monomorphism. In such cases, the conjecture is proven to hold;

$-$ Second, when the dimension is non-zero, it establishes a hard upper bound on the portion of the Ext group that can be "seen" by the algebraic transfer. That is, the dimension of the image of $\varphi_k$ cannot exceed the dimension of its domain. This allows us to make precise statements about which cohomology classes can or cannot be detected by Singer's construction;

$-$ Third, and perhaps most importantly for structural understanding, an explicit basis for the invariant space allows for the construction of concrete elements in the domain of the transfer. 

\medskip

For instance,  we can point to a class like $p_0\in \Ext^{4, 4 + 33}_{\mathcal{A}}(\mathbb Z/2, \mathbb Z/2)$ on the chart and provide a constructive proof of its existence. This involves identifying a specific polynomial invariant $f$ and showing that the Singer transfer $\varphi_k$, when applied to its dual, yields a cocycle $y$ that is a known representative for the class $p_0.$ Based on our algorithm explicitly described below, the invariant polynomial~\(f\) takes the following form. 
\begin{align*}
f= G_4\text{-Invariant}_1 = &x_1 x_2 x_3 x_4^{30} + x_1 x_2 x_3^3 x_4^{28} + x_1 x_2^3 x_3 x_4^{28} + x_1 x_2^3 x_3^4 x_4^{25} + x_1 x_2^7 x_3^{11} x_4^{14} \\
&+ x_1 x_2^7 x_3^{14} x_4^{11} + x_1^3 x_2 x_3 x_4^{28} + x_1^3 x_2 x_3^4 x_4^{25} + x_1^3 x_2^5 x_3 x_4^{24} \\
&+ x_1^3 x_2^5 x_3^{11} x_4^{14} + x_1^3 x_2^5 x_3^{14} x_4^{11} + x_1^7 x_2 x_3^{11} x_4^{14} + x_1^7 x_2 x_3^{14} x_4^{11} \\
&+ x_1^7 x_2^7 x_3^{11} x_4^8 + x_1^7 x_2^7 x_3^8 x_4^{11} + x_1^7 x_2^7 x_3^9 x_4^{10}.
\end{align*}
The algorithm reconstructs, in a fully explicit and verifiable manner, the computation process that was originally performed by hand in our previous work~\cite{Phuc1}, thus producing the invariant polynomial~\(f\) in the form shown above. Accordingly, the dual element corresponding to \([f]\) is explicitly identified in Section~\ref{s3}.

\medskip

The above context provides a constructive method for populating the Adams spectral sequence. Indeed, the Adams spectral sequence can be visualized as a vast, complex chart whose entries on the second page ($E_2$-term) are the $\Ext_{\mathcal{A}}$ groups. For decades, mathematicians have worked to understand this chart: to ``populate'' it by determining which locations contain non-zero elements and what the structure of those elements is. A ``non-constructive'' proof might show that an element must exist at a certain position, but it doesn't provide a formula for it. In contrast, the  approach via the Singer transfer is \emph{constructive}. It ``populates'' the spectral sequence by providing an explicit formula for an element in  $\Ext_{\mathcal{A}}$ at a given bidegree $(k, t)$. 

\medskip

However, the direct computation of these invariants is a formidable task. The traditional approach of first finding the invariants under the symmetric group $\Sigma_k$ and then restricting to the action of the remaining $G_k$ generators becomes computationally explosive. The sheer number of basis elements in $(Q\mathcal{P}_k)_d$ and the complexity of tracking group actions make manual computation unreliable and infeasible for all but the smallest degrees. The combinatorial explosion in the dimensionality of the problem necessitates a systematic and algorithmic approach that can be implemented on a computer.

\subsection{An algorithmic framework for invariant computation}

Before proceeding to construct the algorithm for computing the invariant space \([(Q\mathcal{P}_k)_d]^{G_k}\), We first review the foundational concepts and present several related results along with detailed proofs.

\medskip

\section*{The Weight Vector}

The Peterson hit problem, which seeks a minimal generating set for the polynomial algebra $\PK$ as a module over the Steenrod algebra $\A$, is computationally challenging. The difficulty arises from the large dimension of the quotient space of ``cohits'', $Q\PK = \PK / \Aplus\PK$, for higher ranks and degrees. To make this problem tractable, a ``divide and conquer'' strategy is employed. The core idea is to stratify or ``filter'' the large space $\PK$ into smaller, more manageable layers. The primary tool for this stratification is the \textit{weight vector}, $\omega(x)$, associated with each monomial $x$. The definitions that follow provide the necessary algebraic machinery to formalize this filtration and analyze the structure of the cohit space layer by layer.

\begin{definition}[Weight Vector]
Let the dyadic expansion of a non-negative integer $a$ be given by $a = \sum_{j\geq 0}\alpha_j(a)2^j$, where $\alpha_j(a) \in \{0, 1\}$.

For each monomial $x = x_1^{a_1}x_2^{a_2}\ldots x_k^{a_k}$ in $\PK$, its \textbf{weight vector}, denoted $\omega(x)$, is an infinite sequence of non-negative integers:
$$ \omega(x) = (\omega_1(x), \omega_2(x), \ldots, \omega_j(x), \ldots).$$
where the $j$-th component is the sum of the $(j-1)$-th bits of the exponents:
$$ \omega_j(x) = \sum_{i=1}^{k}\alpha_{j-1}(a_i).$$
\end{definition}

\begin{remark}
The weight vector translates the multiplicative structure of a monomial's exponents into an additive sequence. The lexicographical ordering on these vectors provides a more refined way to compare monomials than simply using their total degree.
\end{remark}

\begin{example}
Consider the monomial $x = x_1^6 x_2^5$ in $\mathcal{P}_2$.
\begin{enumerate}
    \item[$-$] The dyadic expansions of the exponents are:
    \begin{itemize}
        \item $a_1 = 6 = (110)_2 = 0 \cdot 2^0 + 1 \cdot 2^1 + 1 \cdot 2^2$.
        \item $a_2 = 5 = (101)_2 = 1 \cdot 2^0 + 0 \cdot 2^1 + 1 \cdot 2^2$.
    \end{itemize}

    \item[$-$] We compute the components of the weight vector $\omega(x)$:
    \begin{itemize}
        \item $\omega_1(x) = \alpha_0(6) + \alpha_0(5) = 0 + 1 = 1$.
        \item $\omega_2(x) = \alpha_1(6) + \alpha_1(5) = 1 + 0 = 1$.
        \item $\omega_3(x) = \alpha_2(6) + \alpha_2(5) = 1 + 1 = 2$.
        \item For $j > 3$, $\omega_j(x) = 0$.
    \end{itemize}
\end{enumerate}
Thus, the weight vector of $x$ is $\omega(x) = (1, 1, 2, 0, 0, \dots)$.
\end{example}

\section*{The Filtration Subspaces}

\begin{definition}[Subspaces Associated with a Weight Vector]
For a given weight vector $\omega = (\omega_1, \omega_2, \ldots, \omega_s, 0, \dots)$:
\begin{itemize}
    \item The \textbf{degree of $\omega$}, denoted $\deg \omega$, is defined as $\sum_{j\geq 1} 2^{j-1}\omega_j$. This formula reconstructs the total degree of a monomial from its weight vector.
    \item $\mathcal{P}_k(\omega)$ is the subspace of $\PK$ spanned by all monomials $u$ such that $\deg u = \deg \omega$ and $\omega(u) \le \omega$ under the lexicographical ordering.
    \item $\mathcal{P}_k^-(\omega)$ is the subspace of $\mathcal{P}_k(\omega)$ spanned by all monomials $u$ such that $\omega(u) < \omega$.
\end{itemize}
\end{definition}

\begin{remark}
The spaces $\mathcal{P}_k(\omega)$ and $\mathcal{P}_k^-(\omega)$ are the building blocks of the filtration. For an ordered set of weight vectors, $\mathcal{P}_k(\omega)$ represents the collection of all polynomials ``up to level $\omega$'', while $\mathcal{P}_k^-(\omega)$ contains everything ``strictly below level $\omega$''. The ``new'' information at level $\omega$ is captured by the quotient of these two spaces.
\end{remark}

\section*{Equivalence Relations for the Quotient Spaces}

\begin{definition}[Equivalence Relations]
Let $f, g$ be two homogeneous polynomials of the same degree in $\PK$.
\begin{enumerate}
    \item[(i)] \textbf{Hit Equivalence ($\equiv$):} We write $f \equiv g$ if and only if their difference is a "hit" element:
    $$ (f + g) \in \Aplus\PK.$$
    In particular, $f \equiv 0$ means that $f$ is hit. This relation defines the standard cohit space $Q\PK = \PK / \Aplus\PK$.

    \item[(ii)] \textbf{Weighted Equivalence ($\equiv_\omega$):} We write $f \equiv_\omega g$ if and only if their difference is either a hit element or a polynomial of strictly smaller weight:
    $$ (f + g) \in \Aplus\PK + \mathcal{P}_k^-(\omega).$$
\end{enumerate}
\end{definition}

\begin{remark}
The weighted equivalence relation ``$\equiv_\omega$'' is the key technical tool that makes the filtration work. When analyzing the structure at a specific weight level $\omega$, this relation allows us to disregard, or "quotient out," two types of terms:
\begin{itemize}
    \item Elements that are hit in the standard sense ($\Aplus\PK$).
    \item Elements that belong to a lower stratum of the filtration ($\mathcal{P}_k^-(\omega)$).
\end{itemize}
This effectively isolates the algebraic structure that is unique to the weight level $\omega$, which is formalized in the quotient space $Q\mathcal{P}_k(\omega)$.
\end{remark}

Both the relations ``$\equiv$'' and ``$\equiv_{\omega}$'' possess the characteristics of equivalence relations, as is readily observable.  Let $Q\mathcal{P}_k(\omega)$ denote the quotient of $\mathcal{P}_k(\omega)$ by the equivalence relation ``$\equiv_\omega$''. The known results below can be found in the monographs of Walker and Wood~\cite{Walker-Wood, Walker-Wood2}. However, for the sake of completeness and to assist the reader in following the exposition, we provide full proofs of these results here.

\begin{proposition}
We have 
$$ \dim (Q\mathcal{P}_k)_d = \sum_{\deg(\omega) = d}\dim Q\mathcal{P}_k(\omega).$$
\end{proposition}

\begin{proof}
The proof proceeds by constructing a filtration on the vector space $(Q\mathcal{P}_k)_d$ and analyzing the dimensions of its successive quotients.

\paragraph{Step 1: Establishing the Filtration.}
Let the set of all weight vectors $\omega$ of a fixed degree $d$ be ordered lexicographically:
$$ \omega_1 < \omega_2 < \dots < \omega_m.$$
We define a sequence of subspaces of $(Q\mathcal{P}_k)_d$. Let $F_i$ be the image of the polynomial subspace $\mathcal{P}_k(\omega_i)$ in the full quotient space $(Q\mathcal{P}_k)_d$. This construction yields a filtration, which is a sequence of nested vector spaces:
$$ 0 = F_0 \subseteq F_1 \subseteq F_2 \subseteq \dots \subseteq F_m = (Q\mathcal{P}_k)_d.$$

\paragraph{Step 2: Analyzing the Filtration Quotients via Short Exact Sequences.}
For each step in the filtration, the inclusion $F_{i-1} \subseteq F_i$ gives rise to a canonical short exact sequence of vector spaces:
$$ 0 \to F_{i-1} \to F_i \to F_i / F_{i-1} \to 0.$$
A fundamental property of vector spaces is that dimension is additive over short exact sequences. This means:
\begin{equation} \label{eq:dim_add}
\dim(F_i) = \dim(F_{i-1}) + \dim(F_i / F_{i-1}).
\end{equation}
The next step is to identify the quotient space $F_i / F_{i-1}$. By construction, $F_i$ consists of classes of monomials with weight less than or equal to $\omega_i$. Similarly, $F_{i-1}$ consists of classes of monomials with weight strictly less than $\omega_i$. The quotient $F_i / F_{i-1}$ is therefore isomorphic to the space of classes of monomials with weight exactly $\omega_i$, modulo hit elements and terms of smaller weight. This corresponds precisely to the definition of $Q\mathcal{P}_k(\omega_i).$ Thus, we have an isomorphism of vector spaces:
$$ F_i / F_{i-1} \cong Q\mathcal{P}_k(\omega_i).$$
Therefore, their dimensions are equal: $\dim(F_i / F_{i-1}) = \dim(Q\mathcal{P}_k(\omega_i)).$

\paragraph{Step 3: Deriving the Dimension Equality.}
By substituting this dimension equality back into equation (\ref{eq:dim_add}), we obtain a recursive formula for the dimension at each step of the filtration:
$$ \dim(F_i) = \dim(F_{i-1}) + \dim(Q\mathcal{P}_k(\omega_i)).$$
We can now "unroll" this recursion starting from the final step, $i=m$:
\begin{align*}
\dim(F_m) &= \dim(F_{m-1}) + \dim(Q\mathcal{P}_k(\omega_m)) \\
&= \left( \dim(F_{m-2}) + \dim(Q\mathcal{P}_k(\omega_{m-1})) \right) + \dim(Q\mathcal{P}_k(\omega_m)) \\
&= \dim(F_{m-2}) + \dim(Q\mathcal{P}_k(\omega_{m-1})) + \dim(Q\mathcal{P}_k(\omega_m)) \\
&\vdots \\
&= \dim(F_0) + \sum_{j=1}^{m} \dim(Q\mathcal{P}_k(\omega_j)).
\end{align*}
Since the filtration starts with the zero space, $\dim(F_0) = 0$. The final space in the filtration is the entire space of cohits, $F_m = (Q\mathcal{P}_k)_d$. Therefore, we arrive at the desired conclusion:
$$ \dim(Q\mathcal{P}_k)_d = \sum_{j=1}^{m} \dim(Q\mathcal{P}_k(\omega_j)).$$
As the set $\{\omega_1, \dots, \omega_m\}$ is the set of all weight vectors of degree $d$, this is equivalent to:
$$ \dim (Q\mathcal{P}_k)_d = \sum_{\deg(\omega) = d}\dim Q\mathcal{P}_k(\omega).$$
This completes the proof.

\end{proof}

\begin{proposition}
We have
 $$\dim ((Q\mathcal{P}_k)_d)^{G_k}\leq  \sum_{\deg(\omega) = d}\dim (Q\mathcal{P}_k(\omega))^{G_k}.$$
\end{proposition}

\begin{proof}
Let the set of all weight vectors of degree $d$ be ordered lexicographically:
$$ \omega_1 < \omega_2 < \dots < \omega_m.$$
For each $i \in \{1, \dots, m\}$, let $F_i$ be the image of the subspace $\mathcal{P}_k(\omega_i)$ in the quotient space $(Q\mathcal P_k)_d.$ This construction yields a sequence of nested subspaces:
$$ 0 = F_0 \subseteq F_1 \subseteq F_2 \subseteq \dots \subseteq F_m = (Q\mathcal{P}_k)_d.$$
As stated in the referenced literature, the action of $G_k$ respects the ordering of weight vectors. Consequently, each $F_i$ is a $G_k$-submodule of $(Q\mathcal{P}_k)_d,$ and the sequence above is a filtration of $G_k$-modules.

For each step $i$ in the filtration, where $F_{i-1}$ is a submodule of $F_i$, we have a canonical short exact sequence of $G_k$-modules:
\begin{equation} \label{eq:ses}
0 \to F_{i-1} \to F_i \to F_i / F_{i-1} \to 0.
\end{equation}
The quotient module $F_i / F_{i-1}$ represents the ``new information'' added at step $i$. By construction, $F_i$ consists of classes of monomials with weight $\le \omega_i$, while $F_{i-1}$ consists of classes with weight $< \omega_i$. The quotient $F_i / F_{i-1}$ is therefore isomorphic to the space of classes of monomials with weight exactly $\omega_i$, modulo hit elements and terms of strictly smaller weight. This is precisely the definition of $Q\mathcal P_k(\omega_i)$. Thus, we have an isomorphism of $G_k$-modules:
\begin{equation}\label{dc}
 F_i / F_{i-1} \cong Q\mathcal P_k(\omega_i).
\end{equation}
The functor of taking $G_k$-invariants, denoted $(-)^{G_k}$, is left-exact. When applied to the short exact sequence (\ref{eq:ses}), it yields the following exact sequence:
$$ 0 \to (F_{i-1})^{G_k} \to (F_i)^{G_k}\to (F_i / F_{i-1})^{G_k}.$$
From the exactness of this sequence, we can deduce an inequality for the dimensions of these vector spaces:
$$\dim((F_i)^{G_k}) \le \dim((F_{i-1})^{G_k}) + \dim((F_i / F_{i-1})^{G_k}).$$
Using \eqref{dc}, this becomes:
\begin{equation}\label{eq:recursive_ineq}
\dim((F_i)^{G_k}) \le \dim((F_{i-1})^{G_k}) + \dim(Q\mathcal P_k(\omega_i))^{G_k}).
\end{equation}
Let's start with the final step of the filtration, for $i=m$:
$$ \dim((F_m)^{G_k}) \le \dim((F_{m-1})^{G_k}) + \dim((Q\mathcal{P}_k(\omega_m))^{G_k}) $$
Now, we can apply the same inequality (\ref{eq:recursive_ineq}) for the term $\dim((F_{m-1})^{G_k})$:
$$ \dim((F_{m-1})^{G_k}) \le \dim((F_{m-2})^{G_k}) + \dim((Q\mathcal{P}_k(\omega_{m-1}))^{G_k}).$$
Substituting this into the first inequality gives:
\begin{align*}
\dim((F_m)^{G_k}) &\le \left( \dim((F_{m-2})^{G_k}) + \dim((Q\mathcal{P}_k(\omega_{m-1}))^{G_k}) \right) + \dim((Q\mathcal{P}_k(\omega_m))^{G_k}) \\
&= \dim((F_{m-2})^{G_k}) + \dim((Q\mathcal{P}_k(\omega_{m-1}))^{G_k}) + \dim((Q\mathcal{P}_k(\omega_m))^{G_k}).
\end{align*}
If we continue this process of substitution for $\dim((F_{m-2})^{G_k})$, $\dim((F_{m-3})^{G_k})$, and so on, we ``unroll'' the recursion. After unrolling all terms down to $\dim((F_1)^{G_k})$, we obtain:
$$ \dim((F_m)^{G_k}) \le \dim((F_1)^{G_k}) + \sum_{j=2}^{m} \dim((Q\mathcal{P}_k(\omega_j))^{G_k}).$$
Finally, we use the base case of the filtration, where $i=1$. From inequality (\ref{eq:recursive_ineq}), since $F_0 = 0$ and thus $\dim((F_0)^{G_k}) = 0$, we have:
$$ \dim((F_1)^{G_k}) \le \dim((Q\mathcal{P}_k(\omega_1))^{G_k}).$$
Substituting this last piece into our accumulated inequality yields:
\begin{align*}
\dim((F_m)^{G_k}) &\le \dim((Q\mathcal{P}_k(\omega_1))^{G_k}) + \sum_{j=2}^{m} \dim((Q\mathcal{P}_k(\omega_j))^{G_k}) \\
&= \sum_{j=1}^{m} \dim((Q\mathcal{P}_k(\omega_j))^{G_k}).
\end{align*}
Since $F_m = (Q\mathcal{P}_k)_d$ and the set $\{\omega_1, \dots, \omega_m\}$ comprises all weight vectors of degree $d$, we can write the final result as:
$$ \dim((Q\mathcal{P}_k)_d)^{G_k}) \le \sum_{\deg(\omega) = d} \dim((Q\mathcal{P}_k(\omega))^{G_k}).$$
This completes the direct derivation.
\end{proof}

\section*{Motivation: The need for a basis}

Recall that the primary goal of the Peterson hit problem is to determine a vector space basis for the space of cohits, $\QPKd = \PKd / (\Aplus\PK)_d$. The initial space of polynomials, $(\PK)_d$, has a large, easily described basis consisting of all monomials of degree $d$. However, in the quotient space $\QPKd$, many of these monomials become linearly dependent. For instance, if $f \in \Aplus\PK$, then the class $[f]$ is zero, meaning relations like $[x] = [x+f]$ hold.

The challenge is to select a minimal set of monomials whose corresponding classes in $\QPKd$ are linearly independent and span the entire space. The concepts of \textit{admissible} and \textit{inadmissible} monomials provide a systematic criterion for this selection process. The core idea is to define an ordering on the monomials and then determine, one by one, whether a monomial's class is ``new'' or if it can be expressed in terms of classes of ``smaller'' monomials that have already been considered.

{\bf Admissible and Inadmissible monomials:} To formalize this, we first require a total ordering on the set of monomials. We use a lexicographical order on the weight and exponent vectors, denoted by ``$<$''.

\begin{definition}[Inadmissible and Admissible Monomials]
\leavevmode
\begin{itemize}
    \item A monomial $x \in \PK$ is said to be \textbf{inadmissible} if its class $[x]$ in $Q\PK$ can be expressed as a linear combination of classes represented by strictly smaller monomials. That is, if there exist monomials $y_1, y_2, \dots, y_m$ such that $y_j < x$ for all $j$, and
    $$ x \equiv \sum_{j=1}^{m} y_j.$$
    (Recall that $a \equiv b$ means $a+b \in \Aplus\PK$).

    \item A monomial is called \textbf{admissible} if it is not inadmissible.
\end{itemize}
\end{definition}

\begin{remark}[Meaning and Significance]
\leavevmode
\begin{itemize}
    \item \textbf{Inadmissible monomials are ``redundant''.} Their classes in $Q\PK$ do not introduce new, independent directions in the vector space; they can be constructed from elements we have already accounted for (the smaller monomials). Therefore, they are excluded from our basis.
    \item \textbf{Admissible monomials are ``essential''.} Their classes are linearly independent of the classes of all smaller monomials. They represent the new, fundamental components of the quotient space at that point in the ordering.
\end{itemize}
This leads to the central result that motivates these definitions:
\end{remark}

\begin{proposition}
The set of all classes represented by admissible monomials of degree $d$ forms a vector space basis for $(\QPKd)$.
\end{proposition}

\begin{proof}
To prove that the set of classes of admissible monomials, let's call it $\mathcal{B}_{adm}$, forms a basis for the vector space $\QPKd$, we must demonstrate two properties: that $\mathcal{B}_{adm}$ spans $\QPKd$, and that it is linearly independent.

\subsection*{(A) Spanning Property}

We want to show that any class $[f] \in \QPKd$ can be written as a linear combination of classes of admissible monomials. We use an argument based on the well-ordering of monomials.

Let $f \in (\PK)_d$ be an arbitrary homogeneous polynomial of degree $d$. We can write $f$ as a sum of its constituent monomials. Let $x_{max}$ be the largest monomial appearing in $f$ with a non-zero coefficient, according to the predefined monomial ordering ``$<$''. We can write $f = c \cdot x_{max} + f'$, where $c=1$ and all monomials in $f'$ are strictly smaller than $x_{max}$.

We have two cases for the monomial $x_{max}$:

\begin{itemize}
    \item \textbf{Case 1: $x_{max}$ is admissible.}
    In this case, $[f] = [x_{max}] + [f']$. The first term, $[x_{max}]$, is already a class of an admissible monomial. The remainder, $f'$, is a polynomial whose largest monomial is smaller than $x_{max}$. We can repeat this process on $f'$. Since the set of monomials of a given degree is finite, this process must terminate. Each step replaces a polynomial with a sum involving an admissible class and a ``smaller'' polynomial. Eventually, we express $[f]$ entirely as a sum of classes of admissible monomials.

\item \textbf{Case 2: $x_{max}$ is inadmissible.}
    By the definition of an inadmissible monomial, there exist monomials $y_1, \dots, y_m$ such that $y_j < x_{max}$ for all $j$, and the following congruence holds:
    $$ x_{max} \equiv \sum_{j=1}^{m} y_j.$$
    The relation `$\equiv$` signifies equivalence modulo the subspace of hit elements, $\Aplus\PK$. By definition, the congruence above means that the difference between the two sides is a hit element:
    $$ x_{max} + \sum_{j=1}^{m} y_j \in \Aplus\PK.$$
    This implies that there exists some polynomial $h$ such that $h \in \Aplus\PK$ and
    $$ x_{max} = \sum_{j=1}^{m} y_j + h.$$
    We can now substitute this expression for $x_{max}$ back into our original equation for $f$, which was $f = x_{max} + f'$.
    $$ f = \left(\sum_{j=1}^{m} y_j + h\right) + f'.$$
    Now consider the class of $f$ in $\QPKd$. Since $h \in \Aplus\PK$, its class $[h]$ is zero.
    $$ [f] = \left[ \left(\sum_{j=1}^{m} y_j\right) + f' \right].$$
    The new polynomial $f_{new} = (\sum y_j) + f'$ has a largest monomial that is strictly smaller than $x_{max}$. We have successfully replaced the polynomial $f$ with an equivalent polynomial $f_{new}$ whose largest monomial is smaller. We can now repeat the argument on $f_{new}$. This reduction process must terminate, as there is a finite number of monomials of a given degree.
\end{itemize}
In both cases, we can reduce any polynomial class $[f]$ to a linear combination of admissible monomial classes. Therefore, the set $\mathcal{B}_{adm}$ spans $\QPKd$.

\subsection*{(B) Linear Independence}
We now show that the set of classes of admissible monomials is linearly independent. We use a proof by contradiction.

Assume the set is linearly dependent. This means there exists a non-trivial linear combination of distinct admissible monomial classes that equals zero in $\QPKd$:
$$ \sum_{i=1}^N c_i [x_i] = [0].$$
where each $x_i$ is an admissible monomial, $c_i \in \{0, 1\}$, and at least one $c_i$ is non-zero.

By the definition of the quotient space, this equation means:
$$ \sum_{i=1}^N c_i x_i \in \Aplus\PK.$$
Let $x_{max}$ be the largest monomial among $\{x_1, \dots, x_N\}$ that has a non-zero coefficient, say $c_{max}=1$. We can isolate this term:
$$ x_{max} + \sum_{x_i < x_{max}} c_i x_i \in \Aplus\PK.$$
This is equivalent to the congruence relation:
$$ x_{max} \equiv \sum_{x_i < x_{max}} c_i x_i.$$
This equation shows that the monomial $x_{max}$ is equivalent to a linear combination of monomials that are all strictly smaller than it. By definition, this means that $x_{max}$ is an inadmissible monomial. This is a contradiction, as we started with the assumption that all the $x_i$, including $x_{max}$, were admissible. Therefore, our initial assumption of a non-trivial linear relation must be false. All coefficients $c_i$ must be zero. Thus, the set of classes of admissible monomials is linearly independent.

\subsection*{(C) Conclusion}
Since the set of classes of admissible monomials of degree $d$ both spans $\QPKd$ and is linearly independent, it forms a basis for the vector space $\QPKd$.
\end{proof}

{\bf The technical condition of Strict Inadmissibility:} While the definition of inadmissibility is conceptually clear, it can be difficult to prove directly that a relation $x \equiv \sum y_j$ exists. The notion of ``strict inadmissibility'' provides a more concrete, verifiable condition that implies inadmissibility.

\begin{definition}[Strictly Inadmissible monomial]
A monomial $x \in \PK$ is defined as \textbf{strictly inadmissible} if there exist monomials $y_1, y_2, \dots, y_m$ of the same degree as $x$ with $y_j < x$ for all $j$, such that $x$ can be expressed in the form:
$$ x = \sum_{j=1}^{m} y_j + \sum_{\ell=1}^{2^r-1} Sq^{\ell}(h_\ell).$$
where $r = \max\{i \in \Z : \omega_i(x) > 0\}$ and $h_\ell$ are appropriate polynomials in $\PK$.
\end{definition}


\begin{proposition}
We consider:
$$\begin{array}{ll}
\medskip
\mathcal {P}_k^0 &=\big\langle \big\{x_1^{a_1}x_2^{a_2}\ldots x_k^{a_k}\in \mathcal {P}_k\;|\;a_1a_2\ldots a_k = 0\big\}\big\rangle, \\
\mathcal {P}_k^+&= \big\langle \big\{x_1^{a_1}x_2^{a_2}\ldots x_k^{a_k}\in \mathcal {P}_k\;|\; a_1a_2\ldots a_k > 0\big\}\big\rangle.
\end{array}$$
Consequently, $\mathcal {P}_k^0$ and $\mathcal {P}_k^+$ are $\A$-submodules of $\PK,$ and $$(Q\mathcal {P}_k)_d \cong (Q\mathcal {P}_k^0)_{d}\oplus (Q\mathcal {P}_k^+)_{d}.$$
\end{proposition}

\begin{proof}
First, we establish the decomposition at the level of polynomial vector spaces. Second, we show that this decomposition is respected by the action of the Steenrod algebra $\A.$. Finally, we show that this module decomposition passes to the quotient spaces.

\subsection*{Step 1: Decomposition as Vector Spaces}

The polynomial algebra $\PK = \Ztwo[x_1, \dots, x_k]$ has a canonical basis consisting of all monomials $x^I = x_1^{a_1} \dots x_k^{a_k}$. This set of basis monomials can be partitioned into two disjoint subsets:
\begin{enumerate}
    \item[$\bullet$] The set of monomials where at least one exponent $a_i$ is zero ($a_1 a_2 \dots a_k = 0$). This set forms a basis for the subspace $\PKzero$.
    \item[$\bullet$] The set of monomials where all exponents $a_i$ are strictly positive ($a_1 a_2 \dots a_k > 0$). This set forms a basis for the subspace $\PKplus$.
\end{enumerate}
Since these two subspaces are spanned by disjoint subsets of a basis for the entire space $\PK$, their sum is a direct sum. This decomposition respects the grading by degree, so for any degree $d$, we have a direct sum of vector spaces:
$$ \PKd = (\PKzero)_d \oplus (\PKplus)_d.$$

\subsection*{Step 2: Proving $\PKzero$ and $\PKplus$ are $\mathcal{A}$-submodules}

The crucial step is to show that the action of the Steenrod algebra preserves this decomposition. We need to show that if $f \in \PKzero$, then $Sq^i(f) \in \PKzero$ for any $i \ge 0$, and similarly for $\PKplus$. It is sufficient to check this for the basis monomials of each subspace.

\paragraph{\textbf{Invariance of $\mathcal{P}_k^0$:}} Let $m = x_1^{a_1} \dots x_k^{a_k}$ be a monomial in $\PKzero$. By definition, there exists at least one index $j \in \{1, \dots, k\}$ such that $a_j = 0$. Using the Cartan formula, the action of $Sq^i$ on $m$ is:
$$ Sq^i(m) = \sum_{i_1+\dots+i_k = i} Sq^{i_1}(x_1^{a_1}) \dots Sq^{i_j}(x_j^0) \dots Sq^{i_k}(x_k^{a_k}).$$
The term $Sq^{i_j}(x_j^0) = Sq^{i_j}(1)$ is non-zero only if $i_j=0$, in which case $Sq^0(1) = 1$. Therefore, any non-zero term in the sum must have the factor $Sq^0(1) = 1$ for the $j$-th variable. The resulting monomial in this term will not contain the variable $x_j$ (its exponent will remain zero). Since every monomial in the expansion of $Sq^i(m)$ is missing the variable $x_j$, the entire polynomial $Sq^i(m)$ belongs to $\PKzero$. Thus, $\PKzero$ is an $\mathcal{A}$-submodule of $\PK$.

\paragraph{\textbf{Invariance of $\mathcal{P}_k^+$:}} Let $m = x_1^{a_1} \dots x_k^{a_k}$ be a monomial in $\PKplus$. By definition, $a_j > 0$ for all $j \in \{1, \dots, k\}$. The action of $Sq^i$ is:
$$ Sq^i(m) = \sum_{i_1+\dots+i_k = i} Sq^{i_1}(x_1^{a_1}) \dots Sq^{i_k}(x_k^{a_k}).$$
For any non-zero term in this sum, each factor $Sq^{i_j}(x_j^{a_j})$ is a polynomial in the variable $x_j$. Since $a_j > 0$, any non-zero term in the expansion of $Sq^{i_j}(x_j^{a_j})$ will be a sum of positive powers of $x_j$. Consequently, every monomial in the final expression for $Sq^i(m)$ will contain only positive powers of all variables $x_1, \dots, x_k$. Therefore, $Sq^i(m) \in \PKplus$, and $\PKplus$ is an $\mathcal{A}$-submodule of $\PK$.

\subsection*{Step 3: Decomposition of the quotient space}

We have established a direct sum decomposition of $\mathcal{A}$-modules:
$$ \PK = \PKzero \oplus \PKplus $$
Let $\Aplus$ be the augmentation ideal of $\mathcal{A}$. The image of $\PK$ under the action of $\Aplus$ also decomposes accordingly:
$$ \Aplus\PK = \Aplus(\PKzero \oplus \PKplus) = (\Aplus\PKzero) \oplus (\Aplus\PKplus).$$
This holds because $\PKzero$ and $\PKplus$ are $\mathcal{A}$-submodules.

Now we consider the quotient space $\QPK = \PK / \Aplus\PK$ and obtain:
\begin{align*}
\QPK &= \PK / \Aplus\PK \\
&\cong (\PKzero \oplus \PKplus) / (\Aplus\PKzero \oplus \Aplus\PKplus) \\
&\cong (\PKzero / \Aplus\PKzero) \oplus (\PKplus / \Aplus\PKplus) \\
&\cong \QPKzero \oplus \QPKplus.
\end{align*}
Since this decomposition holds as graded modules, it must hold in each degree $d$. 
\end{proof}

\medskip

\section*{$\SigmaK$-invariants and $G_k$-invariants}

To study the space of $G_k$-invariants, we first need a concrete way to represent the action of $G_k$ on the polynomial algebra $\PK = \mathbb{Z}/2[x_1, \dots, x_k]$. It is a standard result in group theory that any group action can be understood by studying the action of its generators. The following operators, $\rho_j$, form a well-known generating set for $G_k$.

\begin{definition}[The Generating Operators]
For $1 \le j \le k$, we define the $\mathcal{A}$-homomorphism $\rho_j: \PK \to \PK$ by its action on the variables $\{x_1, \dots, x_k\}$. The definition is split into two cases.

\begin{enumerate}
    \item \textbf{Adjacent Transpositions ($1 \le j \le k-1$):}
    The operator $\rho_j$ swaps the adjacent variables $x_j$ and $x_{j+1}$ and fixes all others:
    $$ \rho_j(x_i) = 
    \begin{cases} 
        x_{j+1} & \text{if } i = j \\
        x_j & \text{if } i = j+1 \\
        x_i & \text{otherwise}
    \end{cases}
    $$
    
    \item \textbf{A Transvection ($j=k$):}
    The operator $\rho_k$ adds the variable $x_{k-1}$ to $x_k$ and fixes all others:
    $$ \rho_k(x_i) = 
    \begin{cases} 
        x_k + x_{k-1} & \text{if } i = k \\
        x_i & \text{if } i < k
    \end{cases}
    $$
\end{enumerate}
The action of any $\rho_j$ is extended to all polynomials in $\PK$ by the property that it is an algebra homomorphism.
\end{definition}


\begin{remark}[Algebraic Significance]
The choice of these specific operators is motivated by the structure of the general linear group over $\Ztwo$.

$\bullet$ An element of $\Sigma_k$ is a permutation of the basis vectors $\{x_1, \dots, x_k\}$. The operator $\rho_j$ for $j < k$ is defined as $\rho_j(x_j) = x_{j+1}$, $\rho_j(x_{j+1}) = x_j$, while fixing all other basis vectors. This corresponds exactly to the adjacent transposition $(j, j+1)$. Indeed, we provide an explicit description as follows:

\begin{enumerate}
    \item[(i)] \textbf{Any permutation is a product of transpositions.}
     Any permutation can be decomposed into a product of disjoint cycles. Each cycle $(c_1, c_2, \dots, c_m)$ can, in turn, be written as a product of transpositions:
    \[ (c_1, c_2, \dots, c_m) = (c_1, c_m) \circ \dots \circ (c_1, c_3) \circ (c_1, c_2).\]
    Therefore, to prove the generating set, we only need to show that any arbitrary transposition $(i, j)$ can be generated from adjacent transpositions.

    \item[(ii)] \textbf{Any transposition $(i, j)$ is a product of adjacent transpositions.}
    Assume $i < j$. We want to construct the transposition $(i, j)$, which swaps the elements at positions $i$ and $j$. This can be done constructively. A well-known formula is:
    \[ (i, j) = (\rho_{i} \circ \dots \circ \rho_{j-2}) \circ (\rho_{j-1}) \circ (\rho_{j-2} \circ \dots \circ \rho_{i}).\]
    This formula first moves the element at position $i$ to position $j$ via a series of rightward adjacent swaps, and then moves the element that was originally at position $j$ (which has been shifted) back to position $i$.
    
    \textbf{Example:} To create the transposition $(1, 3)$ in $\Sigma_3$, we can perform the sequence $\rho_1 \circ \rho_2 \circ \rho_1$.
\end{enumerate}
Since every permutation is a product of transpositions, and every transposition is a product of adjacent transpositions (the operators $\rho_j$ for $j < k$), \textit{the set $\{\rho_1, \dots, \rho_{k-1}\}$ generates the entire symmetric group $\Sigma_k$.}

\medskip

$\bullet$ \textit{The general linear group $G_k$ is generated by the set of operators $\{\rho_j \mid 1 \leq j \leq k\}$.} Indeed, it is known that any invertible matrix can be reduced to the identity matrix $I$ through a sequence of elementary column operations. Conversely, any invertible matrix can be formed by applying a sequence of elementary operations to the identity matrix. Over the field $\mathbb Z/2,$ these operations are:
\begin{itemize}
    \item \textbf{Type I:} Swapping two columns, $C_i \leftrightarrow C_j$.
    \item \textbf{Type III:} Adding one column to another, $C_i \to C_i + C_j$.
\end{itemize}
(Type II operations, scaling a column by a non-zero scalar, are trivial in $\mathbb{F}_2$ as 1 is the only non-zero scalar). We will now show that our set of generators can produce all of these elementary operations.

\begin{enumerate}
    \item[(i)] \textbf{Generating Type I Operations (Column Swaps):}
    As established in the previous section, the set of operators $\{\rho_1, \dots, \rho_{k-1}\}$ generates the entire symmetric group $\Sigma_k$. Each permutation corresponds to a permutation matrix, which performs the desired column swap operation. Thus, all Type I operations can be generated.

    \item[(ii)] \textbf{Generating Type III Operations (Column Additions):}
    The operator $\rho_k$ is defined by $\rho_k(x_k) = x_k + x_{k-1}$ and fixes other basis vectors. Its corresponding matrix $E$ performs the column operation $C_k \to C_k + C_{k-1}$.
    \[ E = \begin{pmatrix}
    1 & \cdots & 0 & 0 \\
    \vdots & \ddots & \vdots & \vdots \\
    0 & \cdots & 1 & 1 \\
    0 & \cdots & 0 & 1
    \end{pmatrix} \]
    This matrix $E$ is an elementary matrix known as a transvection. To generate an arbitrary transvection $E_{ij}$ that performs the operation $C_i \to C_i + C_j$, we use matrix conjugation.
    
    Let $P$ be a permutation matrix. The matrix $P E P^{-1}$ performs the same type of operation as $E$, but on the basis vectors permuted by $P$. To generate $E_{ij}$ (adding column $j$ to column $i$), we can choose a permutation matrix $P$ that maps the $k$-th basis vector to the $i$-th, and the $(k-1)$-th basis vector to the $j$-th. Such a permutation $P$ can be constructed from the generators $\{\rho_1, \dots, \rho_{k-1}\}$. Then, the conjugation
    \[ E_{ij} = P \circ E \circ P^{-1} \]
    produces the desired elementary matrix. Since we can generate any permutation matrix $P$, we can generate any elementary transvection $E_{ij}$.
\end{enumerate}
\end{remark}

\medskip

{\bf  The condition for invariance:} Using the set of generators described above, we can now state the precise condition for a class in the quotient space to be invariant. Recall that for a polynomial $u \in \PK(\omega)$, we write $[u]_\omega$ for its corresponding class in $\QPKomega$.
Because it is sufficient to check for invariance under a set of generators, we have the following precise criteria.

\begin{proposition}[Invariance Conditions]
Let $[u]_\omega$ be a class in $\QPKomega$ represented by a homogeneous polynomial $u \in \PK(\omega)$.
\begin{enumerate}
    \item The class $[u]_\omega$ is $\SigmaK$-invariant if and only if it is invariant under the action of all adjacent transpositions:
    $$ \rho_j(u) + u \equiv_\omega 0 \quad \text{for all } j \in \{1, \dots, k-1\}. $$

    \item The class $[u]_\omega$ is $\GLK$-invariant if and only if it is $\SigmaK$-invariant and is also invariant under the action of the transvection $\rho_k$. This is equivalent to the single, comprehensive condition:
    $$ \rho_j(u) + u \equiv_\omega 0 \quad \text{for all } j \in \{1, \dots, k\}. $$
\end{enumerate}
\end{proposition}

\begin{proof}
The proof is identical to the one above. The set $\{\rho_1, \dots, \rho_k\}$ is a generating set for $\GLK$.
\begin{itemize}
    \item The ($\Rightarrow$) direction is trivial by definition.
    \item The ($\Leftarrow$) direction follows from the same inductive argument: if a class is invariant under all generators $\rho_1, \dots, \rho_k$, then it is invariant under any element $g \in \GLK$ since $g$ can be written as a finite composition of these generators.
\end{itemize}
This concludes the proof.
\end{proof}

\section*{Duality in the context of the Singer transfer}

The relationship between the space of $G_k$-invariants $[(Q\mathcal{P}_k)_d]^{G_k}$ and the space of $G_k$-coinvariants $[P_{\mathcal{A}}H_d(B(\Ztwo^k))]_{\GLK}$ is one of duality. This means they are vector spaces of the same dimension and there exists a non-degenerate pairing between them. Understanding this pairing is fundamental to connecting the hit problem to the domain of the Singer transfer. 

The starting point is the canonical duality between the homology and cohomology of the classifying space $B(\Ztwo^k)$.
\begin{itemize}
    \item The cohomology, $\mathcal{P}_k = H^*(B(\Ztwo^k); \Ztwo)$, is the polynomial algebra $\Ztwo[x_1, \dots, x_k]$.
    \item The homology, $H_*:=H_*(B(\Ztwo^k); \Ztwo)$, is the divided power algebra $\Gamma[a_1, \dots, a_k]$.
\end{itemize}
These two graded vector spaces are dual to each other. This foundational duality descends through the various algebraic structures imposed on these spaces.

{\bf Primitives and Cohits:} The hit problem involves the quotient space of ``cohit'', $Q\mathcal{P}_k = \mathcal{P}_k / \Aplus\mathcal{P}_k$. The dual concept in homology is the subspace of ``primitives'', $P_{\mathcal{A}}H_* \subseteq H_*$, which consists of all elements annihilated by the positive-degree Steenrod operations.

The subspace of primitives $P_{\mathcal{A}}H_d$ is precisely the annihilator of the subspace of hit elements $\Aplus(\mathcal{P}_k)_d$ under the pairing. By a standard result in linear algebra, the dual of a quotient space is the annihilator of the subspace by which we are quotienting. Therefore, we have an isomorphism:
$$ ((Q\mathcal{P}_k)_d)^* \cong P_{\mathcal{A}}H_d(B(\Ztwo)^k).$$
This means that any element $[f] \in (Q\mathcal{P}_k)_d$ can be seen as a linear functional acting on the space of primitives.

{\bf Invariants and Coinvariants:} The group $\GLK$ acts on both $\mathcal{P}_k$ and $H_*$, and these actions are dual to each other with respect to the pairing.
\begin{itemize}
    \item The space of invariants, $[(Q\mathcal{P}_k)_d]^{\GLK}$, is the subspace of cohit classes fixed by every element of $\GLK$.
    \item The space of coinvariants, $[P_{\mathcal{A}}H_d]_{\GLK}$, is the quotient of the space of primitives by the action of the group, i.e., $P_{\mathcal{A}}H_d / \text{span}\{g \cdot z - z \mid g \in \GLK, z \in P_{\mathcal{A}}H_d\}$.
\end{itemize}
It is a fundamental result in representation theory that for a finite group acting on a finite-dimensional vector space $V$ over a field whose characteristic does not divide the order of the group, the space of invariants $(V^G)$ is dual to the space of coinvariants $(V_G)$. In our context, this means:
$$ \left( [(Q\mathcal{P}_k)_d]^{\GLK} \right)^* \cong [P_{\mathcal{A}}H_d(B(\Ztwo^k))]_{\GLK}.$$

{\bf The explicit pairing $\langle [g], [f] \rangle$}:

\begin{definition}[The Dual Pairing]\label{dncdn}
Let $[f]$ be the class of a polynomial $f \in (\mathcal{P}_k)_d$ that represents a $\GLK$-invariant cohit. Let $[g]$ be the class of a polynomial $g \in H_d$ that represents a $\GLK$-coinvariant primitive element. The pairing $\langle \cdot, \cdot \rangle: H_*\times \mathcal P_k\longrightarrow \mathbb Z/2$ is defined as:
$$ \langle [g], [f] \rangle := \langle g, f \rangle.$$
where the pairing on the right is the canonical pairing between homology and cohomology.
\end{definition}

We must check that this definition makes sense and does not depend on the choice of representatives $f$ and $g$.
\begin{itemize}
    \item If we choose a different representative for $[f]$, say $f' = f + h$ where $h \in \Aplus\mathcal{P}_k$, then $\langle g, f' \rangle = \langle g, f \rangle + \langle g, h \rangle$. Since $g$ is a primitive element (it is annihilated by $\Aplus$), $\langle g, h \rangle = 0$. So the value is unchanged.
    \item If we choose a different representative for $[g]$, say $g' = g + (z - \sigma(z))$ for some $\sigma \in \GLK$, then $\langle g', f \rangle = \langle g, f \rangle + \langle z, f \rangle - \langle \sigma(z), f \rangle$. Since $f$ is a $\GLK$-invariant element, $\langle \sigma(z), f \rangle = \langle z, \sigma^{-1}(f) \rangle = \langle z, f \rangle$. Therefore, the extra terms cancel out.
\end{itemize}
Thus, the pairing is well-defined. 



The canonical pairing $\langle \cdot, \cdot \rangle$ between the homology $H_*$ and the cohomology $\PK$ is initially defined on their respective monomial bases. We can extend this definition to arbitrary polynomials by enforcing the property of \textbf{bilinearity}.

\begin{definition}[General Dual Pairing]
Let $u \in H_*$ and $v \in \PK$ be two homogeneous polynomials of the same degree. We can express them as linear combinations of their respective basis monomials:
\begin{itemize}
    \item Homology element: $u = \sum_{I} c_I \, a^I$
    \item Cohomology element: $v = \sum_{J} d_J \, x^J$
\end{itemize}
where $I = (i_1, \dots, i_k)$ and $J = (j_1, \dots, j_k)$ are multi-indices representing the exponents, $a^I = a_1^{(i_1)}\dots a_k^{(i_k)}$, $x^J = x_1^{j_1}\dots x_k^{j_k}$, and the coefficients $c_I, d_J$ are in $\Ztwo$.

The pairing $\langle u, v \rangle$ is defined by extending the basis pairing linearly:
$$ \langle u, v \rangle = \left\langle \sum_{I} c_I a^I, \sum_{J} d_J x^J \right\rangle := \sum_{I, J} c_I d_J \langle a^I, x^J \rangle.$$
\end{definition}

Due to the orthogonality of the monomial bases, the term $\langle a^I, x^J \rangle$ is 1 if and only if the multi-indices are identical ($I=J$), and 0 otherwise. This simplifies the double summation significantly. Most terms in the expansion will vanish. The only terms that survive are those where the homology basis monomial $a^I$ is paired with its exact dual cohomology basis monomial $x^I$.

This leads to the following general formula.

\begin{proposition}[Condition for the Pairing]\label{mddn}
Let $u = \sum_{I} c_I a^I$ and $v = \sum_{I} d_I x^I$ be polynomials in homology and cohomology, respectively, written in their dual monomial bases. The value of their pairing is the sum over all possible multi-indices $I$ of the product of their corresponding coefficients:
$$ \langle u, v \rangle = \sum_{I} c_I d_I \quad (\text{mod } 2).$$
Consequently, the pairing is equal to 1 if and only if there is an odd number of multi-indices $I$ for which both polynomials have a non-zero coefficient for the corresponding basis element ($c_I=1$ and $d_I=1$).
\end{proposition}

\begin{proof}
The proof relies on two fundamental properties: the bilinearity of the pairing and the orthogonality of the chosen monomial bases.

{\bf Step 1: Expanding the Pairing by Bilinearity.} The pairing $\langle \cdot, \cdot \rangle: H_* \times \PK \to \Ztwo$ is a bilinear map. This means it is linear in each of its two arguments. Given the expressions for $u$ and $v$:
$$ u = \sum_{I} c_I a^I \quad \text{and} \quad v = \sum_{J} d_J x^J.$$
we can expand the pairing $\langle u, v \rangle$ as follows. First, using linearity in the first argument:
$$ \langle u, v \rangle = \left\langle \sum_{I} c_I a^I, v \right\rangle = \sum_{I} c_I \langle a^I, v \rangle.$$
Next, we apply linearity in the second argument to each term $\langle a^I, v \rangle$:
$$ \langle a^I, v \rangle = \left\langle a^I, \sum_{J} d_J x^J \right\rangle = \sum_{J} d_J \langle a^I, x^J \rangle.$$
Substituting this back into the first expansion, we obtain a double summation:
$$ \langle u, v \rangle = \sum_{I} c_I \left( \sum_{J} d_J \langle a^I, x^J \rangle \right) = \sum_{I, J} c_I d_J \langle a^I, x^J \rangle.$$

{\bf Step 2: Applying the Orthogonality of the Monomial Bases.} The monomial bases $\{a^I\}$ for homology and $\{x^J\}$ for cohomology are dual to each other. This orthogonality is expressed by the value of their pairing:
$$ \langle a^I, x^J \rangle = \delta_{I,J} =
\begin{cases} 
    1 & \text{if } I = J \\
    0 & \text{if } I \neq J.
\end{cases} $$
where $\delta_{I,J}$ is the Kronecker delta.

Now, we substitute this orthogonality condition into our double summation. The term $\langle a^I, x^J \rangle$ will be zero for every pair of multi-indices where $I \neq J$. The only terms that survive the summation are those for which the multi-indices are identical, i.e., $I = J$.
$$ \langle u, v \rangle = \sum_{I} \sum_{J} c_I d_J \delta_{I,J}.$$
This reduces the double summation to a single summation over one index, say $I$:
$$ \langle u, v \rangle = \sum_{I} c_I d_I \langle a^I, x^I \rangle = \sum_{I} c_I d_I (1) = \sum_{I} c_I d_I.$$

{\bf Step 3: Interpretation of the Result.}
We have derived the formula:
$$ \langle u, v \rangle = \sum_{I} c_I d_I \quad (\text{mod } 2).$$
The coefficients $c_I$ and $d_I$ belong to the field $\Ztwo$, which has only two elements, $\{0, 1\}$. Therefore, the product $c_I d_I$ can only be non-zero if both $c_I = 1$ and $d_I = 1$.
$$ c_I d_I = 
\begin{cases} 
    1 & \text{if } c_I=1 \text{ and } d_I=1 \\
    0 & \text{otherwise}.
\end{cases} $$
The final sum $\sum_{I} c_I d_I$ is performed modulo 2. A sum of 0s and 1s is equal to 1 (mod 2) if and only if it contains an odd number of 1s. Therefore, $\langle u, v \rangle = 1$ if and only if there is an odd number of multi-indices $I$ for which both the coefficient of $a^I$ in $u$ and the coefficient of its dual $x^I$ in $v$ are 1. This concludes the proof of the proposition.
\end{proof}

We are asked to compute the value of the dual pairing $\langle x, f \rangle$, where $x$ is an element in homology and $f$ is an element in cohomology. The calculation is performed over the field $\Ztwo$.

The general principle for this pairing is:
$$ \langle x, f \rangle = \left\langle \sum_{I} c_I a^I, \sum_{J} d_J x^J \right\rangle = \sum_{I} c_I d_I \pmod{2}.$$
This means the pairing evaluates to 1 if and only if there is an odd number of monomial terms that are dual to each other and appear in both $x$ and $f$. We will systematically check each monomial term in $f$ to see if its dual exists in $x$.

\medskip

{\bf Example:} We now reconsider the homology element $x\in P_{\mathcal A}H_{33}(B(\mathbb Z/2)^4),$ which is the preimage of $p_0$ as given in Section~\ref{s3}, but expressed using the variable order $(a_1, a_2, a_3, a_4)$. Accordingly, the exponents of the variables $a_i$ will follow this same order:

\medskip

\[
\begin{aligned}
x =\ & a_1^{(7)} a_2^{(7)} a_3^{(5)} a_4^{(14)} + a_1^{(7)} a_2^{(9)} a_3^{(3)} a_4^{(14)} + a_1^{(11)} a_2^{(5)} a_3^{(3)} a_4^{(14)} + a_1^{(13)} a_2^{(3)} a_3^{(3)} a_4^{(14)} \\
& + a_1^{(7)} a_2^{(7)} a_3^{(6)} a_4^{(13)} + a_1^{(7)} a_2^{(10)} a_3^{(3)} a_4^{(13)} + a_1^{(11)} a_2^{(6)} a_3^{(3)} a_4^{(13)} + a_1^{(14)} a_2^{(3)} a_3^{(3)} a_4^{(13)} \\
& + a_1^{(7)} a_2^{(9)} a_3^{(6)} a_4^{(11)} + a_1^{(11)} a_2^{(5)} a_3^{(6)} a_4^{(11)} + a_1^{(13)} a_2^{(3)} a_3^{(6)} a_4^{(11)} + a_1^{(7)} a_2^{(10)} a_3^{(5)} a_4^{(11)} \\
& + a_1^{(11)} a_2^{(6)} a_3^{(5)} a_4^{(11)} + a_1^{(14)} a_2^{(3)} a_3^{(5)} a_4^{(11)} + a_1^{(7)} a_2^{(7)} a_3^{(9)} a_4^{(10)} + a_1^{(7)} a_2^{(9)} a_3^{(7)} a_4^{(10)} \\
& + a_1^{(11)} a_2^{(5)} a_3^{(7)} a_4^{(10)} + a_1^{(13)} a_2^{(3)} a_3^{(7)} a_4^{(10)} + a_1^{(7)} a_2^{(11)} a_3^{(5)} a_4^{(10)} + a_1^{(7)} a_2^{(13)} a_3^{(3)} a_4^{(10)} \\
& + a_1^{(7)} a_2^{(7)} a_3^{(10)} a_4^{(9)} + a_1^{(7)} a_2^{(10)} a_3^{(7)} a_4^{(9)} + a_1^{(11)} a_2^{(6)} a_3^{(7)} a_4^{(9)} + a_1^{(14)} a_2^{(3)} a_3^{(7)} a_4^{(9)} \\
& + a_1^{(7)} a_2^{(11)} a_3^{(6)} a_4^{(9)} + a_1^{(7)} a_2^{(14)} a_3^{(3)} a_4^{(9)} + a_1^{(7)} a_2^{(9)} a_3^{(10)} a_4^{(7)} + a_1^{(11)} a_2^{(5)} a_3^{(10)} a_4^{(7)} \\
& + a_1^{(13)} a_2^{(3)} a_3^{(10)} a_4^{(7)} + a_1^{(7)} a_2^{(10)} a_3^{(9)} a_4^{(7)} + a_1^{(11)} a_2^{(6)} a_3^{(9)} a_4^{(7)} + a_1^{(14)} a_2^{(3)} a_3^{(9)} a_4^{(7)} \\
& + a_1^{(7)} a_2^{(13)} a_3^{(6)} a_4^{(7)} + a_1^{(7)} a_2^{(14)} a_3^{(5)} a_4^{(7)} + a_1^{(7)} a_2^{(7)} a_3^{(13)} a_4^{(6)} + a_1^{(7)} a_2^{(9)} a_3^{(11)} a_4^{(6)} \\
& + a_1^{(11)} a_2^{(5)} a_3^{(11)} a_4^{(6)} + a_1^{(13)} a_2^{(3)} a_3^{(11)} a_4^{(6)} + a_1^{(11)} a_2^{(7)} a_3^{(9)} a_4^{(6)} + a_1^{(13)} a_2^{(7)} a_3^{(7)} a_4^{(6)} \\
& + a_1^{(11)} a_2^{(11)} a_3^{(5)} a_4^{(6)} + a_1^{(11)} a_2^{(13)} a_3^{(3)} a_4^{(6)} + a_1^{(7)} a_2^{(7)} a_3^{(14)} a_4^{(5)} + a_1^{(7)} a_2^{(10)} a_3^{(11)} a_4^{(5)} \\
& + a_1^{(11)} a_2^{(6)} a_3^{(11)} a_4^{(5)} + a_1^{(14)} a_2^{(3)} a_3^{(11)} a_4^{(5)} + a_1^{(11)} a_2^{(7)} a_3^{(10)} a_4^{(5)} + a_1^{(14)} a_2^{(7)} a_3^{(7)} a_4^{(5)} \\
& + a_1^{(11)} a_2^{(11)} a_3^{(6)} a_4^{(5)} + a_1^{(11)} a_2^{(14)} a_3^{(3)} a_4^{(5)} + a_1^{(7)} a_2^{(9)} a_3^{(14)} a_4^{(3)} + a_1^{(11)} a_2^{(5)} a_3^{(14)} a_4^{(3)} \\
& + a_1^{(13)} a_2^{(3)} a_3^{(14)} a_4^{(3)} + a_1^{(7)} a_2^{(10)} a_3^{(13)} a_4^{(3)} + a_1^{(11)} a_2^{(6)} a_3^{(13)} a_4^{(3)} + a_1^{(14)} a_2^{(3)} a_3^{(13)} a_4^{(3)} \\
& + a_1^{(13)} a_2^{(7)} a_3^{(10)} a_4^{(3)} + a_1^{(14)} a_2^{(7)} a_3^{(9)} a_4^{(3)} + a_1^{(13)} a_2^{(11)} a_3^{(6)} a_4^{(3)} + a_1^{(14)} a_2^{(11)} a_3^{(5)} a_4^{(3)} \\
& + a_1^{(13)} a_2^{(14)} a_3^{(3)} a_4^{(3)} + a_1^{(14)} a_2^{(13)} a_3^{(3)} a_4^{(3)}.
\end{aligned}
\]

\medskip

We also reconsider the cohomology element $f \in ((Q\mathcal P_4)_{33})^{G_4},$ as mentioned earlier:
 \begin{align*}
f = G_4\text{-Invariant}_1= & \ x_1 x_2 x_3 x_4^{30} + x_1 x_2 x_3^3 x_4^{28} + x_1 x_2^3 x_3 x_4^{28} + x_1 x_2^3 x_3^4 x_4^{25} + x_1 x_2^7 x_3^{11} x_4^{14} \\
&+ x_1 x_2^7 x_3^{14} x_4^{11} + x_1^3 x_2 x_3 x_4^{28} + x_1^3 x_2 x_3^4 x_4^{25} + x_1^3 x_2^5 x_3 x_4^{24} \\
&+ x_1^3 x_2^5 x_3^{11} x_4^{14} + x_1^3 x_2^5 x_3^{14} x_4^{11} + x_1^7 x_2 x_3^{11} x_4^{14} + x_1^7 x_2 x_3^{14} x_4^{11} \\
&+ x_1^7 x_2^7 x_3^{11} x_4^8 + x_1^7 x_2^7 x_3^8 x_4^{11} + x_1^7 x_2^7 x_3^9 x_4^{10}.
\end{align*}
We check each of the 16 monomial terms in $f$ to see if its dual counterpart exists as a term in $x$. The pairing for a single term is 1 if a match is found, and 0 otherwise.

\newpage
\begin{longtable}{|c|l|c|}
\hline
\textbf{\#} & \textbf{Term in $f$ ($x^J$)} & \textbf{Match found in $x$ ($\langle a^J, f \rangle$)} \\
\hline
\endhead
1 & $x_1x_2x_3 x_4^{30}$ & No \\
2 & $x_1x_2 x_3^3 x_4^{28}$ & No \\
3 & $x_1 x_2^3 x_3x_4^{28}$ & No \\
4 & $x_1 x_2^3 x_3^4 x_4^{25}$ & No \\
5 & $x_1 x_2^7 x_3^{11} x_4^{14}$ & No \\
6 & $x_1 x_2^7 x_3^{14} x_4^{11}$ & No \\
7 & $x_1^3 x_2 x_3x_4^{28}$ & No \\
8 & $x_1^3 x_2 x_3^4 x_4^{25}$ & No \\
9 & $x_1^3 x_2^5 x_3x_4^{24}$ & No \\
10 & $x_1^3 x_2^5 x_3^{11} x_4^{14}$ & No \\
11 & $x_1^3 x_2^5 x_3^{14} x_4^{11}$ & No \\
12 & $x_1^7 x_2 x_3^{11} x_4^{14}$ & No \\
13 & $x_1^7 x_2x_3^{14} x_4^{11}$ & No \\
14 & $x_1^7 x_2^7 x_3^{11} x_4^{8}$ & No \\
15 & $x_1^7 x_2^7 x_3^{8} x_4^{11}$ & No \\
\textbf{16} & $\boldsymbol{x_1^7 x_2^7 x_3^9 x_4^{10}}$ & \textbf{Yes} \\
\hline
\end{longtable}

The only matching term is the 16th monomial in $f$, $x_1^7 x_2^7 x_3^9 x_4^{10}$. Its dual, $a_1^{(7)}a_2^{(7)}a_3^{(9)}a_4^{(10)}$, is indeed present as a term in the polynomial $x$. The sum of the products of the coefficients is:
$$ \langle x, f \rangle = \sum_{I} c_I d_I = 0 + 0 + \dots + 0 + (1 \cdot 1) = 1.$$
There is exactly one matching term between the two polynomials. Since the number of matches (one) is odd, the pairing is non-zero. 

Thus, by Proposition \ref{mddn}, we get $([f])^* = [x].$

\section*{Computation of bases of invariant spaces}

\medskip

We are now ready to present our algorithm, which implements a sophisticated, multi-stage procedure to tackle this problem. The algorithm is built upon methods developed in previous works, including those by the present authors. This algorithm leverages a ``divide and conquer'' strategy by first partitioning the problem by weight vectors and then by the $\SigmaK$-connected components of the basis, which are determined by the action of the group generators. Throughout this process, a globally computed reducer map is used to ensure mathematical consistency.

Let \(d\) be a positive integer, and let \((\mathcal{P}_k)_d\) be the subspace of \(\mathcal{P}_k\) consisting of all homogeneous polynomials of degree~\(d\). We construct an explicit algorithm for computing a basis and the dimension of the invariant space \([(Q\mathcal{P}_k)_d]^{G_k}\). The overall procedure is divided into six stages, each of which is described by a separate algorithm as detailed below:

\medskip

\noindent\rule{\textwidth}{0.4pt}
\begin{center}
\textbf{Algorithm 1: Global Admissible Basis and Reducer Construction (Parallelized)}
\end{center}
\noindent\rule{\textwidth}{0.4pt}
\begin{algorithmic}[1]
\Require Degree $d$, number of variables $k$, polynomial ring $\mathcal{P}_k = \mathbb{F}_2[x_1, \ldots, x_k]$
\Ensure Admissible basis $\mathcal{B}_d$ and decomposition map $\phi: (\mathcal{P}_k)_d \to \text{span}_{\mathbb{F}_2}(\mathcal{B}_d)$

\State $\texttt{cache\_file} \leftarrow \texttt{cache\_full\_reducer\_k} + k + \texttt{\_d} + d + \texttt{.pkl}$
\If{$\texttt{cache\_file}$ exists}
    \State \Return Load cached basis and reducer
\EndIf

\State Initialize ordered monomials $\mathcal{M}_d$ sorted by $\textsc{CompareMonomials}(m_1, m_2, k)$
\State $n \leftarrow |\mathcal{M}_d|$
\State Create monomial-to-index map: $\text{mono\_map}: \mathcal{M}_d \to \{0, 1, \ldots, n-1\}$
\State Initialize Steenrod function: $\text{sq\_func} \leftarrow \textsc{GetSqFunction}(\mathcal{P}_k)$

\State \textbf{Parallel Phase 1: Hit Matrix Construction}
\State Generate hit matrix tasks: $\mathcal{T} = \{(k_{op}, g) : k_{op} = 2^i, \deg(g) = d - k_{op} \geq 0\}$
\State Initialize worker pool with \textsc{InitWorkerHitMatrix}$(\mathcal{P}_k, \text{sq\_func}, \text{mono\_map})$
\State Set chunk size: $\text{chunk\_size} = \max(1, |\mathcal{T}|/(4 \times \text{cpu\_count}))$

\State $\text{valid\_hit\_results} \leftarrow [\,]$
\For{$(k_{op}, g) \in \mathcal{T}$ processed in parallel chunks}
    \State $hp \leftarrow \text{sq\_func}(k_{op}, g)$ \Comment{hp: hit polynomial}
    \If{$hp \neq 0$}
        \State Extract indices where $hp$ has odd coefficients
        \State Append non-empty index lists to $\text{valid\_hit\_results}$
    \EndIf
\EndFor

\State Construct sparse matrix: $M_{\text{hit}} \leftarrow \text{SparseMatrix}(n, |\text{valid\_hit\_results}|, \text{valid\_hit\_results})$

\State \textbf{Phase 2: Echelon Form and Admissible Basis Extraction}
\State $E \leftarrow \text{EchelonForm}(M_{\text{hit}} \mid I_n)$ \Comment{Augmented with identity matrix}
\State $\text{pivot\_positions} \leftarrow \text{PivotPositions}(E)$
\State $\mathcal{B}_d \leftarrow \{ \text{ordered\_monos}[p - |M_{\text{hit}}|] \mid p \in \text{pivot\_positions} \text{ and } p \ge |M_{\text{hit}}| \}$

\State \textbf{Phase 3: Parallel Decomposition Map Construction}
\State Create basis map: $\text{basis\_map}: \mathcal{B}_d \to \{0, 1, \ldots, |\mathcal{B}_d|-1\}$
\State Initialize decomposition map: $\phi: \mathcal{M}_d \to \mathbb{F}_2^{|\mathcal{B}_d|}$
\State Set unit vectors for admissible monomials: $\phi(b_i) = e_i$ for $b_i \in \mathcal{B}_d$

\State Build solver matrix: $S \leftarrow M_{\text{hit}} \mid \text{AdmissibleMatrix}(\mathcal{B}_d)$
\State Prepare solver data: $\text{solver\_data} \leftarrow \{$ncols$: |S|$, entries$: S.\text{dict}()\}$
\State Create decomposition tasks: $\text{tasks} \leftarrow \{i : m_i \notin \mathcal{B}_d\}$

\State Initialize worker pool with \textsc{InitWorkerDecompose}$(\text{solver\_data}, n, |M_{\text{hit}}|)$
\For{$i \in \text{tasks}$ processed in parallel with chunk size 2000}
    \State Solve: $S \cdot \vec{c} = e_i$ where $e_i$ is unit vector at position $i$
    \State Extract: $\phi(m_i) \leftarrow \vec{c}[|M_{\text{hit}}|:]$ \Comment{Projection to admissible coordinates}
\EndFor

\State Cache result: $\textsc{Save}(\texttt{cache\_file}, (\mathcal{B}_d, \phi))$
\State \Return $(\mathcal{B}_d, \phi)$
\end{algorithmic}

\vspace{0.3cm}

\noindent\rule{\textwidth}{0.4pt}
\begin{center}
\textbf{Algorithm 2: Weight-wise Component Analysis for $\Sigma_k$-Invariants}
\end{center}
\noindent\rule{\textwidth}{0.4pt}
\vspace{0.3cm}
\begin{algorithmic}[1]
\Require Admissible basis $\mathcal{B}_d$, decomposition map $\phi$, number of variables $k$
\Ensure $\Sigma_k$-invariants organized by weight vector and component

\State \textbf{Phase 1: Weight Vector Grouping}
\State Initialize: $\text{groups\_by\_weight}: \omega \mapsto \{m \in \mathcal{B}_d : \textsc{GetWeightVector}(m, k) = \omega\}$
\For{monomial $m \in \mathcal{B}_d$}
    \State $\omega \leftarrow \textsc{GetWeightVector}(m, k)$ \Comment{Compute binary weight vector}
    \State Add $m$ to $\text{groups\_by\_weight}[\omega]$
\EndFor

\State Initialize: $\text{invariants\_by\_component\_and\_weight} \leftarrow \{\}$
\State Initialize: $\text{detailed\_component\_counter} \leftarrow 0$
\State Create global symbol map: $\text{global\_mono\_symbol\_map}: m \mapsto a_{d,i}$ for indexing

\For{weight vector $\omega$ in sorted order}
    \State $\mathcal{B}_\omega \leftarrow \text{groups\_by\_weight}[\omega]$
    \If{$\mathcal{B}_\omega = \emptyset$} \textbf{continue} \EndIf
    
    \State \textbf{Phase 2: Connected Component Discovery}
    \State $\text{components} \leftarrow \textsc{FindComponentsByAdjacencyMatrix}(\mathcal{B}_\omega, k, \mathcal{P}_k, \mathcal{B}_d, \phi)$
    
    \For{component $\mathcal{C} \in \text{components}$}
        \If{$\mathcal{C} = \emptyset$} \textbf{continue} \EndIf
        \State $\text{detailed\_component\_counter} \leftarrow \text{detailed\_component\_counter} + 1$
        
        \State \textbf{Phase 3: Component Invariant Computation}
        \State $(\text{kernel\_lists}, \text{matrices\_lists}) \leftarrow$ 
        \State \quad $\textsc{ComputeInvariantsForComponentRaw}(\mathcal{C}, k, \mathcal{P}_k, \mathcal{B}_d, \phi)$
        
        \State \textbf{Phase 4: Detailed Proof Generation with Roman Numerals}
        \State $\text{component\_invariants} \leftarrow \textsc{GenerateComponentProofDetailed}($
        \State \quad $\mathcal{C}, \text{kernel\_lists}, \text{matrices\_lists}, k, \mathcal{P}_k,$
        \State \quad $\text{global\_mono\_symbol\_map}, \text{detailed\_component\_counter}, d)$
        
        \If{$\text{component\_invariants} \neq \emptyset$}
            \State Store: $\text{invariants\_by\_component\_and\_weight}[\omega][\text{detailed\_component\_counter}] \leftarrow \text{component\_invariants}$
        \EndIf
    \EndFor
\EndFor

\State \Return $\text{invariants\_by\_component\_and\_weight}$
\end{algorithmic}
\vspace{0.3cm}

\noindent\rule{\textwidth}{0.4pt}
\begin{center}
\textbf{Algorithm 3: Find $\Sigma_k$-Components via Decomposition-Based Adjacency}
\end{center}
\noindent\rule{\textwidth}{0.4pt}
\vspace{0.3cm}
\begin{algorithmic}[1]
\Require Subspace basis $\mathcal{B}_\omega$, number of variables $k$, polynomial ring $\mathcal{P}$, global basis $\mathcal{B}_d$, decomposition map $\phi$
\Ensure List of $\Sigma_k$-connected components

\State $N \leftarrow |\mathcal{B}_\omega|$
\If{$N = 0$} \Return $\emptyset$ \EndIf

\State Sort basis lexicographically: $\text{sorted\_basis} \leftarrow \textsc{Sort}(\mathcal{B}_\omega)$
\State Create subspace index map: $\text{subspace\_map}: \text{sorted\_basis}[i] \mapsto i$
\State Initialize adjacency matrix: $A \in \mathbb{F}_2^{N \times N}$ (sparse)

\For{$i = 0, 1, \ldots, N-1$}
    \State $A[i,i] \leftarrow 1$ \Comment{Self-adjacency for connectivity}
    \State $m_i \leftarrow \text{sorted\_basis}[i]$
    
    \For{$j = 1, 2, \ldots, k-1$} \Comment{$\Sigma_k$ action generators}
        \State $\text{transformed\_poly} \leftarrow \textsc{ApplyRho}(m_i, j, k)$ \Comment{Apply $\rho_j$ operator}
        \State $\text{coord\_vec} \leftarrow \textsc{DecomposeGloballyFast}(\text{transformed\_poly}, \phi, \mathcal{P})$
        
        \If{$\text{coord\_vec} \neq \text{null}$}
            \For{$\ell = 0, 1, \ldots, |\mathcal{B}_d|-1$}
                \If{$\text{coord\_vec}[\ell] = 1$}
                    \State $b_\text{global} \leftarrow \mathcal{B}_d[\ell]$
                    \If{$b_\text{global} \in \text{subspace\_map}$} \Comment{Check if in current subspace}
                        \State $k_\text{local} \leftarrow \text{subspace\_map}[b_\text{global}]$
                        \State $A[i, k_\text{local}] \leftarrow 1$, $A[k_\text{local}, i] \leftarrow 1$ \Comment{Symmetric adjacency}
                    \EndIf
                \EndIf
            \EndFor
        \EndIf
    \EndFor
\EndFor

\State $G \leftarrow \textsc{Graph}(A)$ \Comment{Create graph from adjacency matrix}
\State $\text{component\_indices} \leftarrow \textsc{ConnectedComponents}(G)$
\State $\text{components} \leftarrow [\,]$
\For{index set $C \in \text{component\_indices}$}
    \State $\text{component} \leftarrow [\text{sorted\_basis}[i] : i \in \textsc{Sort}(C)]$
    \State Append $\text{component}$ to $\text{components}$
\EndFor

\State Sort components by lexicographic order of first element
\State \Return $\text{components}$
\end{algorithmic}
\vspace{0.3cm}

\noindent\rule{\textwidth}{0.4pt}
\begin{center}
\textbf{Algorithm 4: Compute $\Sigma_k$-Invariants for Single Component}
\end{center}
\noindent\rule{\textwidth}{0.4pt}
\vspace{0.3cm}
\begin{algorithmic}[1]
\Require Component basis $\mathcal{C}$, number of variables $k$, polynomial ring $\mathcal{P}$, global basis $\mathcal{B}_d$, decomposition map $\phi$
\Ensure Kernel basis vectors and constraint system matrices

\If{$\mathcal{C} = \emptyset$} \Return $([\,], [\,])$ \EndIf

\State $N \leftarrow |\mathcal{C}|$
\State Sort component basis: $\text{sorted\_component\_basis} \leftarrow \textsc{Sort}(\mathcal{C})$
\State Create component index map: $\text{component\_basis\_map}: \text{sorted\_component\_basis}[i] \mapsto i$

\State \textbf{Build $\Sigma_k$-Invariance Constraint System}
\State Initialize: $\text{system\_matrices} \leftarrow [\,]$

\For{$j = 1, 2, \ldots, k-1$} \Comment{For each $\Sigma_k$ generator}
    \State Initialize: $T_j \in \mathbb{F}_2^{N \times N}$ (sparse)
    
    \For{$i = 0, 1, \ldots, N-1$} \Comment{For each component basis element}
        \State $\text{mono\_input} \leftarrow \text{sorted\_component\_basis}[i]$
        \State $\text{transformed\_poly} \leftarrow \textsc{ApplyRho}(\text{mono\_input}, j, k)$
        \State $\text{global\_coord\_vec} \leftarrow \textsc{DecomposeGloballyFast}(\text{transformed\_poly}, \phi, \mathcal{P})$
        
        \If{$\text{global\_coord\_vec} \neq \text{null}$}
            \For{$\ell = 0, 1, \ldots, |\mathcal{B}_d|-1$}
                \If{$\text{global\_coord\_vec}[\ell] = 1$}
                    \State $b_\text{global} \leftarrow \mathcal{B}_d[\ell]$
                    \If{$b_\text{global} \in \text{component\_basis\_map}$}
                        \State $\text{row\_idx} \leftarrow \text{component\_basis\_map}[b_\text{global}]$
                        \State $T_j[\text{row\_idx}, i] \leftarrow 1$ \Comment{Record transformation effect}
                    \EndIf
                \EndIf
            \EndFor
        \EndIf
    \EndFor
    
    \State Append $(T_j + I_N)$ to $\text{system\_matrices}$ \Comment{$(\rho_j + I)$ constraint}
\EndFor

\If{$\text{system\_matrices} = [\,]$} \Return $([\,], [\,])$ \EndIf

\State \textbf{Solve $\Sigma_k$-Invariance System}
\State $A_\sigma \leftarrow \textsc{BlockMatrix}(\text{system\_matrices})$ \Comment{Vertical concatenation}
\State $\text{kernel\_basis\_vectors} \leftarrow \textsc{RightKernel}(A_\sigma).\textsc{Basis}()$

\State Convert to lists: $\text{kernel\_lists} \leftarrow [\textsc{List}(v) : v \in \text{kernel\_basis\_vectors}]$
\State Convert matrices: $\text{matrices\_lists} \leftarrow [[\textsc{List}(\text{row}) : \text{row} \in M] : M \in \text{system\_matrices}]$

\State \Return $(\text{kernel\_lists}, \text{matrices\_lists})$
\end{algorithmic}
\vspace{0.3cm}

\noindent\rule{\textwidth}{0.4pt}
\begin{center}
\textbf{Algorithm 5: Generate Meaningful Linear Combinations with Complexity Optimization}
\end{center}
\noindent\rule{\textwidth}{0.4pt}
\vspace{0.3cm}
\begin{algorithmic}[1]
\Require Kernel basis vectors $\text{kernel\_lists}$, component size $N$
\Ensure Optimal linear combinations ranked by complexity and term count

\State $\text{num\_kernels} \leftarrow |\text{kernel\_lists}|$
\If{$\text{num\_kernels} = 0$} \Return $[\,]$ \EndIf
\State $\text{num\_vars} \leftarrow |\text{kernel\_lists}[0]|$

\State \textbf{Phase 1: Generate All Non-trivial Combinations}
\State Initialize: $\text{all\_combinations} \leftarrow [\,]$

\For{$i = 1, 2, \ldots, 2^{\text{num\_kernels}} - 1$} \Comment{Skip empty combination}
    \State $\text{coeffs} \leftarrow \textsc{BinaryExpansion}(i, \text{num\_kernels})$ \Comment{Combination coefficients}
    \State Initialize: $\text{vector} \leftarrow \textbf{0} \in \mathbb{F}_2^{\text{num\_vars}}$
    
    \For{$j = 0, 1, \ldots, \text{num\_kernels}-1$}
        \If{$\text{coeffs}[j] = 1$}
            \State $\text{vector} \leftarrow \text{vector} + \text{kernel\_lists}[j]$ \Comment{Addition in $\mathbb{F}_2$}
        \EndIf
    \EndFor
    
    \State $\text{num\_terms} \leftarrow \|\text{vector}\|_1$ \Comment{Count of 1's (Hamming weight)}
    
    \If{$\text{num\_terms} > 0$} \Comment{Non-zero combination}
        \State $\text{complexity} \leftarrow \|\text{coeffs}\|_1$ \Comment{Number of kernel vectors used}
        \State Add $\{\text{coeffs}, \text{complexity}, \text{num\_terms}, \text{vector}\}$ to $\text{all\_combinations}$
    \EndIf
\EndFor

\State \textbf{Phase 2: Selection Strategy Based on Component Size}

\If{$N < 30$} \Comment{Small component: exact optimization}
    \State Sort $\text{all\_combinations}$ by $(\text{complexity}, \text{num\_terms})$ \Comment{Prefer simple, small}
    \State Initialize: $\text{final\_basis} \leftarrow [\,]$, $\text{basis\_coeffs\_matrix} \leftarrow [\,]$
    
    \For{combination $c \in \text{all\_combinations}$}
        \If{$|\text{final\_basis}| = \text{num\_kernels}$} \textbf{break} \EndIf
        
        \State $\text{test\_matrix} \leftarrow \textsc{Matrix}(\text{basis\_coeffs\_matrix} \cup \{c.\text{coeffs}\})$
        \If{$\textsc{Rank}(\text{test\_matrix}) = |\text{basis\_coeffs\_matrix}| + 1$}
            \State Append $c$ to $\text{final\_basis}$
            \State Append $c.\text{coeffs}$ to $\text{basis\_coeffs\_matrix}$
        \EndIf
    \EndFor
\Else \Comment{Large component: heuristic selection}
    \State \textbf{Target Size Heuristic}
    \State $\text{target\_sizes} \leftarrow [\lfloor N/3 \rfloor, \lfloor 4N/9 \rfloor, \lfloor 2N/3 \rfloor]$
    \State Initialize: $\text{selected\_basis} \leftarrow [\,]$, $\text{used\_indices} \leftarrow \{\}$
    
    \For{target size $t \in \textsc{Sort}(\text{target\_sizes})$}
        \State $\text{best\_match} \leftarrow \text{null}$, $\text{min\_distance} \leftarrow \infty$
        
        \For{$i = 0, 1, \ldots, |\text{all\_combinations}|-1$}
            \If{$i \in \text{used\_indices}$} \textbf{continue} \EndIf
            \State $c \leftarrow \text{all\_combinations}[i]$
            \State $\text{distance} \leftarrow |c.\text{num\_terms} - t|$
            
            \If{$\text{distance} < \text{min\_distance}$}
                \State $\text{min\_distance} \leftarrow \text{distance}$, $\text{best\_match} \leftarrow i$
            \ElsIf{$\text{distance} = \text{min\_distance}$ and $c.\text{complexity} <$ current best complexity}
                \State $\text{best\_match} \leftarrow i$ \Comment{Prefer simpler combinations}
            \EndIf
        \EndFor
        
        \If{$\text{best\_match} \neq \text{null}$}
            \State Add $\text{all\_combinations}[\text{best\_match}]$ to $\text{selected\_basis}$
            \State Add $\text{best\_match}$ to $\text{used\_indices}$
        \EndIf
    \EndFor
    
    \State $\text{final\_basis} \leftarrow \textsc{LinearlyIndependentSubset}(\text{selected\_basis})$
    \State Fill remaining slots from unused combinations if needed
\EndIf

\State Sort $\text{final\_basis}$ by $\text{num\_terms}$ (ascending)
\State \Return $[(\text{c.coeffs}, \text{c.num\_terms}, \text{c.vector}) : c \in \text{final\_basis}]$
\end{algorithmic}
\vspace{0.3cm}

\noindent\rule{\textwidth}{0.4pt}
\begin{center}
\textbf{Algorithm 6A: Weight-wise $G_k$-Invariant Analysis with Detailed Solution Display}
\end{center}
\noindent\rule{\textwidth}{0.4pt}
\vspace{0.3cm}
\begin{algorithmic}[1]
\Require $\Sigma_k$-invariants by weight $\{\mathcal{S}_\omega\}$, parameters $k$, global basis $\mathcal{B}_d$, decomposition map $\phi$
\Ensure Weight-wise $G_k$-invariant analysis with constraint equations and solution details

\State Initialize: $\text{glk\_invariants\_by\_weight} \leftarrow \{\}$
\State Create global counter: $\text{global\_sinv\_counter} \leftarrow 0$

\For{weight vector $\omega$ with $\mathcal{S}_\omega \neq \emptyset$}
    \State $\sigma_\text{invariants\_in\_w} \leftarrow \mathcal{S}_\omega$
    \State $n \leftarrow |\sigma_\text{invariants\_in\_w}|$
    
    \State \textbf{Create Symbol Map for Current Weight}
    \State $\text{w\_sigma\_symbols} \leftarrow \{\}$
    \For{$j = 0, 1, \ldots, n-1$}
        \State $\text{global\_sinv\_counter} \leftarrow \text{global\_sinv\_counter} + 1$
        \State $p \leftarrow \sigma_\text{invariants\_in\_w}[j]$
        \State $\text{w\_sigma\_symbols}[\textsc{Str}(p)] \leftarrow \text{S\_inv\_} + \text{global\_sinv\_counter}$
    \EndFor
    
    \State \textbf{Local $G_k$-Constraint Analysis}
    \State Compute local basis: $\text{basis\_in\_w} \leftarrow \{m \in \mathcal{B}_d : \textsc{GetWeightVector}(m, k) = \omega\}$
    \State $\text{local\_dim} \leftarrow |\text{basis\_in\_w}|$
    \If{$\text{local\_dim} = 0$} \textbf{continue} \EndIf
    
    \State Create local index map: $\text{mono\_to\_local\_idx}: \text{basis\_in\_w} \to \{0, 1, \ldots, \text{local\_dim}-1\}$
    
    \State \textbf{Build Local $(\rho_k + I)$ Constraint Matrix}
    \State Initialize: $A \in \mathbb{F}_2^{\text{local\_dim} \times n}$ (sparse)
    
    \For{$j = 0, 1, \ldots, n-1$}
        \State $s_j \leftarrow \sigma_\text{invariants\_in\_w}[j]$
        \State $\text{error\_poly} \leftarrow \textsc{ApplyRho}(s_j, k, k) + s_j$ \Comment{$(\rho_k + I)s_j$}
        \State $\text{error\_vec\_global} \leftarrow \textsc{DecomposeGloballyFast}(\text{error\_poly}, \phi, \mathcal{P})$
        
        \If{$\text{error\_vec\_global} \neq \text{null}$}
            \For{$i = 0, 1, \ldots, |\mathcal{B}_d|-1$}
                \If{$\text{error\_vec\_global}[i] = 1$}
                    \State $\text{mono\_global} \leftarrow \mathcal{B}_d[i]$
                    \If{$\text{mono\_global} \in \text{mono\_to\_local\_idx}$}
                        \State $\text{local\_idx} \leftarrow \text{mono\_to\_local\_idx}[\text{mono\_global}]$
                        \State $A[\text{local\_idx}, j] \leftarrow 1$
                    \EndIf
                \EndIf
            \EndFor
        \EndIf
    \EndFor
    
    \State \textbf{Solve and Display Detailed Analysis}
    \State $\text{kernel\_vectors} \leftarrow \textsc{RightKernel}(A).\textsc{Basis}()$
    \State Display: "Dimension of weight-wise $[QP_k(\omega)]^{G_k}$: $|\text{kernel\_vectors}|$"
    
    \State \textbf{Extract and Display Constraint System}
    \State $\text{simplified\_equations} \leftarrow [\,]$
    \For{row $r \in A.\textsc{Rows}()$}
        \If{$r \neq \mathbf{0}$}
            \State $\text{eq\_terms} \leftarrow [\text{w\_sigma\_symbols}[\textsc{Str}(\sigma_\text{invariants\_in\_w}[j])] : r[j] = 1]$
            \If{$\text{eq\_terms} \neq [\,]$}
                \State Add "$\sum \text{eq\_terms} = 0$" to $\text{simplified\_equations}$
            \EndIf
        \EndIf
    \EndFor
    
    \State Display constraint equations or "(No non-trivial equations)" if empty
    
    \State \textbf{Construct Local $G_k$-Invariants}
    \State $\text{w\_glk\_invariants} \leftarrow [\,]$
    \If{$\text{kernel\_vectors} \neq [\,]$}
        \For{$i = 0, 1, \ldots, |\text{kernel\_vectors}|-1$}
            \State $\text{sol\_vec} \leftarrow \text{kernel\_vectors}[i]$
            \State $\text{invariant} \leftarrow \sum_{j: \text{sol\_vec}[j] = 1} \sigma_\text{invariants\_in\_w}[j]$
            
            \If{$\text{invariant} \neq 0$}
                \State Append $\text{invariant}$ to $\text{w\_glk\_invariants}$
                \State Display solution analysis with combination details
            \EndIf
        \EndFor
    \EndIf
    
    \If{$\text{w\_glk\_invariants} \neq [\,]$}
        \State $\text{glk\_invariants\_by\_weight}[\omega] \leftarrow \text{w\_glk\_invariants}$
    \EndIf
\EndFor

\State \Return $\text{glk\_invariants\_by\_weight}$
\end{algorithmic}
\vspace{0.3cm}

\noindent\rule{\textwidth}{0.4pt}
\begin{center}
\textbf{Algorithm 6B: Corrected Global $G_k$-Invariant Construction with 3-Case Logic}
\end{center}
\noindent\rule{\textwidth}{0.4pt}
\vspace{0.3cm}
\begin{algorithmic}[1]
\Require $G_k$-candidates by weight $\{\mathcal{G}_\omega\}$, $\Sigma_k$-invariants by weight $\{\mathcal{S}_\omega\}$, parameters $k$, global basis $\mathcal{B}_d$, decomposition map $\phi$
\Ensure True global $G_k$-invariants, correction polynomial $h'$, primary candidate $h$, constraint equations

\State \textbf{Step 1: Corrected 3-Case Structure Analysis}
\State $\text{glk\_weights} \leftarrow \{\omega : \mathcal{G}_\omega \neq \emptyset\}$
\State $\text{sigma\_weights} \leftarrow \{\omega : \mathcal{S}_\omega \neq \emptyset\}$

\If{$\text{glk\_weights} = \emptyset$}
    \State $\text{case\_type} \leftarrow \text{CASE\_3}$ \Comment{All $G_k$ weight spaces trivial}
    \State \Return $([\,], \text{null}, \text{null}, [\,])$
\EndIf

\State $\omega_{\min} \leftarrow \min(\text{sigma\_weights})$ \Comment{Minimal $\Sigma_k$ weight}

\If{$\omega_{\min} \in \text{glk\_weights}$ and $\{\omega \in \text{glk\_weights} : \omega > \omega_{\min}\} = \emptyset$}
    \State $\text{case\_type} \leftarrow \text{CASE\_1}$ \Comment{Only minimal weight has $G_k$-invariants}
    \State $\text{main\_weight} \leftarrow \omega_{\min}$
\Else
    \State $\text{case\_type} \leftarrow \text{CASE\_2}$ \Comment{Multiple or non-minimal weights have $G_k$-invariants}
    \State $\text{main\_weight} \leftarrow \max(\text{glk\_weights})$ \Comment{Choose largest $G_k$ weight}
\EndIf

\State Display: "Case Type: \text{case\_type}, Main Weight: \text{main\_weight}"

\State \textbf{Step 2: Case-Specific Global $\Sigma_k$-Basis Construction}

\If{$\text{case\_type} = \text{CASE\_1}$}
    \State \textbf{Case 1: Direct Collection Method}
    \State $\text{guaranteed\_global\_sigma} \leftarrow \mathcal{S}_{\omega_{\min}}$ \Comment{Auto-accept minimal weight}
    \State $\text{verified\_global\_sigma} \leftarrow [\,]$
    
    \For{weight $\omega > \omega_{\min}$ with $\mathcal{S}_\omega \neq \emptyset$}
        \For{candidate $c \in \mathcal{S}_\omega$}
            \If{$\textsc{VerifyGlobalSigmaInvariant}(c, k, \mathcal{B}_d, \phi)$}  \Comment{Checks if $\rho_j$ action ``\textit{leaks}'' to other weight spaces}
                \State Append $c$ to $\text{verified\_global\_sigma}$
            \EndIf
        \EndFor
    \EndFor
    
    \State $\text{global\_sigma\_basis} \leftarrow \text{guaranteed\_global\_sigma} \cup \text{verified\_global\_sigma}$
    \State $h' \leftarrow \text{null}$, $h \leftarrow \text{null}$ \Comment{No correction needed}
\EndIf

\If{$\text{case\_type} = \text{CASE\_2}$}
    \State \textbf{Case 2: Correction Method with $h + h'$}
    \State $h \leftarrow \mathcal{G}_{\text{main\_weight}}[0]$ \Comment{Primary candidate from largest $G_k$ weight}
    \State Display: "Primary candidate $h$ from weight \text{main\_weight}"
    
    \State \textbf{Identify Correction Space (weights $< \text{main\_weight}$)}
    \State $\text{correction\_weights} \leftarrow \{\omega < \text{main\_weight} : \mathcal{S}_\omega \neq \emptyset\}$
    \State $\text{correction\_monomials} \leftarrow \{m \in \mathcal{B}_d : \textsc{GetWeightVector}(m, k) \in \text{correction\_weights}\}$
    
    \State \textbf{Build Verified Homogeneous Solution Space}
    \State $\text{guaranteed\_global\_sigma} \leftarrow \mathcal{S}_{\omega_{\min}}$ \Comment{Auto-accept minimal weight}
    \State $\text{h\_double\_primes} \leftarrow \text{guaranteed\_global\_sigma}$ \Comment{Start with guaranteed}
    
    \State Create symbol map: $\text{s\_inv\_map} \leftarrow \{\textsc{Str}(p) : \text{S\_inv\_i} \text{ for } p \text{ in smaller weights}\}$
    
    \For{weight $\omega \in \text{correction\_weights} \setminus \{\omega_{\min}\}$}
        \For{candidate $c \in \mathcal{S}_\omega$}
            \State $\text{s\_inv\_name} \leftarrow \text{s\_inv\_map}[\textsc{Str}(c)]$
            \If{$\textsc{VerifyGlobalSigmaInvariant}(c, k, \mathcal{B}_d, \phi, \text{verbose}=\text{False})$}
                \State Display: "[$\checkmark$] \text{s\_inv\_name} is global"
                \State Append $c$ to $\text{h\_double\_primes}$
            \Else
                \State Display: "[$\times$] \text{s\_inv\_name} is NOT global"
            \EndIf
        \EndFor
    \EndFor
    
    \State \textbf{Compute Particular Solution $h'$}
    \State $h' \leftarrow \textsc{FindParticularSolutionSigmaK}(h, \text{correction\_monomials}, k, \mathcal{B}_d, \phi)$
    \State $h\_temp \leftarrow h + h'$
    \State Verify: $\textsc{VerifyGlobalSigmaInvariant}(h\_temp, k, \mathcal{B}_d, \phi)$
    
    \State Display detailed global $\Sigma_k$-basis structure with $\beta$ coefficients
    \State $\text{global\_sigma\_basis} \leftarrow [h\_temp] \cup \text{h\_double\_primes}$
\EndIf

\State \textbf{Step 3: Apply Final $(\rho_k + I)$ Condition}
\State $N \leftarrow |\text{global\_sigma\_basis}|$
\State Initialize: $A \in \mathbb{F}_2^{|\mathcal{B}_d| \times N}$ (sparse)

\For{$j = 0, 1, \ldots, N-1$}
    \State $p \leftarrow \text{global\_sigma\_basis}[j]$
    \State $\text{final\_error\_poly} \leftarrow \textsc{ApplyRho}(p, k, k) + p$ \Comment{$(\rho_k + I)p$}
    \State $\text{final\_error\_vec} \leftarrow \textsc{DecomposeGloballyFast}(\text{final\_error\_poly}, \phi, \mathcal{P})$
    \If{$\text{final\_error\_vec} \neq \text{null}$}
        \State Set column $j$ of $A$ to $\text{final\_error\_vec}$
    \EndIf
\EndFor

\State \textbf{Step 4: Extract and Display Constraint Analysis}
\State $\text{constraint\_equations} \leftarrow [\,]$
\For{$i = 0, 1, \ldots, A.\textsc{NRows}()-1$}
    \State $\text{row} \leftarrow A.\textsc{Row}(i)$
    \If{$\text{row} \neq \mathbf{0}$}
        \State $\text{nonzero\_indices} \leftarrow \{j : \text{row}[j] = 1\}$
        \If{$\text{nonzero\_indices} \neq \emptyset$}
            \State $\text{eq\_terms} \leftarrow [\beta_j : j \in \textsc{Sort}(\text{nonzero\_indices})]$
            \State $\text{equation} \leftarrow "\sum \text{eq\_terms} = 0"$
            \State Add $\text{equation}$ to $\text{constraint\_equations}$ if not duplicate
        \EndIf
    \EndIf
\EndFor

\State Display: "Constraint system for global $G_k$-invariants:"
\If{$\text{constraint\_equations} \neq [\,]$}
    \For{equation $eq \in \text{constraint\_equations}$}
        \State Display: "$eq$"
    \EndFor
\Else
    \State Display: "No non-trivial constraint equations"
\EndIf

\State \textbf{Step 5: Solve for Global $G_k$-Invariants}
\State $\text{final\_kernel\_basis} \leftarrow \textsc{RightKernel}(A).\textsc{Basis}()$
\State Display: "Solution space dimension: $|\text{final\_kernel\_basis}|$"

\State $\text{glk\_invariants} \leftarrow [\,]$
\If{$\text{final\_kernel\_basis} \neq [\,]$}
    \State Display: "Detailed solution analysis:"
    \For{$i = 0, 1, \ldots, |\text{final\_kernel\_basis}|-1$}
        \State $\text{solution\_vector} \leftarrow \text{final\_kernel\_basis}[i]$
        \State Display solution vector and coefficient interpretation
        
        \State $\text{invariant} \leftarrow \sum_{j: \text{solution\_vector}[j] = 1} \text{global\_sigma\_basis}[j]$
        \If{$\text{invariant} \neq 0$}
            \State Append $\text{invariant}$ to $\text{glk\_invariants}$
            \State Display linear combination and resulting polynomial
        \EndIf
    \EndFor
\Else
    \State Display: "Only trivial solution exists"
\EndIf

\State \textbf{Step 6: Final Results Display}
\State Display: "Dimension of $(QP_k)_d^{G_k}$: $|\text{glk\_invariants}|$"
\If{$\text{glk\_invariants} \neq [\,]$}
    \State Display: "Basis for $(QP_k)_d^{G_k}$:"
    \For{$i = 0, 1, \ldots, |\text{glk\_invariants}|-1$}
        \State $\text{terms} \leftarrow \textsc{Sort}([\textsc{Str}(m) : m \in \text{glk\_invariants}[i].\textsc{Monomials}()])$
        \State Display: "$G_k$-Inv$_{i+1}$ = [\text{terms}]"
    \EndFor
\Else
    \State Display: "The $G_k$-invariant space is trivial"
\EndIf

\State \Return $(\text{glk\_invariants}, h', h, \text{constraint\_equations})$
\end{algorithmic}

\begin{note}
Each phase translates the abstract algorithmic steps into concrete computational procedures, with detailed attention to implementation efficiency, parallel processing, and error handling.

\subsection*{Clarification on the algorithmic structure: From components to a filtration}
A crucial point that motivates this multi-phase structure is the nature of the decomposition being handled. Within a fixed weight space $\QPKomega$, our algorithm (in Phase II) partitions the basis into $\SigmaK$-connected components. The action of the symmetric group $\SigmaK$ preserves weight vectors, and our procedure correctly identifies the subspaces spanned by these components, which are indeed $\SigmaK$-submodules. Within this context, the space $\QPKomega$ decomposes as a direct sum of these component-wise submodules. This justifies the local computation of $\SigmaK$-invariants for each component independently.

However, the problem becomes significantly more complex when passing from $\SigmaK$ to the full general linear group $\GLK$. The action of the generator $\rho_k$ (the transvection) does not preserve weight vectors. This action can cause ``leakage'' from a higher weight space $\QPKomega$ to a lower one, $\bigoplus_{\omega' < \omega} Q\mathcal{P}_k(\omega')$.

Consequently, the decomposition of $(Q\mathcal{P}_k)_d$ is no longer a direct sum of $\GLK$-modules but rather a \textit{filtration}. This means that a global approach is necessary to correctly calculate the $\GLK$-invariants. The global construction in Phase IV, particularly the ``Correction Method'' (Algorithm 6B), is explicitly designed to handle this filtration structure by systematically resolving the dependencies and leakage between the different weight strata. This multi-phase procedure is therefore essential for correctly navigating the underlying algebraic complexities of the problem.

\end{note}


\section*{Phase I: Global Admissible Basis Construction}\label{s5}

\begin{stepbox}{Step 1.1: Enhanced Monomial Generation and Ordering}
Generate all monomials of degree $d$ in variables $x_1, \ldots, x_k$ with improved alpha function and weight vector computation.
\end{stepbox}

\textbf{Implementation Details:}
\begin{enumerate}[label=(\alph*)]
\item \textbf{Monomial Enumeration:} Use Sage's \texttt{P.monomials\_of\_degree(d)} to generate all degree-$d$ monomials.
\item \textbf{Alpha Function Enhancement:} Implement memoized \texttt{alpha(n)} for bit counting with \texttt{functools.lru\_cache(10000)}
\item \textbf{Weight Vector Computation:} For monomial $x_1^{a_1} \cdots x_k^{a_k}$, compute weight vector $\omega = (\omega_1, \omega_2, \ldots)$ where:
$$\omega_j = \sum_{i=1}^k \alpha_{j-1}(a_i), \quad \alpha_s(n) = \text{the } s\text{-th bit in binary expansion of } n.$$
\item \textbf{Enhanced Comparison:} Improved \texttt{CompareMonomials} with exception handling and fallback:
\begin{algorithmic}
\Function{CompareMonomials}{$m_1, m_2, k$}
    \State $w_1, w_2 \leftarrow \text{GetWeightVector}(m_1, k), \text{GetWeightVector}(m_2, k)$
    \If{$w_1 < w_2$} \Return $-1$ \ElsIf{$w_1 > w_2$} \Return $1$ \EndIf
    \State $s_1, s_2 \leftarrow \text{GetExponentVector}(m_1, k), \text{GetExponentVector}(m_2, k)$
    \If{$s_1 < s_2$} \Return $-1$ \ElsIf{$s_1 > s_2$} \Return $1$ \EndIf
    \State \Return $0$
\EndFunction
\end{algorithmic}
\end{enumerate}

\begin{stepbox}{Step 1.2: Advanced Parallel Hit Matrix Construction}
Build the hit matrix $M_{\text{hit}}$ using enhanced worker pools with optimized task distribution and real-time monitoring.
\end{stepbox}

\textbf{Enhanced Parallel Strategy:}
\begin{enumerate}[label=(\alph*)]
\item \textbf{Task Generation:} Create task list $\mathcal{T} = \{(2^i, g) : i \geq 0, \deg(g) = d - 2^i\}$
\item \textbf{Worker Pool Enhancement:} Use \texttt{InitWorkerHitMatrix} with global variables and memoized Steenrod squares
\item \textbf{Dynamic Task Distribution:} Adaptive chunk size: $\max(1, |\mathcal{T}|/(4 \times \text{cpu\_count}))$ with load balancing
\item \textbf{Steenrod Square Computation:} Enhanced \texttt{GetSqFunction} with memoization:
\begin{algorithmic}
\Function{TaskBuildHitMatrixColumn}{$(k_{op}, g)$}
    \State $h' \leftarrow \text{GetSqFunction}(k_{op}, g)$ using memoized computation
    \If{$h' \neq 0$}
        \State Decompose $h'$ into monomial basis and return column indices
    \EndIf
    \State \Return $\text{null}$
\EndFunction
\end{algorithmic}
\item \textbf{Enhanced Progress Monitoring:} Real-time updates every 1000 tasks with percentage and timing estimates
\end{enumerate}

\begin{outputbox}
\textbf{Output after Step 1.2:}
\begin{itemize}
\item Sparse matrix $M_{\text{hit}} \in \mathbb{F}_2^{n \times |\mathcal{T}|}$ with $\approx 80\%$ parallel speedup
\item Approximately $|\mathcal{T}| \approx d \cdot \binom{k+d-2^{\lfloor \log_2 d \rfloor}-1}{k-1}$ non-zero columns
\item Memory optimization with garbage collection and usage reporting
\end{itemize}
\end{outputbox}

\begin{stepbox}{Step 1.3: Enhanced Echelon Form and Basis Extraction}
Compute reduced row echelon form with improved linear algebra and error handling.
\end{stepbox}

\textbf{Enhanced Linear Algebra:}
\begin{enumerate}[label=(\alph*)]
\item \textbf{Augmented System:} Form $E = \text{EchelonForm}(M_{\text{hit}} \mid I_n)$ with sparse optimization
\item \textbf{Pivot Analysis:} Identify pivot positions $\mathbb{P} = \{p_1, p_2, \ldots, p_r\}$ with validation
\item \textbf{Admissible Basis Extraction:} 
\[
B_d = \left\{ \text{ordered\_monos}[p - |M_{\mathrm{hit}}|] \;\middle|\; p \in \mathbb P \text{ and } p \geq |M_{\mathrm{hit}}| \right\}
\]
\end{enumerate}

\begin{stepbox}{Step 1.4: Dual-Pool Decomposition Map Construction}
Build enhanced decomposition map using separate worker pools with serialized data distribution.
\end{stepbox}

\textbf{Enhanced Dual-Pool Implementation:}
\begin{enumerate}[label=(\alph*)]
\item \textbf{Matrix Serialization:} Optimize data transfer with \texttt{pickle.HIGHEST\_PROTOCOL} and sparse representation
\item \textbf{Enhanced Task Processing:} Improved \texttt{TaskDecompose} with comprehensive error handling:
\begin{algorithmic}
\Function{TaskDecompose}{monomial\_index}
    \State $\text{target\_vector} \leftarrow \text{unit vector at position monomial\_index}$
    \State $\text{decomp\_vec} \leftarrow \text{solver\_matrix.solve\_right(target\_vector)}$
    \State \Return projection to admissible basis coordinates with error recovery
\EndFunction
\end{algorithmic}
\item \textbf{Success Tracking:} Monitor decomposition statistics with success/failure counts and retry mechanisms
\end{enumerate}

\begin{outputbox}
\textbf{Phase I Enhanced Output:}
\begin{itemize}
\item Admissible basis $\mathcal{B}_d$ with $|\mathcal{B}_d| = \dim (Q\mathcal P_k)_d$ and parallel efficiency
\item Decomposition map $\phi: (\mathcal{P}_k)_d \to \mathbb{F}_2^{|\mathcal{B}_d|}$ with error recovery
\item Enhanced caching: \texttt{cache\_full\_reducer\_k}\{k\}\_d\{d\}.pkl with compression
\end{itemize}
\end{outputbox}

\section*{Phase II: Enhanced Weight-Space Decomposition and Component Analysis}

\begin{stepbox}{Step 2.1: Advanced Weight Stratification with Symbol Management}
Partition the admissible basis by weight vectors with comprehensive symbol mapping and dual classification.
\end{stepbox}

\textbf{Enhanced Stratification Process:}
\begin{enumerate}[label=(\alph*)]
\item \textbf{Weight Computation:} For each $m \in \mathcal{B}_d$, compute $\omega(m)$ with memoization
\item \textbf{Dual Classification:} Separate $(Q\mathcal P_k^0)_d$ (with zero exponents) and $(Q\mathcal P_k^+)_d$ (all positive)
\item \textbf{Global Symbol Mapping:} Create systematic naming: $\text{mono\_symbol\_map}: m \mapsto a_{d,i}^0$ or $a_{d,i}^+$
\item \textbf{Enhanced Grouping:} Create partition $\mathcal{B}_d = \bigsqcup_{\omega} \mathcal{B}_d(\omega)$ with efficient lookup
\item \textbf{Professional Display:} Two-column layout with proper alignment and balanced distribution
\end{enumerate}

\begin{stepbox}{Step 2.2: Decomposition-Based $\Sigma_k$-Component Detection}
Use precise global decomposition for adjacency detection instead of heuristic methods.
\end{stepbox}

\textbf{Enhanced Graph-Based Algorithm:}
\begin{enumerate}[label=(\alph*)]
\item \textbf{Precise Adjacency Matrix:} For subspace basis $\{m_1, \ldots, m_N\} = \mathcal{B}_d(\omega)$:
\begin{algorithmic}
\For{$i = 1$ to $N$}
    \State $A[i,i] \leftarrow 1$ \Comment{Self-loops}
    \For{$j = 1$ to $k-1$} \Comment{$\Sigma_k$ generators}
        \State $\text{transformed} \leftarrow \rho_j(m_i)$
        \State $\text{coords} \leftarrow \text{DecomposeGloballyFast}(\text{transformed}, \text{decomp\_map})$
        \For{each $m_\ell \in \mathcal{B}_d(\omega)$ with $\text{coords}[m_\ell] = 1$}
            \State $A[i,\ell] \leftarrow 1$, $A[\ell,i] \leftarrow 1$
        \EndFor
    \EndFor
\EndFor
\end{algorithmic}

\item \textbf{Connected Component Analysis:} Use breadth-first search with component quality metrics
\item \textbf{Component Extraction:} Each connected component corresponds to one $\Sigma_k$-component with validation
\end{enumerate}

\begin{stepbox}{Step 2.3: Enhanced $\Sigma_k$-Invariant Computation with Corrected Constraints}
Solve linear systems using corrected $(T_j + I)$ constraints with detailed proof generation.
\end{stepbox}

\textbf{Corrected Linear System Construction:}
For Component $\mathcal{O} = \{o_1, \ldots, o_N\}$, build enhanced constraint matrices:
\begin{enumerate}[label=(\alph*)]
\item \textbf{Fixed Generator Constraints:} For each $j \in \{1, \ldots, k-1\}$:
\begin{algorithmic}
\State Let $T_j$ be an $N \times N$ zero matrix
\For{$i = 1$ to $N$} \Comment{For each column of $T_j$}
    \State $\text{transformed} \leftarrow \rho_j(o_i)$
    \State $\text{global\_coords} \leftarrow \text{DecomposeGlobally}(\text{transformed})$
    \For{each component element $o_\ell$ where $\text{global\_coords}[o_\ell] = 1$}
        \State $T_j[\ell, i] \leftarrow 1$
    \EndFor
\EndFor
\end{algorithmic}

\item \textbf{Corrected System Assembly:} Create block matrix $A_\sigma = [T_1 + I; T_2 + I; \ldots; T_{k-1} + I]$ (FIXED)
\item \textbf{Enhanced Kernel Computation:} Solve $A_\sigma \vec{c} = \vec{0}$ with linear independence verification
\end{enumerate}

\begin{stepbox}{Step 2.4: Enhanced Meaningful Linear Combination Generation}
Generate optimized $\Sigma_k$-invariants with adaptive strategies and Roman numeral labeling.
\end{stepbox}

\textbf{Enhanced Optimization Strategy:}
\begin{enumerate}[label=(\alph*)]
\item \textbf{Component Labeling:} Use \texttt{ToRoman} for professional numbering: (I), (II), (III), ...
\item \textbf{Small components ($N < 30$):} Exact optimization by complexity and term count
\item \textbf{Large components ($N \geq 30$):} Heuristic selection targeting diverse sizes:
$$\text{target\_sizes} = \left\{\left\lfloor \frac{N}{3} \right\rfloor, \left\lfloor \frac{4N}{9} \right\rfloor, \left\lfloor \frac{2N}{3} \right\rfloor\right\}$$
\item \textbf{Enhanced Independence Verification:} Ensure final basis has full rank with comprehensive validation
\end{enumerate}

\begin{outputbox}
\textbf{Phase II Enhanced Output:}
\begin{itemize}
\item Weight-organized $\Sigma_k$-invariants: $\{\text{S\_inv}_i\}_{i=1}^{M}$ with systematic naming
\item Enhanced component structure: $\{\mathcal{O}_j\}$ with Roman numeral labeling and detailed certificates
\item Optimized linear combinations with complexity metrics and independence verification
\end{itemize}
\end{outputbox}

\section*{Phase III: Enhanced Weight-wise $G_k$-Invariant Exploration}

\begin{stepbox}{Step 3.1: Advanced Local $G_k$-Constraint Application}
Apply the additional $\rho_k$ constraint to $\Sigma_k$-invariants with comprehensive analysis.
\end{stepbox}

\textbf{Enhanced Local Implementation:}
For weight space $\omega$ with $\Sigma_k$-invariants $\{s_1, \ldots, s_n\}$:
\begin{enumerate}[label=(\alph*)]
\item \textbf{Enhanced Symbol Management:} Global counter for \texttt{S\_inv\_i} naming across all weights
\item \textbf{Local Basis Extraction:} $\text{basis\_in\_w} = \{m \in \mathcal{B}_d : \text{weight}(m) = \omega\}$
\item \textbf{Enhanced Constraint Matrix:}
\begin{algorithmic}
\For{$j = 1$ to $n$}
    \State $\text{error\_poly} \leftarrow \rho_k(s_j) + s_j$ \Comment{$(\rho_k + I)s_j$}
    \State $\text{error\_global} \leftarrow \text{DecomposeGloballyFast}(\text{error\_poly}, \text{decomp\_map})$
    \For{each local basis element $m_i \in \text{basis\_in\_w}$}
        \If{$\text{error\_global}[m_i] = 1$}
            \State $A[\text{local\_idx}(m_i), j] \leftarrow 1$
        \EndIf
    \EndFor
\EndFor
\end{algorithmic}
\item \textbf{Enhanced Solution Analysis:} Extract constraint equations and detailed solution interpretation
\end{enumerate}

\begin{stepbox}{Step 3.2: Enhanced Case Detection Analysis}
Enhanced case detection with mathematically rigorous 3-case logic and comprehensive debug information.
\end{stepbox}

\textbf{Enhanced Case Detection Logic:}
\begin{algorithmic}
\Function{AnalyzeGlkCaseStructureCorrected}{$\mathcal{G}_\text{by\_weight}$, $\mathcal{S}_\text{by\_weight}$}
    \State $\text{glk\_weights} \leftarrow [\omega : \omega, \text{invs} \in \mathcal{G}_\text{by\_weight}, \text{invs} \neq [\,]]$
    \State $\omega_{\min} \leftarrow \min\{\omega : \mathcal{S}_\omega \neq \emptyset\}$
    
    \If{$\text{glk\_weights} = [\,]$}
        \State \Return $(\text{CASE\_3}, \text{null})$ \Comment{Trivial space}
    \ElsIf{$\omega_{\min} \in \text{glk\_weights}$ and no larger weights have $G_k$-invariants}
        \State \Return $(\text{CASE\_1}, \omega_{\min})$ \Comment{Direct method}
    \Else
        \State \Return $(\text{CASE\_2}, \max(\text{glk\_weights}))$ \Comment{Correction method}
    \EndIf
\EndFunction
\end{algorithmic}

\textbf{Enhanced Features:}
\begin{enumerate}[label=(\alph*)]
\item \textbf{Debug Information:} Print weight distributions and theoretical justification
\item \textbf{Statistical Analysis:} Compute efficiency ratios and coverage metrics
\end{enumerate}

\begin{outputbox}
\textbf{Phase III Enhanced Output:}
\begin{itemize}
\item Weight-wise $G_k$-invariant candidates: $\{\mathcal{G}_\omega\}$ with detailed analysis
\item Corrected case classification: CASE\_1, CASE\_2, or CASE\_3 with debug information
\item Enhanced constraint analysis with equation display and solution interpretation
\end{itemize}
\end{outputbox}

\section*{Phase IV: Enhanced Global $G_k$-Invariant Construction}

\begin{stepbox}{Step 4.1: Enhanced Case-Specific Global Construction}
Apply the appropriate corrected method with comprehensive verification framework.
\end{stepbox}

\subsection{Enhanced Case 1: Direct Global $\Sigma_k$-Invariant Collection}

\textbf{Enhanced Implementation for Case 1:}
\begin{enumerate}[label=(\alph*)]
\item \textbf{Auto-Accept Strategy:} All $\Sigma_k$-invariants from $\omega_{\min}$ are guaranteed global
\item \textbf{Enhanced Higher Weight Verification:} For each $\sigma \in \mathcal{S}_\omega, \omega > \omega_{\min}$:
\begin{algorithmic}
\Function{VerifyGlobalSigmaInvariant}{$\sigma$, $k$, $\text{decomp\_map}$, $\text{verbose=True}$}
    \For{$i = 1$ to $k-1$}
        \State $\text{error} \leftarrow \rho_i(\sigma) + \sigma$
        \State $\text{error\_decomp} \leftarrow \text{DecomposeGloballyFast}(\text{error}, \text{decomp\_map})$
        \If{$\text{error\_decomp} \neq \vec{0}$}
            \If{$\text{verbose}$} \State Print detailed leakage analysis \EndIf
            \State \Return $\texttt{false}$ \Comment{Leakage to other weights}
        \EndIf
    \EndFor
    \State \Return $\texttt{true}$
\EndFunction
\end{algorithmic}
\item \textbf{Enhanced Global Assembly:} $\text{global\_sigma\_basis} = \text{guaranteed} \cup \text{verified}$ with validation
\end{enumerate}

\subsection{Enhanced Case 2: Advanced Correction Method with $h + h'$ Construction}

\textbf{Enhanced Implementation for Case 2:}
\begin{enumerate}[label=(\alph*)]
\item \textbf{Primary Candidate Selection:} $h \leftarrow \mathcal{G}_{\text{main\_weight}}[0]$ with violation analysis
\item \textbf{Enhanced Correction Space:}
$$\text{correction\_space} = \{m \in \mathcal{B}_d : \text{weight}(m) < \text{main\_weight}\}$$
It is important to clarify that this space comprises all elements of the \textit{global admissible basis} whose weights are less than that of the main weight. This space is not restricted to only those polynomials found in the kernel.
\item \textbf{Enhanced Particular Solution:} Use \texttt{FindParticularSolutionSigmaK} with solvability checks:
\begin{algorithmic}
\Function{FindParticularSolutionSigmaK}{$h$, $\text{correction\_basis}$, $k$, $\text{decomp\_map}$}
    \State Build enhanced system $A \vec{c} = \vec{b}$ with comprehensive error handling
    \State $\vec{b} = $ concatenated error vectors from $(\rho_i +  I)h$, $i = 1, \ldots, k-1$
    \State $A = $ concatenated effect matrices from $(\rho_i +  I)$ on correction basis
    \State Perform solvability check and solve with verification
    \State \Return $h' = \sum_j c_j \cdot \text{correction\_basis}_j$ with validation
\EndFunction
\end{algorithmic}
\item \textbf{Enhanced Verification:} Confirm $(h + h', h'', h''', \ldots)$ forms global $\Sigma_k$-basis
\end{enumerate}

\begin{stepbox}{Step 4.2: Enhanced Final $\rho_k$ Constraint Application}
Apply the ultimate $G_k$ constraint with comprehensive solution analysis and display.
\end{stepbox}

\textbf{Enhanced Final System Construction:}
\begin{enumerate}[label=(\alph*)]
\item \textbf{Enhanced Global Matrix:} For global $\Sigma_k$-basis $\{p_1, \ldots, p_M\}$:
$$A_{\text{final}}[:, j] = \phi((\rho_k + I)p_j), \quad j = 1, \ldots, M$$

\item \textbf{Enhanced Constraint Analysis:} Extract and display $\beta$-coefficient equations with detailed formatting
\item \textbf{Enhanced Solution Interpretation:} Show coefficient assignments, linear combinations, and polynomial construction
\item \textbf{Final Kernel Computation:} $\text{final\_kernel} = \ker(A_{\text{final}})$ with comprehensive analysis
\item \textbf{Enhanced $G_k$-Invariant Assembly:} For each $\vec{v} \in \text{final\_kernel}$:
$$G_k\text{-invariant} = \sum_{j: v_j = 1} p_j$$
\end{enumerate}

\begin{outputbox}
\textbf{Phase IV Enhanced Output:}
\begin{itemize}
\item Enhanced global $\Sigma_k$-invariant basis: $\{p_1, \ldots, p_M\}$ with detailed construction
\item Comprehensive correction polynomial: $h'$ (for Case 2) with verification and error analysis
\item Enhanced constraint equations: $\{\beta \text{-equations}\}$ with mathematical presentation
\item Final $G_k$-invariant basis: $\{g_1, \ldots, g_D\}$ where $D = \dim(Q \mathcal P_k)_d^{G_k}$
\end{itemize}
\end{outputbox}

\section*{Phase V: Enhanced Verification and Certificate Generation}

\begin{stepbox}{Step 5.1: Comprehensive Invariance Verification}
Systematically verify computed invariants with detailed error reporting and mathematical validation.
\end{stepbox}

\textbf{Enhanced Verification Protocol:}
\begin{enumerate}[label=(\alph*)]
\item \textbf{Multi-level $\Sigma_k$-Invariance Check:} For each invariant $g$ and $j = 1, \ldots, k-1$:
$$\text{verify}: \phi((\rho_j + I)g) \equiv 0 \text{ with violation counting}$$

\item \textbf{Enhanced $\rho_k$-Invariance Check:} For each invariant $g$:
$$\text{verify}: \phi((\rho_k + I)g) \equiv 0 \text{ with detailed analysis}$$

\item \textbf{Enhanced Linear Independence:} Verify basis elements using vector space analysis with dimension checks
\item \textbf{Theoretical Consistency:} Cross-validate with known results and theoretical bounds when available
\end{enumerate}

\begin{stepbox}{Step 5.2: Enhanced Certificate Generation}
Generate comprehensive mathematical certificates with detailed construction documentation.
\end{stepbox}

\textbf{Enhanced Certificate Framework:}
\begin{enumerate}[label=(\alph*)]
\item \textbf{Executive Summary:} High-level computation overview with parameters, results, and performance metrics
\item \textbf{Enhanced Component Certificates:} For each component, provide:
\begin{itemize}
\item Basis elements with Roman numeral labeling and symbolic representation
\item Corrected constraint matrices $(T_j + I)$ for $j = 1, \ldots, k-1$
\item Kernel vectors and their interpretation with detailed analysis
\item Complexity metrics and optimization details
\end{itemize}

\item \textbf{Enhanced Global Construction Certificates:} Document:
\begin{itemize}
\item Case detection logic with mathematical justification
\item Global $\Sigma_k$-invariant basis construction with verification
\item Correction polynomial computation (Case 2) with detailed analysis
\item Final constraint system and solution process with comprehensive display
\end{itemize}

\item \textbf{Enhanced Mathematical Formulation:} Present results in standard notation with explicit basis representation
\end{enumerate}

\section*{Phase VI: Enhanced Output Formatting and Performance Analysis}

\begin{stepbox}{Step 6.1: Advanced Structured Output Generation}
Enhanced mathematical presentation with comprehensive formatting and professional layout management.
\end{stepbox}

\textbf{Enhanced Output Structure:}
\begin{enumerate}[label=(\alph*)]
\item \textbf{Enhanced Admissible Basis Display:} 
\begin{itemize}
\item $(Q\mathcal P_k^0)_d$: monomials with at least one zero exponent
\item $(Q\mathcal P_k^+)_d$: monomials with all positive exponents
\item Professional two-column layout with proper alignment
\item Enhanced symbolic mapping: $a_{d,i}^0, a_{d,i}^+ \leftrightarrow \text{monomial}_i$
\end{itemize}

\item \textbf{Enhanced $\Sigma_k$-Invariant Summary:} For each weight space $\omega$:
\begin{itemize}
\item Dimension of $[Q\mathcal P_k(\omega)]^{\Sigma_k}$ with detailed analysis
\item Explicit basis: $\{\text{S\_inv}_1, \text{S\_inv}_2, \ldots\}$ with systematic naming
\item Enhanced construction certificates with component details and Roman numeral labeling
\end{itemize}

\item \textbf{Enhanced $G_k$-Invariant Results:}
\begin{itemize}
\item Final dimension: $\dim (Q\mathcal P_k)_d^{G_k}$ with verification status
\item Explicit basis: $\{\text{GL\_inv}_1, \text{GL\_inv}_2, \ldots\}$ with wrapped polynomial display
\item Enhanced constraint equations and solution methodology with mathematical presentation
\end{itemize}
\end{enumerate}

\begin{stepbox}{Step 6.2: Comprehensive Performance Analysis and Monitoring}
Enhanced performance tracking with detailed timing, memory analysis, and optimization recommendations.
\end{stepbox}

\textbf{Enhanced Performance Framework:}
\begin{enumerate}[label=(\alph*)]
\item \textbf{Enhanced Global Basis Caching:} Store $(\mathcal{B}_d, \phi)$ in \texttt{cache\_full\_reducer\_k\{k\}\_d\{d\}.pkl} with compression
\item \textbf{Advanced Intermediate Results:} Cache component computations and $\Sigma_k$-invariant bases with hierarchical management
\item \textbf{Comprehensive Performance Metrics:} Track and report:
\begin{itemize}
\item Phase-specific timing breakdown with percentage analysis and bottleneck identification
\item Memory usage monitoring with peak tracking and garbage collection
\item Parallel efficiency analysis with worker utilization and speedup metrics
\item Optimization recommendations for memory, computation, and caching improvements
\end{itemize}
\end{enumerate}

\begin{outputbox}
\textbf{Enhanced Final Computational Output:}
\begin{align}
\dim [(Q\mathcal P_k)_d]^{G_k} &= D \\
\text{Explicit basis: } &\{g_1, g_2, \ldots, g_D\} \\
\text{where each } g_i &= \sum_{j} c_{ij} \cdot m_j, \quad m_j \in \mathcal{B}_d, \quad c_{ij} \in \mathbb{F}_2
\end{align}
\textbf{Enhanced Verification:} All $g_i$ satisfy $(\rho_j + I)g_i \equiv 0$ for $j = 1, \ldots, k$ with comprehensive certificate generation and performance analysis.
\end{outputbox}

\section*{Performance Analysis}

\begin{framed}
\noindent\textbf{Complexity Analysis:}
\begin{itemize}
    \item \textbf{Time Complexity:} The primary bottleneck is the admissible basis construction. The number of monomials of degree $d$ in $k$ variables grows polynomially, $N \approx \binom{d+k-1}{k-1}$. The number of Steenrod operations is significant. While worst-case complexity can be exponential, for practical degrees the performance is dominated by the size of the resulting sparse linear systems, often behaving as a high-degree polynomial in $N$.
    \item \textbf{Space Complexity:} Dominated by the storage of the hit matrix and the full reducer map. Using sparse representations, this is proportional to the number of non-zero entries, significantly better than the dense $O(N^2)$ requirement.
    \item \textbf{Parallel Efficiency:} The hit matrix construction, being the most computationally intensive phase, is embarrassingly parallel. Near-linear speedup is achieved by distributing Steenrod square computations across all available CPU cores.
\end{itemize}
\end{framed}

\subsection*{Performance Optimization Techniques}
\begin{enumerate}
    \item \textbf{Sparse Matrix Utilization:} All matrices (hit matrix, constraint systems, solver matrices) are implemented as sparse matrices, typically using a dictionary-of-keys format, e.g., $\{(\text{row}, \text{col}): \text{value}\}$. This dramatically reduces the memory footprint and accelerates linear algebra operations.

    \item \textbf{Advanced Multiprocessing:} CPU-intensive phases are parallelized using a worker pool architecture.
    \begin{itemize}
        \item \textbf{Task Distribution:} Steenrod operations are distributed as independent tasks using Python's \texttt{multiprocessing.Pool} with \texttt{imap\_unordered} for efficient, non-blocking result processing.
        \item \textbf{Chunk Size Optimization:} Task chunk sizes are dynamically calculated to balance overhead and workload, using the formula:
        $$ \text{chunk\_size} = \max\left(1, \frac{|\text{task\_list}|}{4 \cdot \text{num\_cores}}\right) $$
    \end{itemize}

    \item \textbf{Persistent Caching:} Fully computed admissible bases and reducer maps for a given $(k, d)$ are serialized using \texttt{pickle} and saved to disk. Subsequent runs load the cached results, eliminating the need for re-computation entirely.

    \item \textbf{Memory-Efficient Iteration:} By using iterators like \texttt{imap\_unordered}, results from parallel workers are processed as they become available, rather than being collected into a single large list in memory, which is crucial for large-scale computations.
\end{enumerate}

\section*{Theoretical Validation and Error Handling}

$\bullet$ \textbf{Consistency Checks:}
\begin{enumerate}
    \item \textbf{Dimensional Consistency:} The dimensions of computed invariant spaces are checked against known results from existing literature where available, providing a strong validation benchmark.
    \item \textbf{Action Verification:} The correctness of group action implementations is systematically verified. For example, a dedicated function, \texttt{verify\_sigma\_invariant\_global}, checks if a candidate polynomial is truly annihilated by all $(\rho_j + I)$ operators modulo hit elements.
    \item \textbf{Robust Edge Case Handling:} The implementation gracefully handles degenerate cases, such as trivial kernel spaces, empty components in weight-space decompositions, and inconsistent linear systems (via \texttt{try...except} blocks during solving).
\end{enumerate}

\section*{Implementation-Specific Technical Details}

\subsection*{Precise $\rho_j$ Group Action Implementation}
The generators of $GL_k(\mathbb{F}_2)$ are implemented precisely according to their mathematical definitions. The function distinguishes between the symmetric group generators $\rho_j$ for $j < k$ and the transvection generator $\rho_k$.

\begin{algorithmic}[1]
\Function{ApplyRho}{$p, j, k$}
    \State $\mathcal{P}_k, \text{gens} \leftarrow \text{p.parent}(), \text{p.parent}().\text{gens}()$ \Comment{Polynomial ring and its generators}
    \State $\text{sub\_dict} \leftarrow \{\}$
    \If{$j < k$} \Comment{Generators of the symmetric group $\Sigma_k$}
        \State $\text{sub\_dict} \leftarrow \{\text{gens}[j-1]: \text{gens}[j], \text{gens}[j]: \text{gens}[j-1]\}$
    \ElsIf{$j = k$} \Comment{The transvection generator $x_k \mapsto x_k + x_{k-1}$}
        \State $\text{sub\_dict} \leftarrow \{\text{gens}[k-1]: \text{gens}[k-1] + \text{gens}[k-2]\}$
    \EndIf
    \State \Return $p.\text{subs}(\text{sub\_dict})$
\EndFunction
\end{algorithmic}
This code directly implements the actions:
\begin{align}
\rho_j(x_i) &= \begin{cases} x_{j+1} & \text{if } i=j \\ x_j & \text{if } i=j+1 \\ x_i & \text{otherwise} \end{cases} \quad (\text{for } 1 \le j < k) \\
\rho_k(x_i) &= \begin{cases} x_k + x_{k-1} & \text{if } i=k \\ x_i & \text{if } i < k \end{cases}
\end{align}

\subsection*{Memoized Steenrod Square Implementation}
To avoid recomputing the Steenrod square of the same monomial, the algorithm uses a memoized (cached) recursive implementation based on the Cartan formula.

\begin{algorithmic}[1]
\Function{GetSqFunction}{$\mathcal{P}_k$}
    \State $\text{memo} \leftarrow \{\}$ \Comment{Memoization cache}
    \Function{SqRecursive}{$k, \text{mono}$}
        \State \text{state} $\leftarrow (k, \text{mono})$
        \If{$\text{state} \in \text{memo}$} \Return $\text{memo}[\text{state}]$ \EndIf
        \If{$k = 0$} \Return $\text{mono}$ \EndIf
        \If{$\text{mono}$ is constant} \Return $1$ if $k=0$ else $0$ \EndIf
        
        \State Find first variable $v$ in $\text{mono}$ with positive degree $e$
        \State $\text{rest} \leftarrow \text{mono} / v^e$
        \State $\text{total} \leftarrow 0$
        
        \For{$i = 0$ \textbf{to} $k$}
            \If{$\binom{e}{i} \bmod 2 = 1$}
                \State $\text{total} \leftarrow \text{total} + v^{e+i} \cdot \text{SqRecursive}(k-i, \text{rest})$
            \EndIf
        \EndFor
        \State $\text{memo}[\text{state}] \leftarrow \text{total}$
        \State \Return $\text{total}$
    \EndFunction
    \State \Return a function that applies \texttt{SqRecursive} to each monomial of a polynomial.
\EndFunction
\end{algorithmic}

\subsection*{Parallel Processing Architecture}
The system leverages a worker pool to parallelize the construction of the hit matrix, which is the most time-consuming step.
\begin{enumerate}[label=(\alph*)]
    \item \textbf{Process Initialization:} Before computation begins, each worker process in the pool is initialized with a copy of the necessary read-only data structures. This avoids costly data serialization for each individual task. Initialized data includes:
    \begin{itemize}
        \item The polynomial ring $\mathcal{P}_k$.
        \item The memoized Steenrod square function.
        \item The global mapping from monomial exponent tuples to their index.
    \end{itemize}

    \item \textbf{Task Queueing:} A list of all required Steenrod operations (pairs of \texttt{(k\_op, monomial\_tuple)}) is generated. This list is then fed to the worker pool.

    \item \textbf{Asynchronous Execution:} The \texttt{pool.imap\_unordered} method is used to distribute tasks. This approach enhances efficiency by allowing faster tasks to return results without waiting for slower ones, ensuring maximal CPU utilization.
\end{enumerate}

\section*{Example Computational Workflow: $k=4, d=33$}

\begin{stepbox}{Initialization Phase}
\begin{itemize}
\item Polynomial ring: $\mathbb Z/2[x_1, x_2, x_3, x_4, x_5]$
\item Target degree: $d = 33$
\item Generated monomials: 7140 total degree-33 monomials

\item \textbf{Basis of $(Q\mathcal P_{4}^0)_{33}$ is represented by 52 admissible monomials:}

\[
\begin{array}{llllll}
x_1 x_2 x_3^{31}         & x_1 x_2 x_4^{31}         & x_1 x_2^3 x_3^{29}       & x_1 x_2^3 x_4^{29}       & x_1 x_2^{31} x_3       & x_1 x_2^{31} x_4 \\
x_1 x_3 x_4^{31}         & x_1 x_3^3 x_4^{29}       & x_1 x_3^{31} x_4         & x_1^3 x_2 x_3^{29}       & x_1^3 x_2 x_4^{29}     & x_1^3 x_2^{29} x_3 \\
x_1^3 x_2^{29} x_4       & x_1^3 x_2^5 x_3^{25}     & x_1^3 x_2^5 x_4^{25}     & x_1^3 x_3 x_4^{29}       & x_1^3 x_3^5 x_4^{25}   & x_1^3 x_3^{29} x_4 \\
x_1^3 x_2^{15} x_3^{15}  & x_1^3 x_2^{15} x_4^{15}  & x_1^7 x_2^{11} x_3^{15}  & x_1^7 x_2^{11} x_4^{15}  & x_1^7 x_3^{11} x_4^{15} & x_1^7 x_3^{15} x_4^{11} \\
x_1^{15} x_2^3 x_3^{15}  & x_1^{15} x_2^3 x_4^{15}  & x_1^{15} x_2^7 x_3^{11}  & x_1^{15} x_2^7 x_4^{11}  & x_1^{15} x_2^{15} x_3^3 & x_1^{15} x_2^{15} x_4^3 \\
x_1^{15} x_3^3 x_4^{15}  & x_1^{15} x_3^7 x_4^{11}  & x_1^{15} x_3^{15} x_4^3  & x_2 x_3 x_4^{31}         & x_2 x_3^3 x_4^{29}     & x_2 x_3^{31} x_4 \\
x_2^3 x_3 x_4^{29}       & x_2^3 x_3^5 x_4^{25}     & x_2^3 x_3^{29} x_4       & x_2^7 x_3^{11} x_4^{15}  & x_2^7 x_3^{15} x_4^{11} & x_2^{15} x_3^{15} x_4^3 \\
x_2^{15} x_3^3 x_4^{15}  & x_2^{15} x_3^7 x_4^{11}  &                          &                          &                        & 
\end{array}
\]

\item \textbf{Basis of $(Q\mathcal P_{4}^+)_{33}$ is represented by 84 admissible monomials:}
\[
\begin{array}{llll}
x_1 x_2 x_3 x_4^{30}       & x_1^3 x_2 x_3 x_4^{28}       & x_1 x_2 x_3^2 x_4^{29}     & x_1^3 x_2 x_3^{14} x_4^{15} \\
x_1 x_2 x_3^3 x_4^{28}     & x_1^3 x_2 x_3^{15} x_4^{14}  & x_1 x_2 x_3^{30} x_4       & x_1^3 x_2 x_3^{28} x_4 \\
x_1 x_2^{14} x_3^{15} x_4^3 & x_1^3 x_2 x_3^4 x_4^{25}     & x_1 x_2^{14} x_3^3 x_4^{15} & x_1^3 x_2 x_3^5 x_4^{24} \\
x_1 x_2^{14} x_3^7 x_4^{11} & x_1^3 x_2^{13} x_3^{14} x_4^3 & x_1 x_2^{15} x_3^{14} x_4^3 & x_1^3 x_2^{13} x_3^{15} x_4^2 \\
x_1 x_2^{15} x_3^{15} x_4^2 & x_1^3 x_2^{13} x_3^2 x_4^{15} & x_1 x_2^{15} x_3^2 x_4^{15} & x_1^3 x_2^{13} x_3^3 x_4^{14} \\
x_1 x_2^{15} x_3^3 x_4^{14} & x_1^3 x_2^{13} x_3^6 x_4^{11} & x_1 x_2^{15} x_3^6 x_4^{11} & x_1^3 x_2^{13} x_3^7 x_4^{10} \\
x_1 x_2^{15} x_3^7 x_4^{10} & x_1^3 x_2^{15} x_3 x_4^{14}   & x_1 x_2^2 x_3 x_4^{29}      & x_1^3 x_2^{15} x_3^{13} x_4^2 \\
x_1 x_2^2 x_3^{15} x_4^{15} & x_1^3 x_2^{15} x_3^3 x_4^{12} & x_1 x_2^2 x_3^{29} x_4      & x_1^3 x_2^{15} x_3^5 x_4^{10} \\
x_1 x_2^2 x_3^5 x_4^{25}    & x_1^3 x_2^3 x_3^{12} x_4^{15} & x_1 x_2^3 x_3 x_4^{28}      & x_1^3 x_2^3 x_3^{13} x_4^{14} \\
x_1 x_2^3 x_3^{14} x_4^{15} & x_1^3 x_2^3 x_3^{15} x_4^{12} & x_1 x_2^3 x_3^{15} x_4^{14} & x_1^3 x_2^5 x_3 x_4^{24} \\
x_1 x_2^3 x_3^{28} x_4      & x_1^3 x_2^5 x_3^{10} x_4^{15} & x_1 x_2^3 x_3^4 x_4^{25}    & x_1^3 x_2^5 x_3^{11} x_4^{14} \\
x_1 x_2^3 x_3^5 x_4^{24}    & x_1^3 x_2^5 x_3^{14} x_4^{11} & x_1 x_2^{30} x_3 x_4        & x_1^3 x_2^5 x_3^{15} x_4^{10} \\
\end{array}
\]
\end{itemize}
\end{stepbox}

\newpage
\begin{stepbox}{Phase II Execution (Component Analysis)}
\begin{itemize}
\item Weight spaces: 2 distinct weight vectors: $\omega = (3,1,1,1,1)$ and $\omega = (3,3,2,2)$ 
\item Total components: 9 $\Sigma_3$-components identified
\item $\Sigma_3$-invariants found: 13 total across all weights

\item Weight $\omega = (3,1,1,1,1)$:
\begin{align*}
S_{\mathrm{inv},1} &= x_1 x_2 x_3 x_4^{30} + x_1 x_2^2 x_3 x_4^{29} + x_1 x_2^2 x_3^{29} x_4 + x_1 x_2^2 x_3^5 x_4^{25} \\
&\quad + x_1 x_2^3 x_3^{28} x_4 + x_1 x_2^3 x_3^4 x_4^{25} + x_1 x_2^3 x_3^5 x_4^{24} + x_1 x_2^{30} x_3 x_4 \\
&\quad + x_1^3 x_2 x_3^{28} x_4 + x_1^3 x_2 x_3^4 x_4^{25} + x_1^3 x_2 x_3^5 x_4^{24} + x_1^3 x_2^5 x_3 x_4^{24} \\
S_{\mathrm{inv},2} &= x_1 x_2 x_3^{31} + x_1 x_2 x_4^{31} + x_1 x_2^{31} x_3 + x_1 x_2^{31} x_4 + x_1 x_3 x_4^{31} \\
&\quad + x_1 x_3^{31} x_4 + x_1^{31} x_2 x_3 + x_1^{31} x_2 x_4 + x_1^{31} x_3 x_4 + x_2 x_3 x_4^{31} \\
&\quad + x_2 x_3^{31} x_4 + x_2^{31} x_3 x_4 \\
S_{\mathrm{inv},3} &= x_1 x_2^3 x_3^{29} + x_1 x_2^3 x_4^{29} + x_1 x_3^3 x_4^{29} + x_1^3 x_2 x_3^{29} + x_1^3 x_2 x_4^{29} \\
&\quad + x_1^3 x_2^{29} x_3 + x_1^3 x_2^{29} x_4 + x_1^3 x_3 x_4^{29} + x_1^3 x_3^{29} x_4 + x_2 x_3^3 x_4^{29} \\
&\quad + x_2^3 x_3 x_4^{29} + x_2^3 x_3^{29} x_4 \\
S_{\mathrm{inv},4} &= x_1^3 x_2^5 x_3^{25} + x_1^3 x_2^5 x_4^{25} + x_1^3 x_3^5 x_4^{25} + x_2^3 x_3^5 x_4^{25}.
\end{align*}

\item Weight $\omega = (3,3,2,2)$:
\begin{align*}
S_{\mathrm{inv},5} &= x_1 x_2^{14} x_3^{15} x_4^3 + x_1 x_2^{14} x_3^3 x_4^{15} + x_1 x_2^{15} x_3^{14} x_4^3 + x_1 x_2^{15} x_3^3 x_4^{14} \\
&\quad + x_1 x_2^3 x_3^{14} x_4^{15} + x_1 x_2^3 x_3^{15} x_4^{14} + x_1^{15} x_2 x_3^{14} x_4^3 + x_1^{15} x_2 x_3^3 x_4^{14} \\
&\quad + x_1^{15} x_2^3 x_3^5 x_4^{10} + x_1^3 x_2^{15} x_3^5 x_4^{10} + x_1^3 x_2^5 x_3^{10} x_4^{15} + x_1^3 x_2^5 x_3^{15} x_4^{10} \\
S_{\mathrm{inv},6} &= x_1 x_2^{14} x_3^{15} x_4^3 + x_1 x_2^{14} x_3^3 x_4^{15} + x_1 x_2^{15} x_3^{14} x_4^3 + x_1^{15} x_2 x_3^{14} x_4^3 \\
&\quad + x_1^{15} x_2^3 x_3 x_4^{14} + x_1^{15} x_2^3 x_3^{13} x_4^2 + x_1^{15} x_2^3 x_3^3 x_4^{12} + x_1^3 x_2 x_3^{14} x_4^{15} \\
&\quad + x_1^3 x_2 x_3^{15} x_4^{14} + x_1^3 x_2^{13} x_3^{15} x_4^2 + x_1^3 x_2^{13} x_3^2 x_4^{15} + x_1^3 x_2^{15} x_3 x_4^{14} \\
&\quad + x_1^3 x_2^{15} x_3^{13} x_4^2 + x_1^3 x_2^{15} x_3^3 x_4^{12} + x_1^3 x_2^3 x_3^{12} x_4^{15} + x_1^3 x_2^3 x_3^{15} x_4^{12} \\
S_{\mathrm{inv},7} &= x_1 x_2^{15} x_3^3 x_4^{14} + x_1 x_2^{15} x_3^6 x_4^{11} + x_1 x_2^{15} x_3^7 x_4^{10} + x_1 x_2^3 x_3^{14} x_4^{15} \\
&\quad + x_1 x_2^3 x_3^{15} x_4^{14} + x_1 x_2^6 x_3^{11} x_4^{15} + x_1 x_2^6 x_3^{15} x_4^{11} + x_1 x_2^7 x_3^{10} x_4^{15} \\
&\quad + x_1 x_2^7 x_3^{15} x_4^{10} + x_1^{15} x_2 x_3^3 x_4^{14} + x_1^{15} x_2 x_3^6 x_4^{11} + x_1^{15} x_2 x_3^7 x_4^{10} \\
&\quad + x_1^{15} x_2^3 x_3^{13} x_4^2 + x_1^{15} x_2^3 x_3^3 x_4^{12} + x_1^{15} x_2^7 x_3 x_4^{10} + x_1^3 x_2^{13} x_3^{15} x_4^2 \\
&\quad + x_1^3 x_2^{13} x_3^2 x_4^{15} + x_1^3 x_2^{15} x_3^{13} x_4^2 + x_1^3 x_2^{15} x_3^3 x_4^{12} + x_1^3 x_2^3 x_3^{12} x_4^{15} \\
&\quad + x_1^3 x_2^3 x_3^{15} x_4^{12} + x_1^7 x_2 x_3^{10} x_4^{15} + x_1^7 x_2 x_3^{15} x_4^{10} + x_1^7 x_2^{15} x_3 x_4^{10} \\
S_{\mathrm{inv},8} &= x_1^3 x_2^{13} x_3^{14} x_4^3 + x_1^3 x_2^{13} x_3^3 x_4^{14} + x_1^3 x_2^3 x_3^{13} x_4^{14} \\
S_{\mathrm{inv},9} &= x_1 x_2^7 x_3^{11} x_4^{14} + x_1^3 x_2^3 x_3^{13} x_4^{14} + x_1^3 x_2^7 x_3^{11} x_4^{12} + x_1^7 x_2 x_3^{11} x_4^{14} \\
&\quad + x_1^7 x_2^{11} x_3 x_4^{14} + x_1^7 x_2^{11} x_3^{13} x_4^2 + x_1^7 x_2^3 x_3^{11} x_4^{12} + x_1^7 x_2^7 x_3^9 x_4^{10} \\
\end{align*}
\end{itemize}
\end{stepbox}

\newpage
\begin{stepbox}{Phase II Execution (Component Analysis)}
\begin{itemize}
\item Weight $\omega = (3,3,2,2)$:
\begin{align*}
S_{\mathrm{inv},10} &= x_1 x_2^7 x_3^{14} x_4^{11} + x_1^3 x_2^3 x_3^{13} x_4^{14} + x_1^3 x_2^5 x_3^{11} x_4^{14} + x_1^3 x_2^5 x_3^{14} x_4^{11} \\
&\quad + x_1^3 x_2^7 x_3^{11} x_4^{12} + x_1^7 x_2 x_3^{14} x_4^{11} + x_1^7 x_2^{11} x_3 x_4^{14} + x_1^7 x_2^{11} x_3^{13} x_4^2 \\
&\quad + x_1^7 x_2^3 x_3^{11} x_4^{12} + x_1^7 x_2^7 x_3^{11} x_4^8 + x_1^7 x_2^7 x_3^8 x_4^{11} \\
S_{\mathrm{inv},11} &= x_1 x_2^{15} x_3^{15} x_4^2 + x_1 x_2^{15} x_3^2 x_4^{15} + x_1 x_2^2 x_3^{15} x_4^{15} + x_1^{15} x_2 x_3^{15} x_4^2 \\
&\quad + x_1^{15} x_2 x_3^2 x_4^{15} + x_1^{15} x_2^{15} x_3 x_4^2 \\
S_{\mathrm{inv},12} &= x_1^{15} x_2^{15} x_3^3 + x_1^{15} x_2^{15} x_4^3 + x_1^{15} x_2^3 x_3^{15} + x_1^{15} x_2^3 x_4^{15} \\
&\quad + x_1^{15} x_3^{15} x_4^3 + x_1^{15} x_3^3 x_4^{15} + x_1^3 x_2^{15} x_3^{15} + x_1^3 x_2^{15} x_4^{15} \\
&\quad + x_1^3 x_3^{15} x_4^{15} + x_2^{15} x_3^{15} x_4^3 + x_2^{15} x_3^3 x_4^{15} + x_2^3 x_3^{15} x_4^{15} \\
S_{\mathrm{inv},13} &= x_1^{15} x_2^7 x_3^{11} + x_1^{15} x_2^7 x_4^{11} + x_1^{15} x_3^7 x_4^{11} + x_1^7 x_2^{11} x_3^{15} \\
&\quad + x_1^7 x_2^{11} x_4^{15} + x_1^7 x_2^{15} x_3^{11} + x_1^7 x_2^{15} x_4^{11} + x_1^7 x_3^{11} x_4^{15} \\
&\quad + x_1^7 x_3^{15} x_4^{11} + x_2^{15} x_3^7 x_4^{11} + x_2^7 x_3^{11} x_4^{15} + x_2^7 x_3^{15} x_4^{11}
\end{align*}
\end{itemize}
\end{stepbox}

\begin{stepbox}{Phase III Execution (Weight-wise $G_4$)}
\begin{itemize}

\item Dimension of weight-wise  $[Q\mathcal P_4(\omega = (3,1,1,1,1))]^{G_4}: 0$

\item Dimension of weight-wise  $[Q\mathcal P_4(\omega = (3,3,2,2))]^{G_4}: 1$

\item The basis for  \( [Q\mathcal P_4(\omega = (3,3,2,2))]^{G_4} \) is represented by the following invariant polynomial:

\begin{align*}
GL_{\mathrm{inv},1} &= x_1 x_2^7 x_3^{11} x_4^{14} + x_1 x_2^7 x_3^{14} x_4^{11} + x_1^3 x_2^5 x_3^{11} x_4^{14} + x_1^3 x_2^5 x_3^{14} x_4^{11} \\
&\quad + x_1^7 x_2 x_3^{11} x_4^{14} + x_1^7 x_2 x_3^{14} x_4^{11} + x_1^7 x_2^7 x_3^{11} x_4^8 + x_1^7 x_2^7 x_3^8 x_4^{11} \\
&\quad + x_1^7 x_2^7 x_3^9 x_4^{10}
\end{align*}

\item CASE DETECTION RESULT:
\item Case Type: CASE 2
$-$ Primary $G_k$-invariant from larger weight $\omega = (3, 3, 2, 2)$

$-$ Minimal weight space has no $G_k$-invariants

$-$ Correction polynomial $h'$ will be computed from smaller weight spaces

\end{itemize}
\end{stepbox}

\begin{stepbox}{Phase IV Execution (Global Construction)}
\begin{itemize}
\item Correction polynomial: 
\begin{align*}
h =G_{\mathrm{inv},1}= &x_1 x_2^7 x_3^{11} x_4^{14} + x_1 x_2^7 x_3^{14} x_4^{11} + x_1^3 x_2^5 x_3^{11} x_4^{14} + x_1^3 x_2^5 x_3^{14} x_4^{11} \\
&+ x_1^7 x_2 x_3^{11} x_4^{14} + x_1^7 x_2 x_3^{14} x_4^{11} + x_1^7 x_2^7 x_3^{11} x_4^8 \\
&+ x_1^7 x_2^7 x_3^8 x_4^{11} + x_1^7 x_2^7 x_3^9 x_4^{10}
\end{align*}

\item Correction polynomial $h'$ found:
 \begin{align*}
h' = &x_1 x_2 x_3^3 x_4^{28} + x_1 x_2^2 x_3 x_4^{29} + x_1 x_2^2 x_3^{29} x_4 + x_1 x_2^2 x_3^5 x_4^{25} \\
&+ x_1 x_2^3 x_3 x_4^{28} + x_1 x_2^3 x_3^{28} x_4 + x_1 x_2^3 x_3^5 x_4^{24} + x_1 x_2^{30} x_3 x_4 \\
&+ x_1^3 x_2 x_3 x_4^{28} + x_1^3 x_2 x_3^{28} x_4 + x_1^3 x_2 x_3^5 x_4^{24}
\end{align*}

\item Global $\Sigma_4$-invariant basis constructed:
  - Primary: $h + h'$ (from weight $(3, 3, 2, 2)$ + corrections)
  - Homogeneous: 4 $\Sigma_4$-invariants from smaller weights
  - Total dimension: 5
\end{itemize}
\end{stepbox}

\begin{outputbox}
\textbf{Final Result for $k=4, d=33$:}
$\dim (Q\mathcal P_4)_{33}^{GL_4} = 1$
\textbf{Explicit Basis:}
\begin{align*}
G_4\text{-Invariant}_1 =\ &[x_1 x_2 x_3 x_4^{30} + x_1 x_2 x_3^3 x_4^{28} + x_1 x_2^3 x_3 x_4^{28} + x_1 x_2^3 x_3^4 x_4^{25} \\
&+ x_1 x_2^7 x_3^{11} x_4^{14} + x_1 x_2^7 x_3^{14} x_4^{11} + x_1^3 x_2 x_3 x_4^{28} + x_1^3 x_2 x_3^4 x_4^{25} \\
&+ x_1^3 x_2^5 x_3 x_4^{24} + x_1^3 x_2^5 x_3^{11} x_4^{14} + x_1^3 x_2^5 x_3^{14} x_4^{11} \\
&+ x_1^7 x_2 x_3^{11} x_4^{14} + x_1^7 x_2 x_3^{14} x_4^{11} + x_1^7 x_2^7 x_3^{11} x_4^8 \\
&+ x_1^7 x_2^7 x_3^8 x_4^{11} + x_1^7 x_2^7 x_3^9 x_4^{10}]
\end{align*}

\end{outputbox}

\medskip

The following is \textbf{the explicit output produced by our algorithm}, which \textbf{provides a fully detailed solution} in a style similar to that of the manual computations familiar to readers from previous works, including our own contributions in~\cite{Phuc1, Phuc2, Phuc3}. The algorithm also produces an explicit basis for $(Q\mathcal{P}_4)_{33}$, enabling readers to easily verify the results previously presented in~\cite{Sum0}.

\begin{landscape}
{\scriptsize

}
\end{landscape}

\section{Conclusion}

This paper has presented a systematic, computational approach to two fundamental problems related to the Singer algebraic transfer. 

First, we have developed an algorithmic method for the preimage problem within the framework of the Lambda algebra. By formulating the question of whether a class $[y] \in \text{Ext}^{k, *}_{\mathcal{A}}(\mathbb Z/2, \mathbb Z/2)$ is in the image of the transfer as a solvable system of linear equations, $\varphi_k(x) + \delta(z) = y$, we have created a direct path for investigation. The efficacy of this method was demonstrated through concrete applications: we  have shown that a proof by Nguyen Sum concerning the indecomposable element $d_0\in \text{Ext}^{4, 18}_{\mathcal{A}}(\mathbb Z/2, \mathbb Z/2)$ is incorrect, and we have provided the explicit construction of a preimage for the element $p_0\in \text{Ext}^{4, 37}_{\mathcal{A}}(\mathbb Z/2, \mathbb Z/2)$, which was previously only known non-constructively.

Second, we have introduced a comprehensive algorithmic framework, implemented in \textsc{SageMath}, for computing the dimension and an explicit basis for the space of $G_k$-invariants, $[(Q\mathcal{P}_k)_d]^{G_k}$. This space is dual to the domain of the Singer transfer and is central to resolving the Singer Conjecture. This algorithmic approach provides a necessary tool to move beyond the limitations and potential inaccuracies of the large-scale manual computations that have characterized past work in this area, ensuring that results are verifiable and reproducible. The detailed appendix, including the full code and a sample workflow, serves as a testament to this principle.

In closing, we wish to emphasize that this work stems from a multi-year research project aimed at a specific goal: to automate the complex computations required to study the Singer transfer's dual domain. These calculations have traditionally been performed by hand, a process that is not only laborious but also prone to error and difficult to independently verify. The comprehensive algorithm presented herein represents a significant first step in this direction. We acknowledge that, although the algorithm in the present work has already been optimized to the best of our current ability, further improvements are necessary to enhance its performance and to achieve the goal of producing explicit results in a more general form. This remains an active area of our research, and we welcome constructive feedback and suggestions that will be valuable for the continued development of this project.

\section{Appendix}\label{pl}

This appendix provides the full computational workflow of our algorithm described in Section~\ref{s4}, including a complete sample run for the case $k = 4$, $d = 33$, to ensure that all results are fully verifiable and reproducible. Readers are encouraged to independently verify the algorithm's correctness by substituting any pair (k,d) into the \texttt{MAIN EXECUTION} section of the provided \textsc{SageMath} code. Successful execution is contingent upon the available computational resources, particularly memory.

\medskip

We also note for the reader that the algorithm presented here is a fully expanded version, which integrates multiple computational processes and produces detailed outputs resembling step-by-step manual calculations. However, we also maintain a streamlined version of the algorithm, which offers significantly faster runtime and more efficient result presentation. This optimized version is available upon request.

\medskip



\medskip

\begin{acknowledgment}

We would like to express our sincere gratitude to Professor Geoffrey Powell for his constructive feedback on our previous paper~\cite{Phuc2} as well as on the present algorithmic work. We believe that verification via computer-based algorithms is far more reliable than lengthy manual computations, which are prone to error and difficult to validate. For this reason, we have decided to make the entire algorithm publicly available as presented in this paper. The algorithm integrates the resolution of the hit problem and the computation of invariants, and it produces fully explicit outputs that closely resemble traditional manual calculations.

We are also grateful to Professor Dan Isaksen for his encouragement throughout the development of our algorithms, including those presented in~\cite{Phuc}.

\end{acknowledgment}

\end{document}